\documentclass[12pt, reqno]{amsart}
\usepackage{amsmath}
\usepackage{mathrsfs}
\usepackage{amsfonts}
\usepackage[centertags]{amsmath}
\usepackage{amssymb}
\usepackage{amsthm}
\usepackage{graphicx}
\usepackage{caption}
\usepackage{enumerate}
\usepackage[margin=1in]{geometry}
\usepackage{appendix}
\usepackage{color}
\usepackage[all]{xy}
\usepackage{float}
\usepackage{amsmath,amsfonts,amssymb,graphicx}
\usepackage{longtable}
\usepackage{resizegather}
\usepackage{framed}
\usepackage{tikz}
\usepackage{multirow} 
\usepackage[colorlinks=true,
            linkcolor=blue,
            citecolor=blue
            ]{hyperref}

\newtheorem{theorem}{Theorem}[section]
\newtheorem{corollary}[theorem]{Corollary}
\newtheorem{lemma}[theorem]{Lemma}

\newtheorem{definition}[theorem]{Definition}
\newtheorem{remark}[theorem]{Remark}
\newtheorem{example}[theorem]{Example}

\numberwithin{equation}{section}

\usetikzlibrary{arrows,automata,positioning}
\newcommand{\midarrow}{\tikz \draw[thin,-angle 45] (0,0) -- +(.25,0);}

\newcommand{\CC}{\mathbb{C}}
\newcommand{\ZZ}{\mathbb{Z}}
\newcommand{\fin}{{\rm fin}}
\newcommand{\res}{{\rm res}}

\begin{document}

\title[A path description for $\varepsilon$-characters of simple modules]{A path description for $\varepsilon$-characters of representations of type $A$ restricted quantum loop algebras at roots of unity}
\author{Xiao-Juan An, Jian-Rong Li, Yan-Feng Luo$^\dag$, and Wen-Ting Zhang}\thanks{$\dag$ Corresponding author}

\address{Xiao-Juan An: School of Mathematics and Statistics, Lanzhou University, Lanzhou 730000, P. R. China.}
\email{anxj0820@163.com}

\address{Jian-Rong Li: Faculty of Mathematics, University of Vienna, Oskar-Morgenstern Platz 1, 1090 Vienna, Austria}
\email{jianrong.li@univie.ac.at}

\address{Yan-Feng Luo: School of Mathematics and Statistics, Lanzhou University, Lanzhou 730000, P. R. China.}
\email{luoyf@lzu.edu.cn}

\address{Wen-Ting Zhang: School of Mathematics and Statistics, Lanzhou University, Lanzhou 730000, P. R. China.}
\email{zhangwt@lzu.edu.cn}
\date{}

\begin{abstract}
Fix $\varepsilon^{2\ell}=1$ with $\ell \geq 2$. In this paper, we show that all finite-dimensional simple modules of any restricted quantum loop algebra $U_{\varepsilon}^{\res}({L\mathfrak{sl}_{n+1}})$ in a certain category can be transformed into snake modules. We obtain an effective and concrete path description for $\varepsilon$-characters of any simple module with highest $l$-weight of degree two and any Kirillov-Reshetikhin module of $U_{\varepsilon}^{\res}({L\mathfrak{sl}_{n+1}})$. As an application of our path description, we obtain a necessary and sufficient condition for the tensor product of two fundamental representations of $U_{\varepsilon}^{\res}({L\mathfrak{sl}_{n+1}})$ to be irreducible. Additionally, we obtain a necessary condition for the tensor product of two or more fundamental representations of $U_{\varepsilon}^{\res}({L\mathfrak{sl}_{n+1}})$ to be irreducible.

\hspace{0.15cm}

\noindent
{\bf Key words}: restricted quantum loop algebras at roots of unity; finite-dimensional modules; $\varepsilon$-characters; path description

\hspace{0.15cm}

\noindent
{\bf 2020 Mathematics Subject Classification}: 17B37

\vskip 0.2in

{\bf Dedicated to the memory of Professor Kar Ping Shum (1941-2023)}
\end{abstract}

\maketitle

\setcounter{tocdepth}{1}
\tableofcontents
\vspace{-1cm}

\section{Introduction}
\subsection{Restricted quantum loop algebras at roots of unity} Let $\mathfrak{g}$ be a simple Lie algebra over $\mathbb{C}$ and let $\mathfrak{\widehat{g}}$ be the associated affine Lie algebra and $U_q(\widehat{\mathfrak{g}})$ the corresponding quantum affine algebra. There are different versions of the quantum affine algebra $U_q(\widehat{\mathfrak{g}})$ when $q$ is specialized to a root $\varepsilon$ of unity: the non-restricted specialization $\widetilde{U}_{\varepsilon}(\widehat{\mathfrak{g}})$ studied by Beck and Kac \cite{BK96}, the restricted specialization $U_{\varepsilon}^{\res}(\widehat{\mathfrak{g}})$ studied by Chari and Pressley \cite{CP94}\,(following the general definition due to Lusztig \cite{L93}), and the small affine quantum group $U_{\varepsilon}^{\fin}(\widehat{\mathfrak{g}})$ \cite{CP94}, which is the image of the natural homomorphism $\widetilde{U}_{\varepsilon}(\widehat{\mathfrak{g}}) \rightarrow U_{\varepsilon}^{\res}(\widehat{\mathfrak{g}})$.

Denote by $U_q(L \mathfrak{g})$ the quantum loop algebra, which is isomorphic to a quotient of $U_q(\widehat{\mathfrak{g}})$ where the central charge is mapped to $1$. Denote by $U_q^{\res}({L\mathfrak{g}})$ the $\mathbb{C}[q,q^{-1}]$-subalgebra of $U_q(L\mathfrak{g})$ generated by the $q$-divided powers of the Chevalley generators {\cite[Section 1]{CP97b}}. Let $U_\varepsilon^{\res}({L\mathfrak{g}})$ be the specialization of $U_q^{\res}({L\mathfrak{g}})$ at root $\varepsilon$ of unity, by setting
\begin{align*}
U_\varepsilon^{\res}({L\mathfrak{g}}):=U_q^{\res}({L\mathfrak{g}}) \otimes_{\mathbb{C}[q, q^{-1}]}\mathbb{C}
\end{align*}
via the algebra homomorphism $f_{\varepsilon}: \mathbb{C}[q, q^{-1}]\rightarrow \mathbb{C}$, that takes $q$ to $\varepsilon$.

In 1997, Chari and Pressley \cite{CP97b} classified the finite-dimensional irreducible modules of $U_{\varepsilon}^{\res}(L\mathfrak{g})$ in terms of highest weights, where $\varepsilon$ is a root of unity of odd order. Frenkel and Mukhin \cite{FM02} extended the Chari and Pressley's result to all roots of unity. Denote by $I$ the set of vertices of the Dynkin diagram of $\mathfrak{g}$. The finite-dimensional simple modules of $U_{\varepsilon}^{\res}(L\mathfrak{g})$ are classified by $(P_i(u))_{i \in I}$, see \cite[Theorem 8.2]{CP97b} and \cite[Theorem 2.4]{FM02}, where every $P_i(u) \in \mathbb{C}[u]$ is a polynomial with constant term $1$ which is called a Drinfeld polynomial. Every $I$-tuple $(P_i(u))_{i \in I}$ of Drinfeld polynomials corresponds to a dominant monomial $m$ in formal variables $Y_{i,a}$, $i \in I$, $a \in \mathbb{C}^{\times}$, where dominant means the exponents of $Y_{i,a}$ in $m$ are non-negative. Denote by $L(m)$ the corresponding simple $U_{\varepsilon}^{\res}(L\mathfrak{g})$-module.

\subsection{$q$-characters and $\varepsilon$-characters}
Frenkel and Reshetikhin \cite{FR98} introduced the theory of $q$-characters of finite-dimensional modules of quantum loop algebras. Denote by ${\rm Rep} U_{q}(L{\mathfrak{g}})$ the category of finite-dimensional $U_q (L{\mathfrak{g}})$-modules and by $\mathcal{K}_0({\rm Rep} U_{q}(L{\mathfrak{g}}))$ the Grothendieck ring of ${\rm Rep} U_{q}(L{\mathfrak{g}})$. The $q$-character map is an injective ring homomorphism from the Grothendieck ring $\mathcal{K}_0({\rm Rep} U_{q}(L{\mathfrak{g}}))$ to the ring $\mathbb{Z}[Y_{i,a}^{\pm1}]_{i \in I}^{a\in\mathbb{C^{\times}}}$ of Laurent polynomials in the variables $Y_{i,a}$, $i \in I$, $a \in \mathbb{C}^{\times}$.

Frenkel and Mukhin \cite{FM02} studied $\varepsilon$-characters of finite-dimensional modules of quantum loop algebras. They defined the $\varepsilon$-character of a $U_{\varepsilon}^{\res}(L{\mathfrak{g}})$-module $V$ via the generalized eigenvalues of a commutative subalgebra of $U_{\varepsilon}^{\res}(L{\mathfrak{g}})$ on $V$. Denote by ${\rm Rep}U_\varepsilon^{\res}({L\mathfrak{g}})$ the category of finite-dimensional $U_\varepsilon^{\res}({L\mathfrak{g}})$-modules and by $\mathcal{K}_0({\rm Rep}U_\varepsilon^{\res}({L\mathfrak{g}}))$ the Grothendieck ring of ${\rm Rep}U_\varepsilon^{\res}({L\mathfrak{g}})$. They showed that the map $\chi_{\varepsilon}: \mathcal{K}_0({\rm Rep}U_\varepsilon^{\res}({L\mathfrak{g}}))\rightarrow \mathbb{Z}[Y_{i,a}^{\pm1}]_{i\in I}^{a\in\mathbb{C^{\times}}}$ is an injective homomorphism of rings.

The theories of $q$-characters and $\varepsilon$-characters are important in the study of quantum loop algebras. It is important to give combinatorial descriptions of $q$-characters and $\varepsilon$-characters.

Mukhin and Young \cite{MY12a,MY12b} gave an explicit path description for $q$-characters of snake modules, which is a family of finite-dimensional modules. The family of minimal affinizations contains the family of Kirillov-Reshetikhin modules. The family of snake modules contains the family of minimal affinizations. Brito and Mukhin \cite{BM17} extended the Mukhin and Young's methods to $q$-characters of the extended snake modules of type $B_n$.

Brito and Chari \cite{BC19} studied a family of finite-dimensional modules called Hernandez-Leclerc modules. These modules first appeared in \cite{HL10,HL13}. In \cite{GDL22}, Guo, Duan, and Luo gave a path description for $q$-characters of Hernandez-Leclerc modules of type $A_n$, where overlapped paths are allowed. In \cite{J24}, Jang gave a path description for $q$-characters of fundamental modules of type $C_n$. In \cite{T23}, Tong, Duan, and Luo gave a path description for $q$-characters of fundamental modules of type $D_n$.

\subsection{Path description of $\varepsilon$-characters}
Path description of $\varepsilon$-characters has not been studied in the literature. In this paper, we study path descriptions for $\varepsilon$-characters of simple modules of type $A$ restricted quantum loop algebras $U_\varepsilon^{\res}({L\mathfrak{sl}_{n+1}})$ at roots of unity.

In this paper, unless we say otherwise, for any integer $\ell\geq 2$, we denote by $\varepsilon$ the root of unity
\begin{align*}
\varepsilon:=
\begin{cases} \text{exp}(\frac{i\pi}{\ell}), & \text{if}{\hskip 0.4em} \ell {\hskip 0.4em}\text{is even},\\
              \text{exp}(\frac{2i\pi}{\ell}), & \text{if}{\hskip 0.4em} \ell {\hskip 0.4em} \text{is odd}.
\end{cases}
\end{align*}
The integer $\ell$ is the order of $\varepsilon^{2}$, and we have $\varepsilon^{2 \ell}=1$. Following \cite{FM02}, let us also write $\varepsilon^{*}:=\varepsilon^{\ell^2}$. For $a\leq b \in \ZZ$, we denote $[a,b]=\{a,a+1,\ldots,b\}$.

A finite-dimensional $U_\varepsilon^{\res}(L\mathfrak{g})$-module $V$ is called \textit{special} if $\chi_{\varepsilon}(V)$ contains exactly one dominant monomial.

For $i\in I$, $k\in \mathbb{Z}_{\geq 0}$, $a\in\mathbb{C^{\times}}$, the simple $U_{\varepsilon}^{\res}({L\mathfrak{sl}_{n+1}})$-module
\begin{align}
W_{i,a}^{(k)}=L(Y_{i,a}Y_{i,a\varepsilon^2}\cdots Y_{i,a\varepsilon^{2(k-1)}})
\end{align}
is called a \textit{Kirillov-Reshetikhin} module. In particular, we have $W_{i,a}^{(1)}=L(Y_{i,a})$. By convention, $W_{i,a}^{(0)}$ is the trivial module for every $i \in I$ and $a \in \CC^{\times}$.

For $\mathfrak{a} \in \{0, 1\}$, we denote
\begin{align*}
\mathcal{X}_{\mathfrak{a}}:=\{(i,k)\in I \times \mathbb{Z}: i-k \equiv \mathfrak{a}  \pmod 2\} \subset I \times \mathbb{Z}.
\end{align*}

In this paper, we fix $a \in \CC^{\times}$ and for convenience we write $Y_{i,s}=Y_{i,a\varepsilon^s}$ for $i \in I$, $s \in \mathbb{Z}$.

For $\mathfrak{a} \in \{0,1\}$, denote by $\mathcal{C}_{\mathcal{X}_{\mathfrak{a}}}$ the full subcategory of ${\rm Rep}U_\varepsilon^{\res}({L\mathfrak{g}})$ whose objects have all their composition factors of the form $L(m)$, where $m$ is a dominant monomial in $Y_{i,s}$, $(i,s) \in \mathcal{X}_{\mathfrak{a}}$. We have the following result.
\begin{theorem}[Theorem \ref{Th:simple modules are snake modules}]
For $\mathfrak{a} \in \{0,1\}$ and $\varepsilon^{2\ell}=1$ with $\ell \ge 2$, any simple module of $U_{\varepsilon}^{\res}({L\mathfrak{sl}_{n+1}})$ in $\mathcal{C}_{\mathcal{X}_{\mathfrak{a}}}$ can be converted to a snake module.
\end{theorem}

For a simple $U_{\varepsilon}^{\res}({L\mathfrak{g}})$-module $L(m)$, $m=Y_{i_1,k_1}Y_{i_2,k_2} \cdots Y_{i_z,k_z}$, $i_t\in I$, $k_t \in \mathbb{Z}$, $t \in [1,z]$, $z \in \ZZ_{\ge 0}$, we say that the highest $l$-weight monomial $m$ of $L(m)$ has degree $z$.

For any path $p$, denote by $\mathfrak{m}(p)$ the monomial corresponding to $p$, see (\ref{map: path to monomial}). For any simple $U_{\varepsilon}^{\res}({L\mathfrak{sl}_{n+1}})$-module with highest $l$-weight of degree two, we have the following theorem.

\begin{theorem}[Theorem \ref{Th:path description of degree 2}]\label{thm: introduction path description of module with highest monomial degree 2}
Let $\varepsilon^{2 \ell}=1$ and let $L(Y_{i,k}Y_{j,v})$ be a simple $U_{\varepsilon}^{\res}({L\mathfrak{sl}_{n+1}})$-module, where $i, j \in [1,n]$, $k, v \in \ZZ$. Then $L(Y_{i,k}Y_{j,v})$ is special and
\begin{enumerate}
\item if $|j-i|\equiv |k-v|+1\ (\text{mod}\ 2)$, then
\[
\chi_{\varepsilon}(L(Y_{i,k}Y_{j,v}))=\left(\sum_{p \in \mathscr{P}_{i,k}} \mathfrak{m}(p)\right) \left(\sum_{p \in \mathscr{P}_{j,v}} \mathfrak{m}(p)\right)=\chi_{\varepsilon}(L(Y_{i,k}))\chi_{\varepsilon}(L(Y_{j,v}));
\]
\item if $|j-i|\equiv |k-v|\ (\text{mod}\ 2)$, then
\begin{align}\label{Eq:equation of degree two}
\chi_{\varepsilon}(L(Y_{i,k}Y_{j,v}))=\left(\sum_{(p_{1},p_{2}) \in \overline{\mathscr{P}}_{((i,k),(j,\overline{v}))}} \mathfrak{m}(p_1)\mathfrak{m}(p_2)\right)-\chi_{(i,k), (j,\overline{v})},
 \end{align}
where $\overline{v}$ is defined in $(\ref{Eq:small index})$ and $\chi_{(i,k), (j,\overline{v})}$ is defined in Definition \ref{Def: chi_(i,k,j,v)}. 
\end{enumerate}
\end{theorem}

The first term on the right-hand side of Equation $(\ref{Eq:equation of degree two})$ is given by the path description for snake modules \cite{MY12a}. In order to characterize $\chi_{(i,k), (j,\overline{v})}$, we introduce path translations, see Definition \ref{Def: translation}.

Theorem \ref{thm: introduction path description of module with highest monomial degree 2} gives an effective path description for the $\varepsilon$-character of any simple module with the highest $l$-weight monomial of degree two.

For a Kirillov-Reshetikhin module $L(Y_{i,k_1}\cdots Y_{i,k_z})$ of $U_{\varepsilon}^{\res}({L\mathfrak{sl}_{n+1}})$, where $i\in I$, $z \in\ZZ_{\geq 1}$, $k_t \in \ZZ$, $t \in [1,z]$, we say that the dominant monomial $Y_{i,k_1}\cdots Y_{i,k_z}$ has small values of indices if $Y_{i,k_j}Y_{i,k_{j+1}}$ has small values of indices, that is $k_{j+1}=k_{j}+2$, for $1\leq j<z$ (see Definition \ref{def:small values of indices}).

For any Kirillov-Reshetikhin module of $U_{\varepsilon}^{\res}({L\mathfrak{sl}_{n+1}})$, we have the following theorem.

\begin{theorem}[Theorem \ref{Th: the path description of K-R module}]\label{TH:k-r}
Let $\varepsilon^{2 \ell}=1$ and let $L(Y_{i,k_1}\cdots Y_{i,k_z})$ be a Kirillov-Reshetikhin module of $U_{\varepsilon}^{\res}({L\mathfrak{sl}_{n+1}})$, where $Y_{i,k_1}\cdots Y_{i,k_z}$ has small values of indices, $i \in I=[1,n]$, $z \in [1, \ell]$, $k_t \in \ZZ$, $t \in [1,z]$. Then
\begin{align*}
\chi_{\varepsilon}(L(Y_{i,k_1}\cdots Y_{i,k_z}))=\sum_{(p_{1},\ldots,p_{z}) \in \overline{\mathscr{P'}}_{(i,k_{t})_{1 \leq t \leq z}}} \prod_{t=1}^{z}\mathfrak{m}(p_{t}),
\end{align*}
where $\overline{\mathscr{P'}}_{(i,k_{t})_{1\leq t\leq z}}$ is defined in $(\ref{def: disjoint path in a tube})$.
\end{theorem}

Let $\varepsilon^{2 \ell}=1$ and let $L(Y_{i,k_1}\cdots Y_{i,k_z})$ be a Kirillov-Reshetikhin module of $U_{\varepsilon}^{\res}({L\mathfrak{sl}_{n+1}})$, where $Y_{i,k_1}\cdots Y_{i,k_z}$ has small values of indices. In the case of $z> \ell$, one can apply Theorem \ref{Th:decomposition} and Theorem \ref{TH:k-r} to obtain the $\varepsilon$-character $\chi_{\varepsilon}(L(Y_{i,k_1}\cdots Y_{i,k_z}))$, see Remark \ref{remark:path description for KR modules with large degree}.

\subsection{Conditions of irreducibility of tensor product of fundamental modules of $U_{\varepsilon}^{\res}({L\mathfrak{sl}_{n+1}})$}
In 1997, Chari and Pressley showed that for any finite-dimensional irreducible $U_{\varepsilon}^{\res}({L\mathfrak{sl}_{n+1}})$-module $L(m)$, there exist integers $\xi_{1},\ldots, \xi_{m}$ and indexes $i_{1},\ldots, i_{m}\in I$ such that $L(m)$ is isomorphic to a subquotient of $L(Y_{i_1,\xi_{1}})\otimes L(Y_{i_2,\xi_2})\otimes \cdots \otimes L(Y_{i_m,\xi_m})$. They gave a necessary and sufficient condition for the tensor product of fundamental modules of $U_{\varepsilon}^{\res}({L\mathfrak{sl}_2})$ to be irreducible, where $\varepsilon$ is a root of unity of odd order, see {\cite[Theorem 9.6]{CP97b}}. In 2007, Abe in \cite{A07} gave a necessary and sufficient condition for the tensor product of fundamental modules of $U_{\varepsilon}^{\res}({L\mathfrak{sl}_{n+1}})$ to be irreducible, where $\varepsilon^{2\ell+1}=1$, $\ell\geq 2$.

Fix $\varepsilon$ such that $\varepsilon^{2\ell}=1$ for some $\ell \ge 2$. As an application of our path description, we obtain a necessary and sufficient condition for the tensor product of two fundamental representations of $U_{\varepsilon}^{\res}({L\mathfrak{sl}_{n+1}})$ to be irreducible. Additionally, we obtain a necessary condition for the tensor product of two or more fundamental representations of $U_{\varepsilon}^{\res}({L\mathfrak{sl}_{n+1}})$ to be irreducible.

\begin{theorem}[Theorem \ref{Th: tensor product}, Corollary \ref{Co:necessary condition for the tensor product}]
Let $m\in \mathbb{Z}_{\geq 2}$, $i_{1},\ldots, i_{m}\in I$, $\xi_{1},\ldots, \xi_{m}\in \mathbb{Z}$, and $\varepsilon^{2 \ell}=1$, $\ell \ge 2$.
\begin{enumerate}
    \item Suppose that $m=2$. The tensor product $L(Y_{i_1,\xi_1})\otimes L(Y_{i_2,\xi_2})$ is a simple module of $U_{\varepsilon}^{\res}(L\mathfrak{sl}_{n+1})$ if and only if $|\xi_{2}-\xi_1 | \not\equiv \pm(2t+|i_2-i_1|)\,(\text{mod } 2 \ell)$, where $1\leq t\leq \min\{i_1, i_2, n+1-i_1, n+1-i_2\}$.
    \item Suppose that $m \geq 2$. If the tensor product
$L(Y_{i_1,\xi_1})\otimes L(Y_{i_2,\xi_2})\otimes \cdots \otimes L(Y_{i_m,\xi_m})$ is a simple module of $U_{\varepsilon}^{\res}(L\mathfrak{sl}_{n+1})$,
then for any $1\leq k\neq k'\leq m$ and $1\leq t\leq \min\{i_k, i_{k'}, n+1-i_k, n+1-i_{k'}\}$, $| \xi_{k'}-\xi_{k} | \not\equiv \pm(2t+|i_{k'}-i_k|)\, (\text{mod } 2 \ell)$.
\end{enumerate}
\end{theorem}

\subsection{Organization of the paper}
In Section \ref{Preliminaries}, we review the quantum loop algebras, restricted quantum loop algebras $U_\varepsilon^{\res}({L\mathfrak{g}})$ at roots of unity, finite-dimensional $U_\varepsilon^{\res}({L\mathfrak{g}})$-modules and their $\varepsilon$-characters. In Section \ref{sec:snake modules and three types of translations of paths}, we show that any finite-dimensional simple module of restricted quantum loop algebra $U_{\varepsilon}^{\res}({L\mathfrak{sl}_{n+1}})$ in a certain category can be transformed into a snake module. Meanwhile, we introduce the concept of path translations. In Section \ref{section of path description of degree two}, we obtain an effective and concrete path description for the $\varepsilon$-character of any simple $U_\varepsilon^{\res}({L\mathfrak{sl}_{n+1}})$-module with highest $l$-weight of degree two. Subsequently, we give an application of our path description to study conditions of irreducibility of tensor products of fundamental modules. In Section \ref{Kirillov-Reshetikhin}, we obtain an effective and concrete path description for $\varepsilon$-characters of Kirillov-Reshetikhin modules of $U_\varepsilon^{\res}({L\mathfrak{sl}_{n+1}})$. In Section \ref{Sec: prove the th path description of degree 2}, we prove Theorem \ref{Th:path description of degree 2}.

\section{Preliminaries}\label{Preliminaries}
In this section, we recall some known results about quantum loop algebras $U_q({L\mathfrak{g}})$ and restricted quantum loop algebras $U_\varepsilon^{\res}({L\mathfrak{g}})$ at roots of unity \cite{CP97b, FM02}.

\subsection{Cartan data}\label{cartan data}
Let $\mathfrak{g}$ be a simple Lie algebra over $\mathbb{C}$ and $I$ the set of indices of the Dynkin diagram of $\mathfrak{g}$. Let $\{\alpha_i\}_{i \in I}$ be the set of simple roots. Let $C=(c_{ij})_{i,j \in I}$ be the Cartan matrix of $\mathfrak{g}$, where $c_{ij}=\frac{2 (\alpha_i, \alpha_j) }{(\alpha_i, \alpha_i)}$. There is a matrix $D=\text{diag}(d_{i}\mid i\in I)$ with positive integer entries $d_i$, $i\in I$, such that $B=DC=(b_{ij})_{i,j \in I}$ is symmetric.

Denote by $L\mathfrak{g}= \mathfrak{g}\otimes\mathbb{C}[t,t^{-1}]$ the loop algebra of $\mathfrak{g}$ and denote by $\widehat{\mathfrak{g}}$ the affine Lie algebra corresponding to $\mathfrak{g}$. Let $\hat{I}=I\cup\{0\}$ and let $(c_{ij})_{i,j \in \hat{I}}$ be the generalized Cartan matrix of $\widehat{\mathfrak{g}}$.

\subsection{Quantum loop algebras}\label{definition of quantum affine algebras}

Let $q$ be an indeterminate, $\mathbb{C}(q)$ the field of rational functions of $q$ with complex coefficients, and $\mathbb{C} [q,q^{-1}]$ the ring of complex Laurent polynomials in $q$. For $m \in \mathbb{Z}_{\geq 0}$, set
\begin{align*}
[m]_{q} := \frac{q^m-q^{-m}}{q-q^{-1}},\quad  [m]_{q} ! := [m]_{q}[m-1]_{q} \cdots [1]_{q}.
\end{align*}
Denote $q_i=q^{d_i}$, $i \in \hat{I}$.

The quantum affine algebra $U_q(\widehat{\mathfrak{g}})$ in the Drinfeld-Jimbo realization \cite{Dri85, Jim85} is an associative algebra over $\mathbb{C}(q)$ with generators $e_i^{\pm}, k_i^{\pm 1}\,(i \in \hat{I})$, subject to certain relations. In Drinfeld's new realization \cite{Dri88}, $U_q(\widehat{\mathfrak{g}})$ is generated by $x_{i, r}^{\pm}$ ($i \in I, r \in \mathbb{Z}$), $k_i^{\pm 1}$ $(i \in I)$, $h_{i, r}$ ($i \in I, r \in \mathbb{Z}\backslash \{0\}$) and central elements $c^{\pm 1/2}$, subject to certain relations.

Denote by $U_q({L\mathfrak{g}})$ the quantum loop algebra, which is isomorphic to a quotient of $U_q(\widehat{\mathfrak{g}})$ where the central charge is mapped to $1$. Therefore, $U_q({L\mathfrak{g}})$ inherits a Hopf algebra structure. For more information on $U_q({L\mathfrak{g}})$, we refer the reader to \cite{CH10, CP94, CP95, Le11}.

Note that the algebra $U_q({L\mathfrak{g}})$ is defined over $\CC(q)$. By Theorem 2.1 in \cite{CP97b}, the irreducible highest weight representation of $U_q({L\mathfrak{g}})$ is finite-dimensional if and only if it corresponds to an $I$-tuple of polynomials $(P_i(u))_{i\in I}$, where $P_i(u) \in \CC(q)[u]$, $P_i(0) \ne 0$. These polynomials are called Drinfeld polynomials.

Let $\mathbb{Z}\mathcal{P} = \mathbb{Z}[Y_{i,a}^{\pm1}]_{i\in I}^{a\in\mathbb{C^{\times}}}$. The $q$-character of a $U_q(L\mathfrak{g})$-module $V$ is given by
\begin{align*}
\chi_q(V) = \sum_{m \in \mathcal{P}} \dim(V_{m}) m \in \mathbb{Z}\mathcal{P},
\end{align*}
where $V_m$ is the $l$-weight space with $l$-weight $m$ \cite{FR98}.

\subsection{Restricted quantum loop algebras $U_\varepsilon^{\res}({L\mathfrak{g}})$ at roots of unity}
For $i \in \hat{I}$, $r \in \mathbb{Z}_{>0}$, denote
\[
(e_i^{\pm})^{(r)}=\frac{(e_i^{\pm})^r}{[r]_{q_i}!}.
\]
Let $U_q^{\res}({L\mathfrak{g}})$ be the $\mathbb{C}[q, q^{-1}]$-subalgebra of $U_q({L\mathfrak{g}})$ generated by the $k_i^{\pm 1}$ and the $(e_i^{\pm})^{(r)}$ for all $i \in \hat{I}$, $r \in \mathbb{Z}_{>0}$, see \cite[Section 1]{CP97b}.
Let $U_\varepsilon^{\res}({L\mathfrak{g}})$ be the specialization of $U_q^{\res}({L\mathfrak{g}})$ at root $\varepsilon$ of unity, by setting
\begin{align*}
U_\varepsilon^{\res}({L\mathfrak{g}}):=U_q^{\res}({L\mathfrak{g}}) \otimes_{\mathbb{C}[q, q^{-1}]}\mathbb{C}
\end{align*}
via the algebra homomorphism $f_{\varepsilon}: \mathbb{C}[q, q^{-1}]\rightarrow \mathbb{C}$, that takes $q$ to $\varepsilon$.
For an element $x$ of $U_q^{\res}({L\mathfrak{g}})$, we denote the corresponding element of $U_\varepsilon^{\res}({L\mathfrak{g}})$ also by $x$.

A finite-dimensional $U_\varepsilon^{\res}({L\mathfrak{g}})$-module $V$ has an $\varepsilon$-character $\chi_{\varepsilon}(V)$ \cite[Section 3]{FM02}, which is an element of $\mathbb{Z}[Y_{i,a}^{\pm1}]_{i\in I}^{a\in\mathbb{C^{\times}}}$. The finite-dimensional simple modules of $U_{\varepsilon}^{\res}(L\mathfrak{g})$ are classified by $I$-tuples $(P_i(u))_{i \in I}$, see \cite[Theorem 8.2]{CP97b} and \cite[Theorem 2.4]{FM02}, where each $P_i(u) \in \mathbb{C}[u]$ is a polynomial with constant term $1$ which is called a Drinfeld polynomial. 

Denote by $\mathcal{P}$ the free abelian multiplicative group of monomials in infinitely many formal variables $Y_{i, a}$, $i \in I, a \in \mathbb{C}^{\times}$. A monomial $m=\prod_{i \in I, a \in \mathbb{C}^{\times}} Y_{i, a}^{u_{i, a}}$, where $u_{i, a}$ are integers, is said to be \textit{dominant} (resp. \textit{anti-dominant}) if $u_{i, a} \geq 0$ (resp. $u_{i, a} \leq 0$) for all $a \in \mathbb{C}^{\times}, i \in I$. Let $\mathcal{P}^+ \subset \mathcal{P}$ denote the set of all dominant monomials. Every $I$-tuple $(P_i(u))_{i \in I}$ of Drinfeld polynomials corresponds to a dominant monomial $m$ in formal variables $Y_{i,a}$, $i \in I$, $a \in \mathbb{C}^{\times}$. Denote by $L(m)$ the corresponding simple $U_{\varepsilon}^{\res}(L\mathfrak{g})$-module.

Fundamental modules of $U_\varepsilon^{\res}({L\mathfrak{g}})$ are the simple modules $L(Y_{i,a})$, where $i \in I, a \in \mathbb{C}^{\times}$, and standard modules are the tensor products of fundamental modules.

For a simple module $V$ of $U_q({L\mathfrak{g}})$, with the highest weight vector $v$, it is known \cite[Proposition 2.5]{FM02} that the $U_q^{\res}({L\mathfrak{g}})$-module $V^{\res}:= U_q^{\res}({L\mathfrak{g}}).v$ is a free $\mathbb{C}[q, q^{-1}]$-module. Put $V_\varepsilon^{\res}=V^{\res}\otimes_{\mathbb{C}[q, q^{-1}]}\mathbb{C}$, where as above $q$ acts on $\mathbb{C}$ by multiplication by $\varepsilon$. This is a $U_\varepsilon^{\res}({L\mathfrak{g}})$-module called the specialization of $V$ at $q=\varepsilon$. Denote by ${\rm Rep}U_\varepsilon^{\res}({L\mathfrak{g}})$ the category of finite-dimensional $U_\varepsilon^{\res}({L\mathfrak{g}})$-modules and by $\mathcal{K}_0({\rm Rep}U_\varepsilon^{\res}({L\mathfrak{g}}))$ the Grothendieck ring of the category ${\rm Rep}U_\varepsilon^{\res}({L\mathfrak{g}})$. Frenkel and Mukhin {\cite[Theorem 3.2]{FM02}} proved that the $\varepsilon$-character map $\chi_{\varepsilon}: \mathcal{K}_0({\rm Rep} U_\varepsilon^{\res}({L\mathfrak{g}})) \rightarrow \mathbb{Z}[Y_{i,a}^{\pm1}]_{i\in I}^{a\in\mathbb{C^{\times}}}$ is an injective homomorphism of rings. Moreover, for any irreducible finite-dimensional $U_q(L\mathfrak{g})$-module $V$, $\chi_{\varepsilon}(V_{\varepsilon}^{\res})$ is obtained from $\chi_{q}(V)$ by setting $q$ equal to $\varepsilon$.

\begin{remark}
In general, for a simple $U_q(L\mathfrak{g})$-module $L(m)$, the polynomial obtained from $\chi_q(L(m))$ by sending $q$ to $\varepsilon$ has more monomials than the $\varepsilon$-character $\chi_{\varepsilon}(L(m))$ of the simple module $L(m)$ considered as a $U_\varepsilon^{\res}({L\mathfrak{g}})$-module. For example, consider the simple $U_q(L\mathfrak{sl}_2)$-module $L(Y_{1,a}Y_{1,aq^2})$, we have that
\[
\chi_q(L(Y_{1,a}Y_{1,aq^2})) = Y_{1,a}Y_{1,aq^2}+Y_{1,a}Y_{1,aq^4}^{-1}+Y_{1,aq^2}^{-1}Y_{1,aq^4}^{-1}.
\]
On the other hand,
\[
\chi_{\varepsilon}(L(Y_{1,a}Y_{1,a\varepsilon^2})) =Y_{1,a}Y_{1,a\varepsilon^2}+Y_{1,a\varepsilon^2}^{-1}Y_{1,a\varepsilon^4}^{-1}
\]
when $L(Y_{1,a}Y_{1,a\varepsilon^2})$ is considered as a $U_\varepsilon^{\res}({L\mathfrak{sl}_2})$-module, where $\varepsilon^{2 \ell}=1$, $\ell=2$.
\end{remark}

Similar to modules of quantum loop algebras \cite[Section 2.3]{MY12a}, a finite-dimensional $U_\varepsilon^{\res}(L\mathfrak{g})$-module $V$ is called \textit{special} if $\chi_{\varepsilon}(V)$ contains exactly one dominant monomial.

Since the Dynkin diagram of $\mathfrak{g}$ is a bipartite graph, we may choose a partition of the vertices $I=I_0\cup I_1$, where each edge connects a vertex of $I_0$ with a vertex of $I_1$. For $i\in I$, set
\begin{align}\label{Def: xi_i}
\xi_i=
\begin{cases} 0, & \text{if}{\hskip 0.4em} i\in I_0,\\
              1, & \text{if}{\hskip 0.4em} i\in I_1.
\end{cases}
\end{align}
Following \cite{FM02}, for $i\in I$ and $a\in \mathbb{C^{\times}}$, denote
\begin{align}\label{mathbf{Y}}
\mathbf{Y}_{i,a}:=\prod_{j=0}^{\ell-1} Y_{i,a\varepsilon^{2j+\xi_i}}.
\end{align}
Since $\varepsilon^{2\ell}=1$, we have $\mathbf{Y}_{i,\varepsilon^{2r}}=\mathbf{Y}_{i,1}$ for any $r\in \mathbb{Z}$. A monomial in the variables $Y_{i,a}$ is said to be $\ell$-\textit{acyclic} if it is not divisible by $\mathbf{Y}_{j,b}$ for any $j \in I, b\in \mathbb{C}^{\times}$, see \cite[Section 2.6]{FM02}.

Frenkel and Mukhin {\cite[Section 4]{FM02}} described a quantum Frobenius map
\[
{\rm Fr}:  U_\varepsilon^{\res}({L\mathfrak{g}})\rightarrow U_{{\varepsilon}^{\ast}}^{\res}({L\mathfrak{g}})
\]
that gives rise to the Frobenius pullback
\begin{align*}
{\rm Fr}^{\ast}: \mathcal{K}_0({\rm Rep} U_{{\varepsilon}^{\ast}}^{\res}({L\mathfrak{g}}))\rightarrow \mathcal{K}_0({\rm Rep} U_{{\varepsilon}}^{\res}({L\mathfrak{g}}))
\end{align*}
and proved that this is an injective ring homomorphism such that ${\rm Fr}^{\ast}([L(Y_{i,a})])=[L(\mathbf{Y}_{i,a})]$.

Let $L(M)$ be a $U_{\varepsilon^{\ast}}^{\res}({L\mathfrak{g}})$-module. The $U_\varepsilon^{\res}({L\mathfrak{g}})$-module
${\rm Fr}^{\ast}(L(M))$ obtained by pullback of $L(M)$ via the quantum Frobenius homomorphism is called the \textit{Frobenius pullback} of $L(M)$, see {\cite[Section 5.2]{FM02}.

Recall that we denote $\varepsilon^{*}:=\varepsilon^{\ell^2}$. The following theorem was proved by Chari and Pressley \cite{CP97b} for roots of unity of odd order and generalized by Frenkel and Mukhin \cite{FM02} to roots of unity of arbitrary order.

\begin{theorem}[{\cite[Theorem 9.1]{CP97b}}, {\cite[Theorem 5.4]{FM02}}]\label{Th:decomposition}
Let $L(m)$ be a simple module of $U_\varepsilon^{\res}({L\mathfrak{g}})$. Then
\begin{align*}
L(m)\cong L(m^{0})\otimes L(m^{1}),
\end{align*}
using the decomposition $m=m^{0}m^{1}$, where $m^{1}$ is a monomial in the variables $\mathbf{Y}_{i,a}$ and $m^{0}$ is $\ell$-acyclic. Moreover, $L(m^1)$ is the Frobenius pullback of an irreducible $U_{{\varepsilon}^{\ast}}^{\res}({L\mathfrak{g}})$-module.
\end{theorem}
Note that $L(m^1)$ is the Frobenius pullback of an irreducible $U_{{\varepsilon}^{\ast}}^{\res}({L\mathfrak{g}})$-module $L(\widetilde{m}^1)$. The $\varepsilon$-character $\chi_{\varepsilon}({\rm Fr}^{\ast}(L(\widetilde{m}^1)))$ is obtained from $\chi_{\varepsilon^{\ast}}(L(\widetilde{m}^1))$ by replacing each $Y_{i,a^\ell}^{\pm 1}$ with $\mathbf{Y}_{i,a\varepsilon^{\xi_i}}^{\pm 1}$, where $\xi_i$ is defined in $(\ref{Def: xi_i})$, see \cite[Theorem 5.7]{FM02} and \cite[Section 3]{G16}.

Recall that $\varepsilon^{*}=\varepsilon^{\ell^2}$. If $\ell$ is odd, by the definition of $\varepsilon$ in the Introduction, $\varepsilon^{\ell^2}=1$. If $\ell$ is even, then $\ell^2$ is a multiple of $2\ell$. Since $\varepsilon^{2\ell}=1$, we have that $\varepsilon^{\ell^2}=1$. Therefore $\varepsilon^*=1$. Hence the category of finite-dimensional $U_{{\varepsilon}^{*}}^{\res}({L\mathfrak{g}})$-modules is equivalent to the category of finite-dimensional $L\mathfrak{g}$-modules, see \cite[Section 5.4]{FM02}. For an arbitrary complex Lie algebra $\mathfrak{g}$ and a non-zero constant $a$, we have the evaluation homomorphism
\[\phi_a: L{\mathfrak{g}}=\mathfrak{g}\otimes\mathbb{C}[t,t^{-1}]\rightarrow \mathfrak{g},\quad\,\, g\otimes t^k\mapsto a^kg.\]
For an irreducible $\mathfrak{g}$-module $V_{\lambda}$ with the highest weight $\lambda$, let $V_\lambda(a)$ be its pullback under $\phi_a$ to an irreducible module of $L{\mathfrak{g}}$. Let $\chi(V_\lambda)$ be the ordinary character of the $\mathfrak{g}$-module $V_\lambda$, considered as a polynomial in $y_i^{\pm 1}$, $i\in I$. Then $\chi_{{\varepsilon}^\ast}(V_\lambda(a))$ is obtained from $\chi(V_\lambda)$ by replacing each $y_i^{\pm 1}$ with $Y_{i,a}^{\pm 1}$, see {\cite[Lemma 5.8]{FM02}}.

Therefore, the computation of $\varepsilon$-characters of simple modules of $U_\varepsilon^{\res}({L\mathfrak{g}})$ is reduced, by Theorem \ref{Th:decomposition}, to understanding $\varepsilon$-characters of modules $L(m^0)$, where $L(m^0)$ is $\ell$-acyclic.

From now on, we fix an $a\in \mathbb{C}^{\times}$ and for convenience we write $Y_{i,s}=Y_{i,a\varepsilon^s}$ for $i \in I$, $s \in \mathbb{Z}$.

\section{Snake modules and path translations} \label{sec:snake modules and three types of translations of paths}
In this section, we show that any finite-dimensional simple module of restricted quantum loop algebra $U_{\varepsilon}^{\res}({L\mathfrak{sl}_{n+1}})$ in a certain category, where $\varepsilon$ is a root of unity, can be converted to a snake module. Then we introduce the concept of path translations. These translations of paths will be used in the study of path descriptions for $\varepsilon$-characters.

\subsection{Snake modules }\label{Snake modules}
We first recall the definition of snake modules which was introduced by Mukhin and Young in \cite{MY12a}. For $\mathfrak{a} \in \{0,1\}$, we denote
\begin{align}\label{Eq:X_a}
\mathcal{X}_{\mathfrak{a}}:=\{(i,k)\in I \times \mathbb{Z}: i-k \equiv \mathfrak{a} \pmod 2\} \subset I \times \mathbb{Z}.
\end{align}
For $\mathfrak{a} \in \{0,1\}$, denote by $\mathcal{C}_{\mathcal{X}_{\mathfrak{a}}}$ the full subcategory of ${\rm Rep}U_\varepsilon^{\res}({L\mathfrak{g}})$ whose objects have
all their composition factors of the form $L(m)$, where $m$ is a dominant monomial in $Y_{i,s}$, $(i,s) \in \mathcal{X}_{\mathfrak{a}}$.

For $(i,k) \in \mathcal{X}_{\mathfrak{a}}$, a point $(i',k')$ is said to be in \textit{snake position} with respect to $(i,k)$ if
\[
k'-k \geq |i'-i|+2 \ \text{and} \  k'-k \equiv |i'-i|\pmod 2.
\]
The point $(i',k')$ is in \textit{minimal} snake position with respect to $(i,k)$ if $k'-k$ is equal to the given lower bound. Denote by $h_{0}(i,i')$ the given lower bound, that is,
\[
h_{0}(i,i')=|i'-i|+2.
\]
For $(i,k) \in \mathcal{X}_{\mathfrak{a}}$, a point $(i',k') \in \mathcal{X}_{\mathfrak{a}}$ is said to be in \textit{prime snake position} with respect to $(i,k)$ if
\[
\min \{ 2n+2-i-i', i+i' \} \geq k'-k \geq |i'-i|+2 \ \text{and} \  k'-k \equiv |i'-i| \pmod 2.  \\
\]
For $(i,k) \in \mathcal{X}_{\mathfrak{a}}$, we define
\begin{align}\label{Prime snake}
\mathbf{PS}(i,k)=\{(i',k') \in \mathcal{X}_{\mathfrak{a}}: (i',k')\, \text{is in prime snake position with respect to}\,(i,k)\},
\end{align}
and for $j\in I$, set
\begin{align}\label{Prime snake 2}
\mathbf{PS}^{j}(i,k)=\{(j,k') \in \mathcal{X}_{\mathfrak{a}}: (j,k') \in \mathbf{PS}(i,k)\}.
\end{align}
A finite sequence $(i_{t},k_{t})$, $1 \leq t \leq z$, $z \in \mathbb{Z}_{>0}$, of points in $\mathcal{X}_{\mathfrak{a}}$ is called a \textit{snake} if for all $2 \leq t \leq z$, the point $(i_{t},k_{t})$ is in snake position with respect to $(i_{t-1},k_{t-1})$. It is called a \textit{minimal} (resp. \textit{prime}) snake if all successive points are in minimal (resp. prime) snake position \cite[Section 4]{MY12a}.

The simple module $L(m)$ is called a \textit{snake module} (resp. a \textit{minimal snake module}) if $m=\prod_{t=1}^{z} Y_{i_{t},k_{t}}$ for some snake $(i_{t},k_{t})_{1 \leq t \leq z}$ (resp. for some minimal snake $(i_{t},k_{t})_{1 \leq t \leq z}$). In this case, $(i_{t},k_{t})_{1 \leq t \leq z}$ is called the snake of $L(m)$ \cite[Section 4]{MY12a}.

We have the following result.
\begin{theorem}\label{Th:simple modules are snake modules}
For $\mathfrak{a} \in \{0,1\}$ and $\varepsilon^{2\ell}=1$, any simple module of $U_{\varepsilon}^{\res}({L\mathfrak{sl}_{n+1}})$ in $\mathcal{C}_{\mathcal{X}_{\mathfrak{a}}}$ can be converted to a snake module.
\end{theorem}

\begin{proof}
For $\mathfrak{a} \in \{0,1\}$, we know that any simple $U_{\varepsilon}^{\res}({L\mathfrak{sl}_{n+1}})$-module in $\mathcal{C}_{\mathcal{X}_{\mathfrak{a}}}$ is of the form $L(m)$, where $m=Y_{i_1,k_1}^{a_1}Y_{i_2,k_2}^{a_2} \cdots Y_{i_r,k_r}^{a_r}$, $i_1\leq i_2\leq \cdots \leq i_r$, $(i_t, k_t) \in \mathcal{X}_{\mathfrak{a}}$, $a_t \in \mathbb{Z}_{\geq 0}$, $1\leq t\leq r$, $r\in \mathbb{Z}_{\geq 1}$. Let $b_1=0$ and $b_t$, $2\leq t\leq r$, be an integer such that $b_{t}\geq \frac{k_{t-1}+2(b_{t-1}+a_{t-1}-1) \ell+i_{t}-i_{t-1}+2-k_{t}}{2 \ell}$. Let $m'= \prod_{s=1}^{r} \prod_{t=b_s}^{b_s+a_s-1}Y_{i_s,k_s+2 t\ell}$. Since $\varepsilon^{2 \ell}=1$, we have $m'= m$ and the module $L(m')$ is a snake module.
\end{proof}

\begin{example}
Let $\varepsilon^{2 \ell}=1$ with $\ell=3$ and $\mathfrak{a}=1$. For the $U_{\varepsilon}^{\res}({L\mathfrak{sl}_{5}})$-module $L(Y_{1,0}Y_{2,3}Y_{3,4}^{2}Y_{4,3})$ in $\mathcal{C}_{\mathcal{X}_{\mathfrak{a}}}$, we choose $b_1=0$, $b_2=0$, $b_3=1$, $b_4=2$, $b_5=3$ in the proof of Theorem \ref{Th:simple modules are snake modules}, then we obtain a snake module $L(Y_{1,0}Y_{2,3}Y_{3,10}Y_{3,16}Y_{4,21})$.
\end{example}

Throughout this paper, when we write the highest $l$-weight monomial
\[
m=Y_{i_1,k_1} Y_{i_2,k_2} \cdots Y_{i_N,k_N}
\]
of a snake module $L(m)$, we assume that $k_t$, $1 \leq t \leq N$, are in increasing order, and for $1\leq \xi, \eta \leq N$, we define
\begin{align}\label{Eq:complementing set}
h(i_{\xi},i_{\eta})=|k_{\eta}-k_{\xi}|.
\end{align}

A \textit{path} is a finite sequence of points in the plane $\mathbb{R}^{2}$. We write $(j,l) \in p$ if $(j,l)$ is a point of the path $p$. When we draw paths, we connect consecutive points of the path by line segments, for illustrative purposes, see Figure \ref{F: figure 1} (b).

For $(i,k) \in \mathcal{X}_{\mathfrak{a}}$, let
\begin{align}\label{path}
\mathscr{P}_{i,k}=\{ & ((0,y_{0}),(1,y_{1}),\ldots,(n+1,y_{n+1})): y_{0}=i+k, \nonumber \\
&y_{n+1}=n+1-i+k, \text{ and } y_{j+1}-y_{j}\in \{1,-1\}, \  0\leq j\leq n\}.
\end{align}

For a path $((a, b), (a+1, b+1), \ldots, (a+r, b+r))$, $r \in \ZZ_{\ge 1}$, we also write it as $((a, b), (a+r, b+r))$. Similarly, for a path $((a, b), (a+1, b-1), \ldots, (a+r, b-r))$, $r \in \ZZ_{\ge 1}$, we also write it as $((a, b), (a+r, b-r))$. Sometimes we use this notation to shorten expressions involving paths. For example, we also denote the path $( (0, 4 ), (1, 3), (2, 2), (3, 3), (4, 4), (5, 3), (6, 2) )$ by any of the following: 
\begin{align*}
& ( (0, 4), (2, 2), (3, 3), (4, 4), (5, 3), (6, 2) ), \\
& ( (0, 4), (1, 3), (2, 2), (4, 4), (5, 3), (6, 2) ), \\
& ( (0, 4), (1, 3), (2, 2), (3, 3), (4, 4), (6, 2) ), \\
& ( (0, 4), (2, 2), (4, 4), (5, 3), (6, 2) ),\\
& ( (0, 4), (2, 2), (3, 3), (4, 4), (6, 2) ), \\
& ( (0, 4), (1, 3), (2, 2), (4, 4), (6, 2) ), \\
& ( (0, 4), (2, 2), (4, 4), (6, 2) ). \\ 
\end{align*} 

For two paths $p, p'$, we denote by $p \cdot p'$ their concatenation if the end point of $p$ is the same as the starting point of $p'$, and $p \cdot p'$ is undefined otherwise. For two sets $P, P'$ of paths, we denote 
\begin{align}\label{eq:concatenation of two sets of paths}
P \uplus P' = \{p \cdot p': p \in P, p' \in P' \}.
\end{align}

The sets $C_{p}^{\pm}$ of upper and lower corners of a path $p=((r,y_{r}))_{0\leq r \leq n+1}\in \mathscr{P}_{i,k}$ are defined as follows:
\begin{align*}
C^{+}_{p}=\{(r,y_{r})\in p: r\in I, \ y_{r-1}=y_{r}+1=y_{r+1}\},\\
C^{-}_{p}=\{(r,y_{r})\in p: r\in I, \ y_{r-1}=y_{r}-1=y_{r+1}\}.
\end{align*}

For $(i,k) \in \mathcal{X}_{\mathfrak{a}}$, we denote by ${\bf P}_{i,k}$ a rectangle with four vertices $(i,k), (0,i+k),(n+1-i,n+1+k)$, and $(n+1,n+1-i+k)$, see Figure \ref{F: figure 1} (a).

A map $\mathfrak{m}$ sending paths to monomials is defined by
\begin{align}\label{map: path to monomial}
\mathfrak{m}: \bigsqcup_{(i,k)\in \mathcal{X}_{\mathfrak{a}}}{\hskip -0.5em}\mathscr{P}_{i,k} & \longrightarrow \mathbb{Z}[Y_{j,l}^{\pm 1}]_{(j,l)\in \mathcal{X}_{\mathfrak{a}}} \nonumber \\
p\quad & \longmapsto  \mathfrak{m}(p)=\prod_{(j,l)\in C^{+}_{p}}{\hskip -0.5em}Y_{j,l}{\hskip -0.5em}\prod_{(j,l)\in C^{-}_{p}}{\hskip -0.5em}Y_{j,l}^{-1}.
\end{align}

Let $p,p'$ be paths. We say that $p$ is \textit{above} (resp. \textit{strictly above}) $p'$ or $p'$ is \textit{below} (resp. \textit{strictly below}) $p$ if
\begin{align*}
(x,y)\in p \text{ and } (x,z)\in p' \Longrightarrow y \leq z \quad (\text{resp. } y<z).
\end{align*}
We say that a $z$-tuple of paths $(p_{1},\ldots,p_{z})$ is \textit{non-overlapping} if $p_{s}$ is strictly above $p_{t}$ for all $s<t$. For any snake $(i_{t},k_{t})\in \mathcal{X}_{\mathfrak{a}}$, $1\leq t\leq z$, $z\in \mathbb{Z}_{\geq1}$, let
\begin{align}\label{Non-overlapping paths}
\overline{\mathscr{P}}_{(i_{t},k_{t})_{1\leq t\leq z}}=\{(p_{1},\ldots,p_{z}): p_{t}\in \mathscr{P}_{i_{t},k_{t}}, 1\leq t\leq z, (p_{1},\ldots,p_{z})\text{ is } \text {non-overlapping}\}.
\end{align}

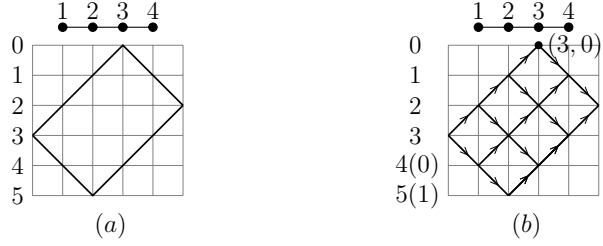
\begin{figure}
\resizebox{0.8\width}{0.8\height}{
\begin{minipage}[b]{0.4\linewidth}
\centerline{
\begin{tikzpicture}
\draw[step=.5cm,gray,thin] (-0.5,5.5) grid (2,8) (-0.5,5.5)--(2,5.5);
\draw[fill] (0,8.3) circle (2pt) -- (0.5,8.3) circle (2pt) --(1,8.3) circle (2pt) --(1.5,8.3) circle (2pt);
\draw[thick] (-0.5,6.5)--(0,7)--(0.5,7.5)--(1,8)--(1.5,7.5)--(2,7);
\draw[thick] (-0.5,6.5)--(0,6)--(0.5,5.5)--(1,6)--(1.5,6.5)--(2,7);
\node [above] at (0,8.3)   {$1$};
\node [above] at (0.5,8.3) {$2$};
\node [above] at (1,8.3)   {$3$};
\node [above] at (1.5,8.3) {$4$};
\node [left] at (-0.5,8)   {$0$};
\node [left] at (-0.5,7.5) {$1$};
\node [left] at (-0.5,7)   {$2$};
\node [left] at (-0.5,6.5) {$3$};
\node [left] at (-0.5,6)   {$4$};
\node [left] at (-0.5,5.5) {$5$};
\node at (0.8,5) {$(a)$};
\end{tikzpicture}}
\end{minipage}
\begin{minipage}[b]{0.4\linewidth}
\centerline{
\begin{tikzpicture}
\draw[step=.5cm,gray,thin] (-0.5,5.5) grid (2,8) (-0.5,5.5)--(2,5.5);
\draw[fill] (0,8.3) circle (2pt) -- (0.5,8.3) circle (2pt) --(1,8.3) circle (2pt) --(1.5,8.3) circle (2pt);
\begin{scope}[thick, every node/.style={sloped,allow upside down}]
\draw (-0.5,6.5)--node {\midarrow}(0,6);
\draw (0,6)--node {\midarrow}(0.5,5.5);
\draw (0.5,5.5)--node {\midarrow}(1,6);
\draw (1,6)--node {\midarrow}(1.5,6.5);
\draw (1.5,6.5)--node {\midarrow}(2,7);
\draw (-0.5,6.5)--node {\midarrow}(0,7);
\draw (0,7)--node {\midarrow}(0.5,7.5);
\draw (0.5,7.5)--node {\midarrow}(1,8);
\draw (1,8)--node {\midarrow}(1.5,7.5);
\draw (1.5,7.5)--node {\midarrow}(2,7);
\draw (0,7)--node {\midarrow}(0.5,6.5);
\draw (0.5,6.5)--node {\midarrow}(1,6);
\draw (0.5,7.5)--node {\midarrow}(1,7);
\draw (1,7)--node {\midarrow}(1.5,6.5);
\draw (1,7)--node {\midarrow}(1.5,7.5);
\draw (0,6)--node {\midarrow}(0.5,6.5);
\draw (0.5,6.5)--node {\midarrow}(1,7);
\draw (0.5,5.5)--node {\midarrow}(1,6);
\draw (1,6)--node {\midarrow}(1.5,6.5);
\draw[fill] (1,8) circle (1.5pt);
\end{scope}
\node [above] at (0,8.3)   {$1$};
\node [above] at (0.5,8.3) {$2$};
\node [above] at (1,8.3)   {$3$};
\node [above] at (1.5,8.3) {$4$};
\node [left] at (-0.8,8)   {$0$};
\node [left] at (-0.8,7.5) {$1$};
\node [left] at (-0.8,7)   {$2$};
\node [left] at (-0.8,6.5) {$3$};
\node [left] at (-0.5,6)   {$4(0)$};
\node [left] at (-0.5,5.5) {$5(1)$};
\node[right] at (1,8)    {$(3,0)$};
\node at (0.8,5) {$(b)$};
\end{tikzpicture}}
\end{minipage}}
\caption{(a) ${\bf P}_{3,0}$. (b)
$\varepsilon^{2\ell}=1, \ell=2$. The paths corresponding to the monomials of $U_{\varepsilon}^{\res}({L\mathfrak{sl}_5})$-module $\chi_\varepsilon(L(Y_{3,0}))$.}\label{F: figure 1}
\end{figure}

For $(i,k) \in \mathcal{X}_{\mathfrak{a}}$, the highest path in $\mathscr{P}_{i,k}$ is the unique path with no lower corners, and the lowest path in $\mathscr{P}_{i,k}$ is the unique path with no upper corners.

When we draw paths, we relabel the row $r$ below $x=2 \ell$ by the remainder of $r$ modulo $2 \ell$, see for example, Figure \ref{F: figure 1} (b).

For any snake module $L(Y_{i_1, k_1}Y_{i_2, k_2}\cdots Y_{i_z,k_z})$, $z \in \mathbb{Z}_{>0}$, of $U_q(L\mathfrak{sl}_{n+1})$, Mukhin and Young {\cite[Theorem 6.1]{MY12a}} proved that
\[
\chi_q(L(Y_{i_1, k_1}Y_{i_2, k_2}\cdots Y_{i_z, k_z}))=\sum_{(p_{1},\ldots,p_{z}) \in \overline{\mathscr{P}}_{(i_t,k_t)_{1 \leq t \leq z}}} \prod_{t=1}^{z}\mathfrak{m}(p_{t}).
\]
When $z=1$, the above formula gives $q$-characters of fundamental modules of $U_q(L\mathfrak{sl}_{n+1})$. The $q$-character of a fundamental module of $U_q(L\mathfrak{sl}_{n+1})$ is the summation of the monomials corresponding to paths in a rectangle. For any path in a rectangle, the first indices of any two corners are different. Therefore when sending $q \mapsto \varepsilon$ in the $q$-character of a fundamental module $L(Y_{i,s})$, the only dominant monomial is the highest weight monomial $Y_{i,s}$. We have the following lemma.
\begin{lemma}
The $\varepsilon$-character of any $U_{\varepsilon}^{\res}({L\mathfrak{sl}_{n+1}})$ fundamental module $L(m)$ is obtained from $\chi_{q}(L(m))$ by setting $q$ equal to $\varepsilon$.
\end{lemma}

\subsection{Lowering and raising moves}\label{subsec:lowdef}
For $i\in I$, $s\in \mathbb{Z}$, define
\begin{align*}
 A_{i,s} = Y_{i,s+1} Y_{i,s-1} \prod_{c_{ji}=-1} Y_{j,s}^{-1}.
\end{align*}
For $(j,y_j)\in I\times \mathbb{Z}$, $r\in \mathbb{Z}_{\geq 1}$, $r<j$, we denote
\begin{align*}
\textbf{A}_{j,y_j+r}^{-1}=&\prod_{t=0}^{r-1}(A_{j-t,y_j+t+1}^{-1}A_{j-t+2,y_j+t+1}^{-1}\cdots A_{j+t,y_j+t+1}^{-1}) \times\\
&\times \prod_{t=0}^{r-2} (A_{j-t,y_j+2r-t-1}^{-1}A_{j-t+2,y_j+2r-t-1}^{-1}\cdots A_{j+t,y_j+2r-t-1}^{-1}).
\end{align*}

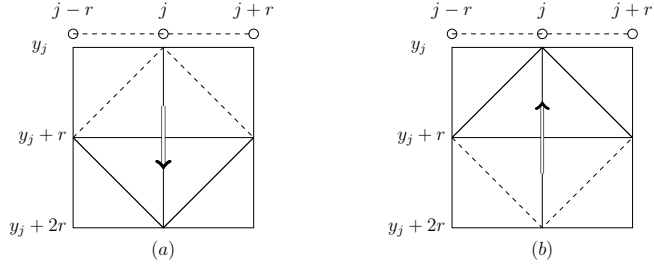
\begin{figure}
\resizebox{0.6\width}{0.6\height}{
\begin{minipage}[b]{0.5\linewidth}
\centerline{
\begin{tikzpicture}
\draw[dashed] (0,10.3)--(2,10.3)--(4,10.3);
\draw (0,10.3) circle (3pt) (2,10.3) circle (3pt) (4,10.3) circle (3pt);
\draw (0,6)--(0,10)--(4,10)--(4,6)--(0,6);
\draw (0,8)--(4,8);
\draw (2,10)--(2,6);
\draw[dashed](0,8)--(2,10)--(4,8);
\draw[thick] (0,8)--(2,6)--(4,8);
\draw[double distance=.6mm,->,shorten <= 3mm,shorten >= 3mm] (2,9) -- (2,7);
\node [above] at (0,10.5) {$j-r$};
\node [above] at (2,10.5) {$j$};
\node [above] at (4,10.5) {$j+r$};
\node [left] at (0,10) {$y_j\quad$};
\node [left] at (0,8) {$y_j+r$};
\node [left] at (0,6) {$y_j+2r$};
\node at (2,5.5) {$(a)$};
\end{tikzpicture}}
\end{minipage}
\begin{minipage}[b]{0.5\linewidth}
\centerline{
\begin{tikzpicture}
\draw[dashed] (0,10.3)--(2,10.3)--(4,10.3);
\draw (0,10.3) circle (3pt) (2,10.3) circle (3pt) (4,10.3) circle (3pt);
\draw (0,6)--(0,10)--(4,10)--(4,6)--(0,6);
\draw (0,8)--(4,8);
\draw (2,10)--(2,6);
\draw[thick](0,8)--(2,10)--(4,8);
\draw[dashed] (0,8)--(2,6)--(4,8);
\draw[double distance=.5mm,->,shorten <= 2mm,shorten >= 2mm] (2,7) -- (2,9);
\node [above] at (0,10.5) {$j-r$};
\node [above] at (2,10.5) {$j$};
\node [above] at (4,10.5) {$j+r$};
\node [left] at (0,10) {$y_j\quad$};
\node [left] at (0,8) {$y_j+r$};
\node [left] at (0,6) {$y_j+2r$};
\node at (2,5.5) {$(b)$};
\end{tikzpicture}}
\end{minipage}}
\caption{(a) Lowering move of a path at $(j,y_j+r)$. (b) Raising move of a path at $(j,y_j+r)$.}\label{F: rasing and lowering moves}
\end{figure}

 Let $(i,k) \in \mathcal{X}_{\mathfrak{a}}$, $\mathfrak{a} \in \{0, 1\}$. We say a path $p\in \mathscr{P}_{i,k}$ can be lowered at point $(j,y_j+r)\in I \times \mathbb{Z}$, $r\in \mathbb{Z}_{\geq 1}$, $r<j$, if $(j,y_j)\in C^{+}_{p}$ and for any $i\in{(j-r,j)}\cup{(j,j+r)}$, $(i,y_i)$ is neither in $C^{+}_{p}$ nor in $C^{-}_{p}$. If so, we define a lowering move on $p$ at $(j,y_j+r)$, resulting in another path in $\mathscr {P}_{i,k}$ which we write as $p\textbf{A}_{j,y_j+r}^{-1}$. We call such lowering move a width $r$ lowering move. That is, if $p={((i,y_i))}_{0\leq i\leq n+1}$, where ${(y_i)_{0\leq i\leq n+1}}\in \mathbb R^{n+2}$, then
\begin{align*}
p\textbf{A}_{j,y_j+r}^{-1}=&((0,y_0), (1,y_1), \ldots, (j-r,y_{j-r}), (j-r+1,y_{j-r}+1), \ldots, (j,y_{j-r}+r), \\
& \ (j+1,y_{j-r}+r-1),\ldots, (j+r-1,y_{j-r}+1), (j+r,y_{j+r}), \ldots, (n+1,y_{n+1})\,),
\end{align*}
where $y_{j-r}=y_j+r=y_{j+r}$, see Figure \ref{F: rasing and lowering moves} $(a)$ for example.

Dually, we say a path $p\in \mathscr{P}_{i,k}$ can be raised at point $(j,y_j+r)\in I \times \mathbb{Z}$, $r\in \mathbb{Z}_{\geq 1}$, $r<j$, if $(j,y_j+2r)\in C^{-}_{p}$ and for any $i\in{(j-r,j)}\cup{(j,j+r)}$, $(i,y_i)$ is neither in $C^{+}_{p}$ nor in $C^{-}_{p}$. If so, we define a raising move on $p$ at $(j,y_j+r)$, resulting in another path $p'$ in $\mathscr{P}_{i,k}$ which we write as $p\textbf{A}_{j,y_j+r}$. We call such raising move a width $r$ raising move. That is, $p=p'\textbf{A}_{j,y_j+r}^{-1}$ for $p'\in \mathscr P_{i,k}$. If $p'$ exists and is unique, then we define $p\textbf{A}_{j,y_j+r}:=p'$, see Figure \ref{F: rasing and lowering moves} $(b)$ for example.

\begin{remark}
Lowering moves and raising moves of width $1$ defined above coincide with the lowering moves and raising moves introduced by Mukhin and Young in \cite[Section 5]{MY12a}. The lowering moves (resp. raising moves) of width $r$, $r \ge 2$, can be obtained by a sequence of lowering moves (resp. raising moves) of width $1$.
\end{remark}

\subsection{Translation of paths }\label{subsection: translations of paths}

Since $\varepsilon^{2\ell}=1$, the value of a path $p$ does not change if we move it vertically to a place with distance $2m\ell$ from $p$ for any $m \in \ZZ_{\ge 1}$, that is, replace $p$ by a path $p'$ which has distance $2m\ell$ from $p$ for some $m \in \ZZ_{\ge 1}$. For any $j \in I$ and any integers $v, v'$, the shape of the rectangle ${\bf P}_{j,v}$ is the same as the shape of the rectangle ${\bf P}_{j,v'}$. For $v \equiv v' \pmod{2\ell}$, every path $p\in \mathscr{P}_{j,v}$ corresponds to a unique path $p'\in \mathscr{P}_{j,v'}$ in the sense that $\mathfrak{m}(p) = \mathfrak{m}(p')$.
\begin{definition}\label{Def: translation}
Let $\varepsilon^{2 \ell}=1$ and let $(i,k),(j,v)\in \mathcal{X}_{\mathfrak{a}}$. We define a translation of paths in $\mathscr{P}_{j,v}$ to paths in $\mathscr{P}_{j,v'}$ with respect to $\mathbf{PS}(i,k)$ as replacing paths in $\mathscr{P}_{j,v}$ by the corresponding paths in $\mathscr{P}_{j,v'}$, where $v'\equiv v \, (\text{mod } 2 \ell)$ and $(j,v')\in\mathbf{PS}(i,k)$.
\end{definition}

For convenience, we denote
\begin{align}\label{Eq:abbreviation for non-overlapping paths}
S_{(i,k)(j,v)} = \sum_{(p_{1},p_{2}) \in \overline{\mathscr{P}}_{((i,k),(j,v))}}\mathfrak{m}(p_1)\mathfrak{m}(p_2).
\end{align}
We say a monomial $m$ is in $S_{(i,k)(j,v)}$ if $m$ is one of the monomials in $S_{(i,k)(j,v)}$. We use the convention that if $i\in \{0, n+1\}$, then $Y_{i,k}$ is the trivial monomial $1$ in $\mathcal{P}$.

\begin{lemma}\label{Le: I translation of degree 2}
Let $\varepsilon^{2 \ell}=1$ and let $L(Y_{i,k}Y_{j,v})$, $k<v$, be a $U_{\varepsilon}^{\res}({L\mathfrak{sl}_{n+1}})$ snake module. Assume that $h(i,j)\equiv h_{0}(i,j)+2s \,(\text{mod } 2 \ell)$, where $0\leq s<|\mathbf{PS}^{j}(i,k)|$, and $h(i,j)\geq h_{0}(i,j)+2 \ell$. Then there exists a translation of paths in $\mathscr{P}_{j,v}$ to paths in $\mathscr{P}_{j,v'}$ with respect to $\mathbf{PS}(i,k)$, where $v'\equiv v \, (\text{mod } 2 \ell)$, $(j,v')\in \mathbf{PS}(i,k)$, and $v'< v$, such that the following properties hold. For $i\leq j$ (resp. $i>j$), we denote $m=Y_{i-r,k+r}Y_{j+r,v'-r}$ (resp. $m=Y_{j-r,v'-r}Y_{i+r,k+r}$), where $r=\frac{v'-k-h_0(i,j)}2+1$. Then
\begin{enumerate}
    \item the dominant monomial $m$ is in $S_{(i,k)(j,v)}$;
    \item there exist paths $\widetilde{p}_1 \in \mathscr{P}_{i,k}$, $\widetilde{p}_2\in \mathscr{P}_{j,v}$ such that $\mathfrak{m}(\widetilde{p}_1)\mathfrak{m}(\widetilde{p}_2)$ is the lowest $\ell$-weight monomial of $L(m)$ and $\mathfrak{m}(\widetilde{p}_1)\mathfrak{m}(\widetilde{p}_2)$ is in $S_{(i,k)(j,v)}$;
    \item the monomials of $\chi_\varepsilon(L(m))$ are contained in $S_{(i,k)(j,v)}$. 
\end{enumerate}
\end{lemma}

\begin{proof}
Assume that $h(i,j)\equiv h_{0}(i,j)+2s \,(\text{mod } 2 \ell)$, where $0\leq s<|\mathbf{PS}^{j}(i,k)|$, and $h(i,j)\geq h_{0}(i,j)+2 \ell$. Then there exists $(j,v')$ such that $v'+2a \ell= v$ for some $a\in \mathbb{Z}_{\geq 1}$ and $(j,v')\in \mathbf{PS}(i,k)$. By Definition \ref{Def: translation}, there is a translation of paths in $\mathscr{P}_{j,v}$ to paths in $\mathscr{P}_{j,v'}$ with respect to $\mathbf{PS}(i,k)$. In the following, we prove (1), (2), and (3) for the case of $i\leq j$, the proof for the case of $i>j$ is similar.

(1) Let $p_1 \in \mathscr{P}_{i,k}$ be the path which has exactly one lower corner $(j, v')$ and $p_2$ be the highest path in $\mathscr{P}_{j,v'}$ with no lower corners. Since $i\leq j$, then $\mathfrak{m}(p_1)=Y_{i-r,k+r}Y_{j,v'}^{-1}Y_{j+r,v'-r}$, $\mathfrak{m}(p_2)=Y_{j,v'}$, where $r=\frac{v'-k-h_0(i,j)}2+1$. Let $p'_2\in\mathscr{P}_{j,v}$ denote the path that corresponds to $p_2$. Then $\mathfrak{m}(p'_2)=Y_{j,v}=Y_{j,v'}$. Since $v>v'$, we have that the paths $p_1$ and $p'_2$ are non-overlapping. Therefore, the dominant monomial $m=Y_{i-r,k+r}Y_{j+r,v'-r}=\mathfrak{m}(p_1)\mathfrak{m}(p'_2)$ is in $S_{(i,k)(j,v)}$.

(2) Following (1), $m=Y_{i-r,k+r}Y_{j+r,v'-r}$. Let $\widetilde{p}_1$ be the lowest path in $\mathscr{P}_{i,k}$ with no upper corners, and let $\widetilde{p}'_2\in \mathscr{P}_{j,v'}$ be the path which has exactly one upper corner $(n+1-i,n+1+k)$. Let $\widetilde{p}_2$ be the path in $\mathscr{P}_{j,v}$ that corresponds to $\widetilde{p}'_2$. Then $\mathfrak{m}(\widetilde{p}'_2)=\mathfrak{m}(\widetilde{p}_2)$. Since $i\leq j$, we have $\mathfrak{m}(\widetilde{p}_1)=Y_{n+1-i,n+1+k}^{-1}$, $\mathfrak{m}(\widetilde{p}_2)=Y_{n+1-j-r,n+1+v-r}^{-1}Y_{n+1-i,n+1+k+2a \ell}Y_{n+1-i+r,n+1+k+r+2a \ell}^{-1}$, where $v=v'+2a \ell$. Clearly the paths $\widetilde{p}_1$ and $\widetilde{p}_2$ are non-overlapping. Therefore, the lowest $l$-weight monomial of the module $L(m)$ is given by
\[
\mathfrak{m}(\widetilde{p}_1)\mathfrak{m}(\widetilde{p}_2)=Y_{n+1-j-r,n+1+v-r}^{-1}Y_{n+1-i+r,n+1+k+r+2a \ell}^{-1}=Y_{n+1-j-r,n+1+v-r}^{-1}Y_{n+1-i+r,n+1+k+r}^{-1},
\]
which is in $S_{(i,k)(j,v)}$.

(3) According to (1), using the translation of paths in $\mathscr{P}_{j,v}$ to paths in $\mathscr{P}_{j,v'}$ with respect to $\mathbf{PS}(i,k)$, we obtain a dominant monomial $Y_{i-r,k+r}Y_{j+r,v'-r}$. Now, we convert $L(Y_{i-r,k+r}Y_{j+r,v'-r})$ into the snake module $L(Y_{j+r,v'-r}Y_{i-r,w})$, where $w\equiv k+r\, (\text{mod } 2 \ell)$ and $w> k+r$, such that $Y_{j+r,v'-r}Y_{i-r,w}$ has small values of indices (see Definition \ref{def:small values of indices}). Assume that $\mathfrak{m}(p)\mathfrak{m}(p')\in S_{(j+r,v'-r)(i-r,w)}$, where $p\in \mathscr{P}_{i-r,w}$, $p'\in \mathscr{P}_{j+r,v'-r}$. Let $\hat{p} \in \mathscr{P}_{i-r,k+r}$ be the path that corresponds to $p$. Then we have $\mathfrak{m}(p)=\mathfrak{m}(\hat{p})$. Furthermore, the starting point of $\hat{p}$ is $(0,i+k)$ and its ending point is $(n+1,n+1-j+v')$. Similarly, the starting point of $p'$ is $(0,j+v')$ and its ending point is $(n+1,n+1-i+k)$.

Therefore, the paths $\hat{p}$ and $p'$ must intersect at some points. This implies there are paths from $(0,i+k)$ to $(n+1,n+1-i+k)$ and from $(0,j+v')$ to $(n+1,n+1-j+v')$. We choose $\hat{p}_1 \in\mathscr{P}_{i,k}$, which is the path from $(0,i+k)$ to $(n+1,n+1-i+k)$ that lies above all other paths connecting these two points (sometimes, there is only one such path). The remaining path is denoted as $\hat{p}_2\in\mathscr{P}_{j,v'}$, which is the path from $(0,j+v')$ to $(n+1,n+1-j+v')$. Let $\hat{p}'_2\in\mathscr{P}_{j,v}$ denote the path that corresponds to $\hat{p}_2$. Then $\mathfrak{m}(\hat{p}_2)=\mathfrak{m}(\hat{p}'_2)$ and we have that $\mathfrak{m}(p)\mathfrak{m}(p')=\mathfrak{m}(\hat{p}_1)\mathfrak{m}(\hat{p}_2)=\mathfrak{m}(\hat{p}_1)\mathfrak{m}(\hat{p}'_2)$. Since $v>v'$, we have that paths $\hat{p}_1$ and $\hat{p}'_2$ are non-overlapping. This indicates that the monomials of $S_{(j+r,v'-r)(i-r,w)}$ are contained in $S_{(i,k)(j,v)}$.

By Theorem \ref{Th:simple modules are snake modules}, we have $\chi_\varepsilon(L(Y_{j+r,v'-r}Y_{i-r,w}))=\chi_\varepsilon(L(Y_{i-r,k+r}Y_{j+r,v'-r}))$. We know that $\chi_\varepsilon(L(Y_{j+r,v'-r}Y_{i-r,w}))$ is a part of $S_{(j+r,v'-r)(i-r,w)}$. Therefore, we conclude that the assertion is valid.
\end{proof}

\begin{remark}\label{R: I type greater}
Let $\varepsilon^{2 \ell}=1$ and let $L(Y_{i,k}Y_{j,v})$, $k<v$, be a $U_{\varepsilon}^{\res}({L\mathfrak{sl}_{n+1}})$ snake module. Assume that there is a translation of paths in $\mathscr{P}_{j,v}$ to paths in $\mathscr{P}_{j,v'}$ with respect to $\mathbf{PS}(i,k)$, where $v'> v$. If $i\leq j$ (resp. $i>j$), then the dominant monomial $m=Y_{i-r,k+r}Y_{j+r,v'-r}$ (resp. $m=Y_{j-r,v'-r}Y_{i+r,k+r}$) is not in $S_{(i,k)(j,v)}$, where $r=\frac{v'-k-h_0(i,j)}2+1$.
\end{remark}

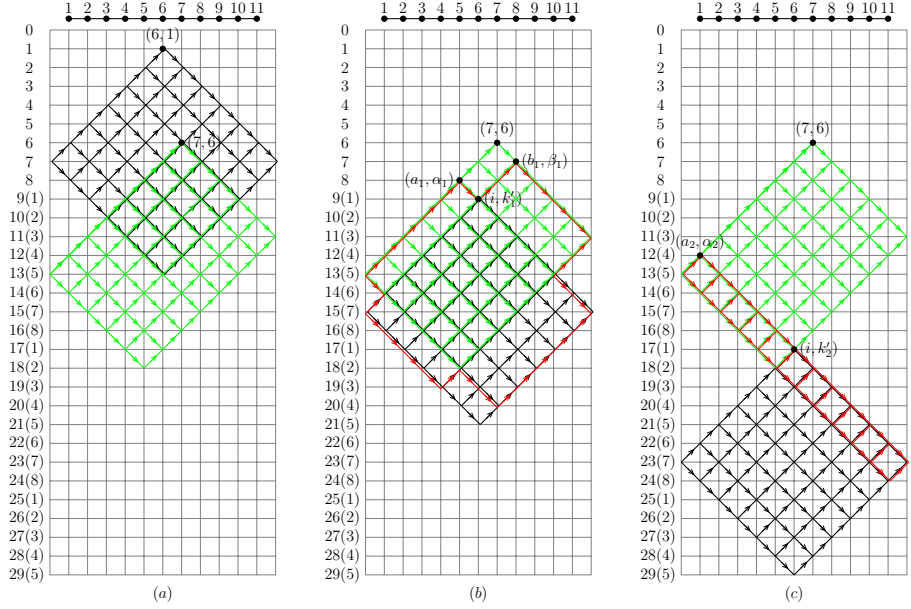
\begin{figure}
\resizebox{0.5\width}{0.5\height}{
\begin{minipage}[b]{0.5\linewidth}
\centerline{
\begin{tikzpicture}
\draw[step=.5cm,gray,thin] (-0.5,5.5) grid (5.5,20) (-0.5,5.5)--(5.5,5.5);
\draw[fill] (0,20.3) circle (2pt) -- (0.5,20.3) circle (2pt) --(1,20.3) circle (2pt) --(1.5,20.3) circle (2pt)--(2,20.3) circle (2pt) -- (2.5,20.3) circle (2pt) --(3,20.3) circle (2pt) --(3.5,20.3) circle (2pt)-- (4,20.3) circle (2pt) -- (4.5,20.3) circle (2pt) --(5,20.3) circle (2pt);
\begin{scope}[thick, every node/.style={sloped,allow upside down}]
\draw (-0.45,16.5)--node {\midarrow}(0.05,17);
\draw (0.05,17)--node {\midarrow}(0.55,17.5);
\draw (0.55,17.5)--node {\midarrow}(1.05,18);
\draw (1.05,18)--node {\midarrow}(1.55,18.5);
\draw (1.55,18.5)--node {\midarrow}(2.05,19);
\draw (2.05,19)--node {\midarrow}(2.55,19.5);
\draw (2.55,19.5)--node {\midarrow}(3.05,19);
\draw (3.05,19)--node {\midarrow}(3.55,18.5);
\draw (3.55,18.5)--node {\midarrow}(4.05,18);
\draw (4.05,18)--node {\midarrow}(4.55,17.5);
\draw (4.55,17.5)--node {\midarrow}(5.05,17);
\draw (5.05,17)--node {\midarrow}(5.55,16.5);
\draw (5.05,16)--node {\midarrow}(5.55,16.5);
\draw (4.55,15.5)--node {\midarrow}(5.05,16);
\draw (4.05,15) --node {\midarrow}(4.55,15.5);
\draw (3.55,14.5) --node {\midarrow}(4.05,15);
\draw (3.05,14) --node {\midarrow}(3.55,14.5);
\draw (2.55,13.5) --node {\midarrow}(3.05,14);
\draw (2.05,14) --node {\midarrow}(2.55,13.5);
\draw (1.55,14.5) --node {\midarrow}(2.05,14);
\draw (1.05,15) --node {\midarrow}(1.55,14.5);
\draw (0.55,15.5)--node {\midarrow}(1.05,15);
\draw (0.05,16)--node {\midarrow}(0.55,15.5);
\draw (-0.45,16.5)--node {\midarrow}(0.05,16);
\draw (0.05,16) --node {\midarrow}(0.55,16.5);
\draw (0.55,16.5) --node {\midarrow}(1.05,17);
\draw (1.05,17) --node {\midarrow}(1.55,17.5);
\draw (1.55,17.5)--node {\midarrow}(2.05,18);
\draw (2.05,18)--node {\midarrow}(2.55,18.5);
\draw (2.55,18.5)--node {\midarrow}(3.05,19);
\draw (0.55,15.5) --node {\midarrow}(1.05,16);
\draw (1.05,16) --node {\midarrow}(1.55,16.5);
\draw (1.55,16.5) --node {\midarrow}(2.05,17);
\draw (2.05,17)--node {\midarrow}(2.55,17.5);
\draw (2.55,17.5)--node {\midarrow}(3.05,18);
\draw (3.05,18)--node {\midarrow}(3.55,18.5);
\draw (1.05,15) --node {\midarrow}(1.55,15.5);
\draw (1.55,15.5) --node {\midarrow}(2.05,16);
\draw (2.05,16) --node {\midarrow}(2.55,16.5);
\draw (2.55,16.5)--node {\midarrow}(3.05,17);
\draw (3.05,17)--node {\midarrow}(3.55,17.5);
\draw (3.55,17.5)--node {\midarrow}(4.05,18);
\draw (1.55,14.5) --node {\midarrow}(2.05,15);
\draw (2.05,15) --node {\midarrow}(2.55,15.5);
\draw (2.55,15.5) --node {\midarrow}(3.05,16);
\draw (3.05,16)--node {\midarrow}(3.55,16.5);
\draw (3.55,16.5)--node {\midarrow}(4.05,17);
\draw (4.05,17)--node {\midarrow}(4.55,17.5);
\draw (2.05,14) --node {\midarrow}(2.55,14.5);
\draw (2.55,14.5) --node {\midarrow}(3.05,15);
\draw (3,15) --node {\midarrow}(3.55,15.5);
\draw (3.55,15.5)--node {\midarrow}(4.05,16);
\draw (4.05,16)--node {\midarrow}(4.55,16.5);
\draw (4.55,16.5)--node {\midarrow}(5.05,17);
\draw (0.05,17) --node {\midarrow}(0.55,16.5);
\draw (0.55,16.5) --node {\midarrow}(1.05,16);
\draw (1.05,16) --node {\midarrow}(1.55,15.5);
\draw (1.55,15.5)--node {\midarrow}(2.05,15);
\draw (2.05,15)--node {\midarrow}(2.55,14.5);
\draw (2.55,14.5)--node {\midarrow}(3.05,14);
\draw (0.55,17.5) --node {\midarrow}(1.05,17);
\draw (1.05,17) --node {\midarrow}(1.55,16.5);
\draw (1.55,16.5) --node {\midarrow}(2.05,16);
\draw (2.05,16)--node {\midarrow}(2.55,15.5);
\draw (2.55,15.5)--node {\midarrow}(3.05,15);
\draw (3.05,15)--node {\midarrow}(3.55,14.5);
\draw (1.05,18) --node {\midarrow}(1.55,17.5);
\draw (1.55,17.5) --node {\midarrow}(2.05,17);
\draw (2.05,17) --node {\midarrow}(2.55,16.5);
\draw (2.55,16.5)--node {\midarrow}(3.05,16);
\draw (3.05,16)--node {\midarrow}(3.55,15.5);
\draw (3.55,15.5)--node {\midarrow}(4.05,15);
\draw (1.55,18.5) --node {\midarrow}(2.05,18);
\draw (2.05,18) --node {\midarrow}(2.55,17.5);
\draw (2.55,17.5) --node {\midarrow}(3.05,17);
\draw (3.05,17)--node {\midarrow}(3.55,16.5);
\draw (3.55,16.5)--node {\midarrow}(4.05,16);
\draw (4.05,16)--node {\midarrow}(4.55,15.5);
\draw (2.05,19) --node {\midarrow}(2.55,18.5);
\draw (2.55,18.5) --node {\midarrow}(3.05,18);
\draw (3.05,18) --node {\midarrow}(3.55,17.5);
\draw (3.55,17.5)--node {\midarrow}(4.05,17);
\draw (4.05,17)--node {\midarrow}(4.55,16.5);
\draw (4.55,16.5)--node {\midarrow}(5.05,16);
\textcolor{green}{
\draw (-0.5,13.5)--node {\midarrow}(0,14);
\draw (0,14)--node {\midarrow}(0.5,14.5);
\draw (0.5,14.5)--node {\midarrow}(1,15);
\draw (1,15)--node {\midarrow}(1.5,15.5);
\draw (1.5,15.5)--node {\midarrow}(2,16);
\draw (2,16)--node {\midarrow}(2.5,16.5);
\draw (2.5,16.5)--node {\midarrow}(3,17);
\draw (3,17)--node {\midarrow}(3.5,16.5);
\draw (3.5,16.5)--node {\midarrow}(4,16);
\draw (4,16)--node {\midarrow}(4.5,15.5);
\draw (4.5,15.5)--node {\midarrow}(5,15);
\draw (5,15)--node {\midarrow}(5.5,14.5);
\draw (-0.5,13.5)--node {\midarrow}(0,13);
\draw (0,13)--node {\midarrow}(0.5,12.5);
\draw (0.5,12.5)--node {\midarrow}(1,12);
\draw (1,12)--node {\midarrow}(1.5,11.5);
\draw (1.5,11.5)--node {\midarrow}(2,11);
\draw (2,11)--node {\midarrow}(2.5,11.5);
\draw (2.5,11.5)--node {\midarrow}(3,12);
\draw (3,12)--node {\midarrow}(3.5,12.5);
\draw (3.5,12.5)--node {\midarrow}(4,13);
\draw (4,13)--node {\midarrow}(4.5,13.5);
\draw (4.5,13.5)--node {\midarrow}(5,14);
\draw (5,14)--node {\midarrow}(5.5,14.5);
\draw (-0.5,13.5)--node {\midarrow}(0,14);
\draw (0,14)--node {\midarrow}(0.5,14.5);
\draw (0.5,14.5)--node {\midarrow}(1,15);
\draw (1,15)--node {\midarrow}(1.5,15.5);
\draw (1.5,15.5)--node {\midarrow}(2,16);
\draw (2,16)--node {\midarrow}(2.5,16.5);
\draw (2.5,16.5)--node {\midarrow}(3,17);
\draw (0,13)--node {\midarrow}(0.5,13.5);
\draw (0.5,13.5)--node {\midarrow}(1,14);
\draw (1,14)--node {\midarrow}(1.5,14.5);
\draw (1.5,14.5)--node {\midarrow}(2,15);
\draw (2,15)--node {\midarrow}(2.5,15.5);
\draw (2.5,15.5)--node {\midarrow}(3,16);
\draw (3,16)--node {\midarrow}(3.5,16.5);
\draw (0.5,12.5)--node {\midarrow}(1,13);
\draw (1,13)--node {\midarrow}(1.5,13.5);
\draw (1.5,13.5)--node {\midarrow}(2,14);
\draw (2,14)--node {\midarrow}(2.5,14.5);
\draw (2.5,14.5)--node {\midarrow}(3,15);
\draw (3,15)--node {\midarrow}(3.5,15.5);
\draw (3.5,15.5)--node {\midarrow}(4,16);
\draw (1,12)--node {\midarrow}(1.5,12.5);
\draw (1.5,12.5)--node {\midarrow}(2,13);
\draw (2,13)--node {\midarrow}(2.5,13.5);
\draw (2.5,13.5)--node {\midarrow}(3,14);
\draw (3,14)--node {\midarrow}(3.5,14.5);
\draw (3.5,14.5)--node {\midarrow}(4,15);
\draw (4,15)--node {\midarrow}(4.5,15.5);
\draw (1.5,11.5)--node {\midarrow}(2,12);
\draw (2,12)--node {\midarrow}(2.5,12.5);
\draw (2.5,12.5)--node {\midarrow}(3,13);
\draw (3,13)--node {\midarrow}(3.5,13.5);
\draw (3.5,13.5)--node {\midarrow}(4,14);
\draw (4,14)--node {\midarrow}(4.5,14.5);
\draw (4.5,14.5)--node {\midarrow}(5,14);
\draw (0,14)--node {\midarrow}(0.5,13.5);
\draw (0.5,13.5)--node {\midarrow}(1,13);
\draw (1,13)--node {\midarrow}(1.5,12.5);
\draw (1.5,12.5)--node {\midarrow}(2,12);
\draw (2,12)--node {\midarrow}(2.5,11.5);
\draw (0.5,14.5)--node {\midarrow}(1,14);
\draw (1,14)--node {\midarrow}(1.5,13.5);
\draw (1.5,13.5)--node {\midarrow}(2,13);
\draw (2,13)--node {\midarrow}(2.5,12.5);
\draw (2.5,12.5)--node {\midarrow}(3,12);
\draw (1,15)--node {\midarrow}(1.5,14.5);
\draw (1.5,14.5)--node {\midarrow}(2,14);
\draw (2,14)--node {\midarrow}(2.5,13.5);
\draw (2.5,13.5)--node {\midarrow}(3,13);
\draw (3,13)--node {\midarrow}(3.5,12.5);
\draw (1.5,15.5)--node {\midarrow}(2,15);
\draw (2,15)--node {\midarrow}(2.5,14.5);
\draw (2.5,14.5)--node {\midarrow}(3,14);
\draw (3,14)--node {\midarrow}(3.5,13.5);
\draw (3.5,13.5)--node {\midarrow}(4,13);
\draw (2,16)--node {\midarrow}(2.5,15.5);
\draw (2.5,15.5)--node {\midarrow}(3,15);
\draw (3,15)--node {\midarrow}(3.5,14.5);
\draw (3.5,14.5)--node {\midarrow}(4,14);
\draw (4,14)--node {\midarrow}(4.5,13.5);
\draw (2.5,16.5)--node {\midarrow}(3,16);
\draw (3,16)--node {\midarrow}(3.5,15.5);
\draw (3.5,15.5)--node {\midarrow}(4,15);
\draw (4,15)--node {\midarrow}(4.5,14.5);
\draw (4.5,14.5)--node {\midarrow}(5,15);}
\draw[fill] (2.5,19.5) circle (2pt) (3,17) circle (2pt);
\end{scope}
\node [above] at (0,20.3)   {$1$};
\node [above] at (0.5,20.3) {$2$};
\node [above] at (1,20.3)   {$3$};
\node [above] at (1.5,20.3) {$4$};
\node [above] at (2,20.3)   {$5$};
\node [above] at (2.5,20.3) {$6$};
\node [above] at (3,20.3)   {$7$};
\node [above] at (3.5,20.3) {$8$};
\node [above] at (4,20.3)   {$9$};
\node [above] at (4.5,20.3) {$10$};
\node [above] at (5,20.3)   {$11$};
\node [left] at (-0.8,20)    {$0$};
\node [left] at (-0.8,19.5)   {$1$};
\node [left] at (-0.8,19)    {$2$};
\node [left] at (-0.8,18.5)   {$3$};
\node [left] at (-0.8,18)    {$4$};
\node [left] at (-0.8,17.5)   {$5$};
\node [left] at (-0.8,17)     {$6$};
\node [left] at (-0.8,16.5)   {$7$};
\node [left] at (-0.8,16)     {$8$};
\node [left] at (-0.5,15.5)   {$9(1)$};
\node [left] at (-0.5,15)    {$10(2)$};
\node [left] at (-0.5,14.5)   {$11(3)$};
\node [left] at (-0.5,14)     {$12(4)$};
\node [left] at (-0.5,13.5)   {$13(5)$};
\node [left] at (-0.5,13)    {$14(6)$};
\node [left] at (-0.5,12.5)   {$15(7)$};
\node [left] at (-0.5,12)     {$16(8)$};
\node [left] at (-0.5,11.5)   {$17(1)$};
\node [left] at (-0.5,11)     {$18(2)$};
\node [left] at (-0.5,10.5)   {$19(3)$};
\node [left] at (-0.5,10)     {$20(4)$};
\node [left] at (-0.5,9.5)   {$21(5)$};
\node [left] at (-0.5,9)     {$22(6)$};
\node [left] at (-0.5,8.5)    {$23(7)$};
\node [left] at (-0.5,8)     {$24(8)$};
\node [left] at (-0.5,7.5)    {$25(1)$};
\node [left] at (-0.5,7)      {$26(2)$};
\node [left] at (-0.5,6.5)    {$27(3)$};
\node [left] at (-0.5,6)     {$28(4)$};
\node [left] at (-0.5,5.5)  {$29(5)$};
\node [above] at (2.5,19.5) {$(6,1)$};
\node [right] at (3,17) {$(7,6)$};
\node at (2.5,5) {$(a)$};
\end{tikzpicture}}
\end{minipage}
\begin{minipage}[b]{0.5\linewidth}
\centerline{
\begin{tikzpicture}
\draw[step=.5cm,gray,thin] (-0.5,5.5) grid (5.5,20) (-0.5,5.5)--(5.5,5.5);
\draw[fill] (0,20.3) circle (2pt) -- (0.5,20.3) circle (2pt) --(1,20.3) circle (2pt) --(1.5,20.3) circle (2pt)--(2,20.3) circle (2pt) -- (2.5,20.3) circle (2pt) --(3,20.3) circle (2pt) --(3.5,20.3) circle (2pt)-- (4,20.3) circle (2pt) -- (4.5,20.3) circle (2pt) --(5,20.3) circle (2pt);
\begin{scope}[thick, every node/.style={sloped,allow upside down}]
\draw (-0.45,12.5)--node {\midarrow}(0.05,13);
\draw (0.05,13)--node {\midarrow}(0.55,13.5);
\draw (0.55,13.5)--node {\midarrow}(1.05,14);
\draw (1.05,14)--node {\midarrow}(1.55,14.5);
\draw (1.55,14.5)--node {\midarrow}(2.05,15);
\draw (2.05,15)--node {\midarrow}(2.55,15.5);
\draw (2.55,15.5)--node {\midarrow}(3.05,15);
\draw (3.05,15)--node {\midarrow}(3.55,14.5);
\draw (3.55,14.5)--node {\midarrow}(4.05,14);
\draw (4.05,14)--node {\midarrow}(4.55,13.5);
\draw (4.55,13.5)--node {\midarrow}(5.05,13);
\draw (5.05,13)--node {\midarrow}(5.55,12.5);
\draw (5.05,12)--node {\midarrow}(5.55,12.5);
\draw (4.55,11.5)--node {\midarrow}(5.05,12);
\draw (4.05,11) --node {\midarrow}(4.55,11.5);
\draw (3.55,10.5) --node {\midarrow}(4.05,11);
\draw (3.05,10) --node {\midarrow}(3.55,10.5);
\draw (2.55,9.5) --node {\midarrow}(3.05,10);
\draw (2.05,10) --node {\midarrow}(2.55,9.5);
\draw (1.55,10.5) --node {\midarrow}(2.05,10);
\draw (1.05,11) --node {\midarrow}(1.55,10.5);
\draw (0.55,11.5)--node {\midarrow}(1.05,11);
\draw (0.05,12)--node {\midarrow}(0.55,11.5);
\draw (-0.45,12.5)--node {\midarrow}(0.05,12);
\draw (0.05,12) --node {\midarrow}(0.55,12.5);
\draw (0.55,12.5) --node {\midarrow}(1.05,13);
\draw (1.05,13) --node {\midarrow}(1.55,13.5);
\draw (1.55,13.5)--node {\midarrow}(2.05,14);
\draw (2.05,14)--node {\midarrow}(2.55,14.5);
\draw (2.55,14.5)--node {\midarrow}(3.05,15);
\draw (0.55,11.5) --node {\midarrow}(1.05,12);
\draw (1.05,12) --node {\midarrow}(1.55,12.5);
\draw (1.55,12.5) --node {\midarrow}(2.05,13);
\draw (2.05,13)--node {\midarrow}(2.55,13.5);
\draw (2.55,13.5)--node {\midarrow}(3.05,14);
\draw (3.05,14)--node {\midarrow}(3.55,14.5);
\draw (1.05,11) --node {\midarrow}(1.55,11.5);
\draw (1.55,11.5) --node {\midarrow}(2.05,12);
\draw (2.05,12) --node {\midarrow}(2.55,12.5);
\draw (2.55,12.5)--node {\midarrow}(3.05,13);
\draw (3.05,13)--node {\midarrow}(3.55,13.5);
\draw (3.55,13.5)--node {\midarrow}(4.05,14);
\draw (1.55,10.5) --node {\midarrow}(2.05,11);
\draw (2.05,11) --node {\midarrow}(2.55,11.5);
\draw (2.55,11.5) --node {\midarrow}(3.05,12);
\draw (3.05,12)--node {\midarrow}(3.55,12.5);
\draw (3.55,12.5)--node {\midarrow}(4.05,13);
\draw (4.05,13)--node {\midarrow}(4.55,13.5);
\draw (2.05,10) --node {\midarrow}(2.55,10.5);
\draw (2.55,10.5) --node {\midarrow}(3.05,11);
\draw (3.05,11) --node {\midarrow}(3.55,11.5);
\draw (3.55,11.5)--node {\midarrow}(4.05,12);
\draw (4.05,12)--node {\midarrow}(4.55,12.5);
\draw (4.55,12.5)--node {\midarrow}(5.05,13);
\draw (0.05,13) --node {\midarrow}(0.55,12.5);
\draw (0.55,12.5) --node {\midarrow}(1.05,12);
\draw (1.05,12) --node {\midarrow}(1.55,11.5);
\draw (1.55,11.5)--node {\midarrow}(2.05,11);
\draw (2.05,11)--node {\midarrow}(2.55,10.5);
\draw (2.55,10.5)--node {\midarrow}(3.05,10);
\draw (0.55,13.5) --node {\midarrow}(1.05,13);
\draw (1.05,13) --node {\midarrow}(1.55,12.5);
\draw (1.55,12.5) --node {\midarrow}(2.05,12);
\draw (2.05,12)--node {\midarrow}(2.55,11.5);
\draw (2.55,11.5)--node {\midarrow}(3.05,11);
\draw (3.05,11)--node {\midarrow}(3.55,10.5);
\draw (1.05,14) --node {\midarrow}(1.55,13.5);
\draw (1.55,13.5) --node {\midarrow}(2.05,13);
\draw (2.05,13) --node {\midarrow}(2.55,12.5);
\draw (2.55,12.5)--node {\midarrow}(3.05,12);
\draw (3.05,12)--node {\midarrow}(3.55,11.5);
\draw (3.55,11.5)--node {\midarrow}(4.05,11);
\draw (1.55,14.5) --node {\midarrow}(2.05,14);
\draw (2.05,14) --node {\midarrow}(2.55,13.5);
\draw (2.55,13.5) --node {\midarrow}(3.05,13);
\draw (3.05,13)--node {\midarrow}(3.55,12.5);
\draw (3.55,12.5)--node {\midarrow}(4.05,12);
\draw (4.05,12)--node {\midarrow}(4.55,11.5);
\draw (2.05,15) --node {\midarrow}(2.55,14.5);
\draw (2.55,14.5) --node {\midarrow}(3.05,14);
\draw (3.05,14) --node {\midarrow}(3.55,13.5);
\draw (3.55,13.5)--node {\midarrow}(4.05,13);
\draw (4.05,13)--node {\midarrow}(4.55,12.5);
\draw (4.55,12.5)--node {\midarrow}(5.05,12);
\textcolor{green}{
\draw (-0.5,13.5)--node {\midarrow}(0,14);
\draw (0,14)--node {\midarrow}(0.5,14.5);
\draw (0.5,14.5)--node {\midarrow}(1,15);
\draw (1,15)--node {\midarrow}(1.5,15.5);
\draw (1.5,15.5)--node {\midarrow}(2,16);
\draw (2,16)--node {\midarrow}(2.5,16.5);
\draw (2.5,16.5)--node {\midarrow}(3,17);
\draw (3,17)--node {\midarrow}(3.5,16.5);
\draw (3.5,16.5)--node {\midarrow}(4,16);
\draw (4,16)--node {\midarrow}(4.5,15.5);
\draw (4.5,15.5)--node {\midarrow}(5,15);
\draw (5,15)--node {\midarrow}(5.5,14.5);
\draw (-0.5,13.5)--node {\midarrow}(0,13);
\draw (0,13)--node {\midarrow}(0.5,12.5);
\draw (0.5,12.5)--node {\midarrow}(1,12);
\draw (1,12)--node {\midarrow}(1.5,11.5);
\draw (1.5,11.5)--node {\midarrow}(2,11);
\draw (2,11)--node {\midarrow}(2.5,11.5);
\draw (2.5,11.5)--node {\midarrow}(3,12);
\draw (3,12)--node {\midarrow}(3.5,12.5);
\draw (3.5,12.5)--node {\midarrow}(4,13);
\draw (4,13)--node {\midarrow}(4.5,13.5);
\draw (4.5,13.5)--node {\midarrow}(5,14);
\draw (5,14)--node {\midarrow}(5.5,14.5);
\draw (-0.5,13.5)--node {\midarrow}(0,14);
\draw (0,14)--node {\midarrow}(0.5,14.5);
\draw (0.5,14.5)--node {\midarrow}(1,15);
\draw (1,15)--node {\midarrow}(1.5,15.5);
\draw (1.5,15.5)--node {\midarrow}(2,16);
\draw (2,16)--node {\midarrow}(2.5,16.5);
\draw (2.5,16.5)--node {\midarrow}(3,17);
\draw (0,13)--node {\midarrow}(0.5,13.5);
\draw (0.5,13.5)--node {\midarrow}(1,14);
\draw (1,14)--node {\midarrow}(1.5,14.5);
\draw (1.5,14.5)--node {\midarrow}(2,15);
\draw (2,15)--node {\midarrow}(2.5,15.5);
\draw (2.5,15.5)--node {\midarrow}(3,16);
\draw (3,16)--node {\midarrow}(3.5,16.5);
\draw (0.5,12.5)--node {\midarrow}(1,13);
\draw (1,13)--node {\midarrow}(1.5,13.5);
\draw (1.5,13.5)--node {\midarrow}(2,14);
\draw (2,14)--node {\midarrow}(2.5,14.5);
\draw (2.5,14.5)--node {\midarrow}(3,15);
\draw (3,15)--node {\midarrow}(3.5,15.5);
\draw (3.5,15.5)--node {\midarrow}(4,16);
\draw (1,12)--node {\midarrow}(1.5,12.5);
\draw (1.5,12.5)--node {\midarrow}(2,13);
\draw (2,13)--node {\midarrow}(2.5,13.5);
\draw (2.5,13.5)--node {\midarrow}(3,14);
\draw (3,14)--node {\midarrow}(3.5,14.5);
\draw (3.5,14.5)--node {\midarrow}(4,15);
\draw (4,15)--node {\midarrow}(4.5,15.5);
\draw (1.5,11.5)--node {\midarrow}(2,12);
\draw (2,12)--node {\midarrow}(2.5,12.5);
\draw (2.5,12.5)--node {\midarrow}(3,13);
\draw (3,13)--node {\midarrow}(3.5,13.5);
\draw (3.5,13.5)--node {\midarrow}(4,14);
\draw (4,14)--node {\midarrow}(4.5,14.5);
\draw (4.5,14.5)--node {\midarrow}(5,14);
\draw (0,14)--node {\midarrow}(0.5,13.5);
\draw (0.5,13.5)--node {\midarrow}(1,13);
\draw (1,13)--node {\midarrow}(1.5,12.5);
\draw (1.5,12.5)--node {\midarrow}(2,12);
\draw (2,12)--node {\midarrow}(2.5,11.5);
\draw (0.5,14.5)--node {\midarrow}(1,14);
\draw (1,14)--node {\midarrow}(1.5,13.5);
\draw (1.5,13.5)--node {\midarrow}(2,13);
\draw (2,13)--node {\midarrow}(2.5,12.5);
\draw (2.5,12.5)--node {\midarrow}(3,12);
\draw (1,15)--node {\midarrow}(1.5,14.5);
\draw (1.5,14.5)--node {\midarrow}(2,14);
\draw (2,14)--node {\midarrow}(2.5,13.5);
\draw (2.5,13.5)--node {\midarrow}(3,13);
\draw (3,13)--node {\midarrow}(3.5,12.5);
\draw (1.5,15.5)--node {\midarrow}(2,15);
\draw (2,15)--node {\midarrow}(2.5,14.5);
\draw (2.5,14.5)--node {\midarrow}(3,14);
\draw (3,14)--node {\midarrow}(3.5,13.5);
\draw (3.5,13.5)--node {\midarrow}(4,13);
\draw (2,16)--node {\midarrow}(2.5,15.5);
\draw (2.5,15.5)--node {\midarrow}(3,15);
\draw (3,15)--node {\midarrow}(3.5,14.5);
\draw (3.5,14.5)--node {\midarrow}(4,14);
\draw (4,14)--node {\midarrow}(4.5,13.5);
\draw (2.5,16.5)--node {\midarrow}(3,16);
\draw (3,16)--node {\midarrow}(3.5,15.5);
\draw (3.5,15.5)--node {\midarrow}(4,15);
\draw (4,15)--node {\midarrow}(4.5,14.5);
\draw (4.5,14.5)--node {\midarrow}(5,15);}
\textcolor{red}{
\draw (-0.5,13.45)--node {\midarrow}(0,13.95)--node {\midarrow}(0.5,14.45)--node {\midarrow}(1,14.95)--node {\midarrow}(1.5,15.45)--node {\midarrow}(2,15.95)--node {\midarrow}(2.5,15.45)--node {\midarrow}(3,15.95)--node {\midarrow}(3.5,16.45)--node {\midarrow}(4,15.95)--node {\midarrow}(4.5,15.45)--node {\midarrow}(5,14.95)--node {\midarrow}(5.5,14.45);
\draw (-0.5,13.45)--node {\midarrow}(0,12.95);
\draw (-0.5,12.45)--node {\midarrow}(0,11.95)--node {\midarrow}(0.5,11.45)--node {\midarrow}(1,10.95)--node {\midarrow}(1.5,10.45)--node {\midarrow}(2,10.95)--node {\midarrow}(2.5,10.45)--node {\midarrow}(3,9.95)--node {\midarrow}(3.5,10.45)--node {\midarrow}(4,10.95)--node {\midarrow}(4.5,11.45)--node {\midarrow}(5,11.95)--node {\midarrow}(5.5,12.45);
\draw (-0.5,12.45)--node {\midarrow}(0,12.95);
\draw (4.5,13.45)--node {\midarrow}(5,13.95)--node {\midarrow}(5.5,14.45);
\draw (4.5,13.45)--node {\midarrow}(5,12.95)--node {\midarrow}(5.5,12.45);}
\draw [fill](2.5,15.5) circle (2pt) (2,16) circle (2pt) (3.5,16.5) circle (2pt) (3,17) circle (2pt) ;
\end{scope}
\node [above] at (0,20.3)   {$1$};
\node [above] at (0.5,20.3) {$2$};
\node [above] at (1,20.3)   {$3$};
\node [above] at (1.5,20.3) {$4$};
\node [above] at (2,20.3)   {$5$};
\node [above] at (2.5,20.3) {$6$};
\node [above] at (3,20.3)   {$7$};
\node [above] at (3.5,20.3) {$8$};
\node [above] at (4,20.3)   {$9$};
\node [above] at (4.5,20.3) {$10$};
\node [above] at (5,20.3)   {$11$};
\node [left] at (-0.8,20)   {$0$};
\node [left] at (-0.8,19.5)   {$1$};
\node [left] at (-0.8,19)     {$2$};
\node [left] at (-0.8,18.5)   {$3$};
\node [left] at (-0.8,18)    {$4$};
\node [left] at (-0.8,17.5)   {$5$};
\node [left] at (-0.8,17)     {$6$};
\node [left] at (-0.8,16.5)   {$7$};
\node [left] at (-0.8,16)     {$8$};
\node [left] at (-0.5,15.5)   {$9(1)$};
\node [left] at (-0.5,15)     {$10(2)$};
\node [left] at (-0.5,14.5)   {$11(3)$};
\node [left] at (-0.5,14)      {$12(4)$};
\node [left] at (-0.5,13.5)   {$13(5)$};
\node [left] at (-0.5,13)     {$14(6)$};
\node [left] at (-0.5,12.5)   {$15(7)$};
\node [left] at (-0.5,12)     {$16(8)$};
\node [left] at (-0.5,11.5)   {$17(1)$};
\node [left] at (-0.5,11)    {$18(2)$};
\node [left] at (-0.5,10.5)   {$19(3)$};
\node [left] at (-0.5,10)    {$20(4)$};
\node [left] at (-0.5,9.5)   {$21(5)$};
\node [left] at (-0.5,9)     {$22(6)$};
\node [left] at (-0.5,8.5)    {$23(7)$};
\node [left] at (-0.5,8)     {$24(8)$};
\node [left] at (-0.5,7.5)    {$25(1)$};
\node [left] at (-0.5,7)     {$26(2)$};
\node [left] at (-0.5,6.5)    {$27(3)$};
\node [left] at (-0.5,6)     {$28(4)$};
\node [left] at (-0.5,5.5)  {$29(5)$};
\node [right] at (2.5,15.5) {$(i, k'_1)$};
\node [left] at  (2,16)    {$(a_1,\alpha_1)$};
\node [right] at (3.5,16.5) {$(b_1,\beta_1)$};
\node [above] at (3,17)     {$(7,6)$};
\node at (2.5,5) {$(b)$};
\end{tikzpicture}}
\end{minipage}
\begin{minipage}[b]{0.5\linewidth}
\centerline{
\begin{tikzpicture}
\draw[step=.5cm,gray,thin] (-0.5,5.5) grid (5.5,20) (-0.5,5.5)--(5.5,5.5);
\draw[fill] (0,20.3) circle (2pt) -- (0.5,20.3) circle (2pt) --(1,20.3) circle (2pt) --(1.5,20.3) circle (2pt)--(2,20.3) circle (2pt) -- (2.5,20.3) circle (2pt) --(3,20.3) circle (2pt) --(3.5,20.3) circle (2pt)-- (4,20.3) circle (2pt) -- (4.5,20.3) circle (2pt) --(5,20.3) circle (2pt);
\begin{scope}[thick, every node/.style={sloped,allow upside down}]
\draw (-0.5,8.5)--node {\midarrow}(0,9);
\draw (0,9)--node {\midarrow}(0.5,9.5);
\draw (0.5,9.5)--node {\midarrow}(1,10);
\draw (1,10)--node {\midarrow}(1.5,10.5);
\draw (1.5,10.5)--node {\midarrow}(2,11);
\draw (2.05,11)--node {\midarrow}(2.55,11.5);
\draw (2.5,11.5)--node {\midarrow}(3,11);
\draw (3,11)--node {\midarrow}(3.5,10.5);
\draw (3.5,10.5)--node {\midarrow}(4,10);
\draw (4,10)--node {\midarrow}(4.5,9.5);
\draw (4.5,9.5)--node {\midarrow}(5,9);
\draw (5,9)--node {\midarrow}(5.5,8.5);
\draw (5,8)--node {\midarrow}(5.5,8.5);
\draw (4.5,7.5)--node {\midarrow}(5,8);
\draw (4,7) --node {\midarrow}(4.5,7.5);
\draw (3.5,6.5) --node {\midarrow}(4,7);
\draw (3,6) --node {\midarrow}(3.5,6.5);
\draw (2.5,5.5) --node {\midarrow}(3,6);
\draw (2,6) --node {\midarrow}(2.5,5.5);
\draw (1.5,6.5) --node {\midarrow}(2,6);
\draw (1,7) --node {\midarrow}(1.5,6.5);
\draw (0.5,7.5)--node {\midarrow}(1,7);
\draw (0,8)--node {\midarrow}(0.5,7.5);
\draw (-0.5,8.5)--node {\midarrow}(0,8);
\draw (0,8) --node {\midarrow}(0.5,8.5);
\draw (0.5,8.5) --node {\midarrow}(1,9);
\draw (1,9) --node {\midarrow}(1.5,9.5);
\draw (1.5,9.5)--node {\midarrow}(2,10);
\draw (2,10)--node {\midarrow}(2.5,10.5);
\draw (2.5,10.5)--node {\midarrow}(3,11);
\draw (0.5,7.5) --node {\midarrow}(1,8);
\draw (1,8) --node {\midarrow}(1.5,8.5);
\draw (1.5,8.5) --node {\midarrow}(2,9);
\draw (2,9)--node {\midarrow}(2.5,9.5);
\draw (2.5,9.5)--node {\midarrow}(3,10);
\draw (3,10)--node {\midarrow}(3.5,10.5);
\draw (1,7) --node {\midarrow}(1.5,7.5);
\draw (1.5,7.5) --node {\midarrow}(2,8);
\draw (2,8) --node {\midarrow}(2.5,8.5);
\draw (2.5,8.5)--node {\midarrow}(3,9);
\draw (3,9)--node {\midarrow}(3.5,9.5);
\draw (3.5,9.5)--node {\midarrow}(4,10);
\draw (1.5,6.5) --node {\midarrow}(2,7);
\draw (2,7) --node {\midarrow}(2.5,7.5);
\draw (2.5,7.5) --node {\midarrow}(3,8);
\draw (3,8)--node {\midarrow}(3.5,8.5);
\draw (3.5,8.5)--node {\midarrow}(4,9);
\draw (4,9)--node {\midarrow}(4.5,9.5);
\draw (2,6) --node {\midarrow}(2.5,6.5);
\draw (2.5,6.5) --node {\midarrow}(3,7);
\draw (3,7) --node {\midarrow}(3.5,7.5);
\draw (3.5,7.5)--node {\midarrow}(4,8);
\draw (4,8)--node {\midarrow}(4.5,8.5);
\draw (4.5,8.5)--node {\midarrow}(5,9);
\draw (0,9) --node {\midarrow}(0.5,8.5);
\draw (0.5,8.5) --node {\midarrow}(1,8);
\draw (1,8) --node {\midarrow}(1.5,7.5);
\draw (1.5,7.5)--node {\midarrow}(2,7);
\draw (2,7)--node {\midarrow}(2.5,6.5);
\draw (2.5,6.5)--node {\midarrow}(3,6);
\draw (0.5,9.5) --node {\midarrow}(1,9);
\draw (1,9) --node {\midarrow}(1.5,8.5);
\draw (1.5,8.5) --node {\midarrow}(2,8);
\draw (2,8)--node {\midarrow}(2.5,7.5);
\draw (2.5,7.5)--node {\midarrow}(3,7);
\draw (3,7)--node {\midarrow}(3.5,6.5);
\draw (1,10) --node {\midarrow}(1.5,9.5);
\draw (1.5,9.5) --node {\midarrow}(2,9);
\draw (2,9) --node {\midarrow}(2.5,8.5);
\draw (2.5,8.5)--node {\midarrow}(3,8);
\draw (3,8)--node {\midarrow}(3.5,7.5);
\draw (3.5,7.5)--node {\midarrow}(4,7);
\draw (1.5,10.5) --node {\midarrow}(2,10);
\draw (2,10) --node {\midarrow}(2.5,9.5);
\draw (2.5,9.5) --node {\midarrow}(3,9);
\draw (3,9)--node {\midarrow}(3.5,8.5);
\draw (3.5,8.5)--node {\midarrow}(4,8);
\draw (4,8)--node {\midarrow}(4.5,7.5);
\draw (2,11) --node {\midarrow}(2.5,10.5);
\draw (2.5,10.5) --node {\midarrow}(3,10);
\draw (3,10) --node {\midarrow}(3.5,9.5);
\draw (3.5,9.5)--node {\midarrow}(4,9);
\draw (4,9)--node {\midarrow}(4.5,8.5);
\draw (4.5,8.5)--node {\midarrow}(5,8);
\textcolor{green}{
\draw (-0.5,13.5)--node {\midarrow}(0,14);
\draw (0,14)--node {\midarrow}(0.5,14.5);
\draw (0.5,14.5)--node {\midarrow}(1,15);
\draw (1,15)--node {\midarrow}(1.5,15.5);
\draw (1.5,15.5)--node {\midarrow}(2,16);
\draw (2,16)--node {\midarrow}(2.5,16.5);
\draw (2.5,16.5)--node {\midarrow}(3,17);
\draw (3,17)--node {\midarrow}(3.5,16.5);
\draw (3.5,16.5)--node {\midarrow}(4,16);
\draw (4,16)--node {\midarrow}(4.5,15.5);
\draw (4.5,15.5)--node {\midarrow}(5,15);
\draw (5,15)--node {\midarrow}(5.5,14.5);
\draw (-0.5,13.5)--node {\midarrow}(0,13);
\draw (0,13)--node {\midarrow}(0.5,12.5);
\draw (0.5,12.5)--node {\midarrow}(1,12);
\draw (1,12)--node {\midarrow}(1.5,11.5);
\draw (1.5,11.5)--node {\midarrow}(2,11);
\draw (2,11)--node {\midarrow}(2.5,11.5);
\draw (2.5,11.5)--node {\midarrow}(3,12);
\draw (3,12)--node {\midarrow}(3.5,12.5);
\draw (3.5,12.5)--node {\midarrow}(4,13);
\draw (4,13)--node {\midarrow}(4.5,13.5);
\draw (4.5,13.5)--node {\midarrow}(5,14);
\draw (5,14)--node {\midarrow}(5.5,14.5);
\draw (-0.5,13.5)--node {\midarrow}(0,14);
\draw (0,14)--node {\midarrow}(0.5,14.5);
\draw (0.5,14.5)--node {\midarrow}(1,15);
\draw (1,15)--node {\midarrow}(1.5,15.5);
\draw (1.5,15.5)--node {\midarrow}(2,16);
\draw (2,16)--node {\midarrow}(2.5,16.5);
\draw (2.5,16.5)--node {\midarrow}(3,17);
\draw (0,13)--node {\midarrow}(0.5,13.5);
\draw (0.5,13.5)--node {\midarrow}(1,14);
\draw (1,14)--node {\midarrow}(1.5,14.5);
\draw (1.5,14.5)--node {\midarrow}(2,15);
\draw (2,15)--node {\midarrow}(2.5,15.5);
\draw (2.5,15.5)--node {\midarrow}(3,16);
\draw (3,16)--node {\midarrow}(3.5,16.5);
\draw (0.5,12.5)--node {\midarrow}(1,13);
\draw (1,13)--node {\midarrow}(1.5,13.5);
\draw (1.5,13.5)--node {\midarrow}(2,14);
\draw (2,14)--node {\midarrow}(2.5,14.5);
\draw (2.5,14.5)--node {\midarrow}(3,15);
\draw (3,15)--node {\midarrow}(3.5,15.5);
\draw (3.5,15.5)--node {\midarrow}(4,16);
\draw (1,12)--node {\midarrow}(1.5,12.5);
\draw (1.5,12.5)--node {\midarrow}(2,13);
\draw (2,13)--node {\midarrow}(2.5,13.5);
\draw (2.5,13.5)--node {\midarrow}(3,14);
\draw (3,14)--node {\midarrow}(3.5,14.5);
\draw (3.5,14.5)--node {\midarrow}(4,15);
\draw (4,15)--node {\midarrow}(4.5,15.5);
\draw (1.5,11.5)--node {\midarrow}(2,12);
\draw (2,12)--node {\midarrow}(2.5,12.5);
\draw (2.5,12.5)--node {\midarrow}(3,13);
\draw (3,13)--node {\midarrow}(3.5,13.5);
\draw (3.5,13.5)--node {\midarrow}(4,14);
\draw (4,14)--node {\midarrow}(4.5,14.5);
\draw (4.5,14.5)--node {\midarrow}(5,14);
\draw (0,14)--node {\midarrow}(0.5,13.5);
\draw (0.5,13.5)--node {\midarrow}(1,13);
\draw (1,13)--node {\midarrow}(1.5,12.5);
\draw (1.5,12.5)--node {\midarrow}(2,12);
\draw (2,12)--node {\midarrow}(2.5,11.5);
\draw (0.5,14.5)--node {\midarrow}(1,14);
\draw (1,14)--node {\midarrow}(1.5,13.5);
\draw (1.5,13.5)--node {\midarrow}(2,13);
\draw (2,13)--node {\midarrow}(2.5,12.5);
\draw (2.5,12.5)--node {\midarrow}(3,12);
\draw (1,15)--node {\midarrow}(1.5,14.5);
\draw (1.5,14.5)--node {\midarrow}(2,14);
\draw (2,14)--node {\midarrow}(2.5,13.5);
\draw (2.5,13.5)--node {\midarrow}(3,13);
\draw (3,13)--node {\midarrow}(3.5,12.5);
\draw (1.5,15.5)--node {\midarrow}(2,15);
\draw (2,15)--node {\midarrow}(2.5,14.5);
\draw (2.5,14.5)--node {\midarrow}(3,14);
\draw (3,14)--node {\midarrow}(3.5,13.5);
\draw (3.5,13.5)--node {\midarrow}(4,13);
\draw (2,16)--node {\midarrow}(2.5,15.5);
\draw (2.5,15.5)--node {\midarrow}(3,15);
\draw (3,15)--node {\midarrow}(3.5,14.5);
\draw (3.5,14.5)--node {\midarrow}(4,14);
\draw (4,14)--node {\midarrow}(4.5,13.5);
\draw (2.5,16.5)--node {\midarrow}(3,16);
\draw (3,16)--node {\midarrow}(3.5,15.5);
\draw (3.5,15.5)--node {\midarrow}(4,15);
\draw (4,15)--node {\midarrow}(4.5,14.5);
\draw (4.5,14.5)--node {\midarrow}(5,15);}
\textcolor{red}{
\draw (-0.45,13.5)--node {\midarrow}(0.05,13)--node {\midarrow}(0.55,13.5);
\draw (0.05,13)--node {\midarrow}(0.55,12.5)--node {\midarrow}(1.05,13);
\draw (0.55,12.5)--node {\midarrow}(1.05,12)--node {\midarrow}(1.55,12.5);
\draw (1.05,12)--node {\midarrow}(1.55,11.5)--node {\midarrow}(2.05,12);
\draw (1.55,11.5)--node {\midarrow}(2.05,11)--node {\midarrow}(2.55,11.5);
\draw (2.05,11)--node {\midarrow}(2.55,10.5)--node {\midarrow}(3.05,11);
\draw (2.55,10.5)--node {\midarrow}(3.05,10)--node {\midarrow}(3.55,10.5);
\draw (3.05,10)--node {\midarrow}(3.55,9.5)--node {\midarrow}(4.05,10);
\draw (3.55,9.5)--node {\midarrow}(4.05,9)--node {\midarrow}(4.55,9.5);
\draw (4.05,9)--node {\midarrow}(4.55,8.5)--node {\midarrow}(5.05,9);
\draw (4.55,8.5)--node {\midarrow}(5.05,8);
\draw (-0.45,13.5)--node {\midarrow}(0.05,14);
\draw (0.05,14)--node {\midarrow}(0.55,13.5);
\draw (0.55,13.5)--node {\midarrow}(1.05,13);
\draw (1.05,13)--node {\midarrow}(1.55,12.5);
\draw (1.55,12.5)--node {\midarrow}(2.05,12);
\draw (2.05,12)--node {\midarrow}(2.55,11.5);
\draw (2.55,11.5)--node {\midarrow}(3.05,11);
\draw (3.05,11)--node {\midarrow}(3.55,10.5);
\draw (3.55,10.5)--node {\midarrow}(4.05,10);
\draw (4.05,10)--node {\midarrow}(4.55,9.5);
\draw (4.55,9.5)--node {\midarrow}(5.05,9);
\draw (5.05,9)--node {\midarrow}(5.55,8.5);
\draw (5.05,8)--node {\midarrow}(5.55,8.5);}
\draw [fill](2.5,11.5) circle (2pt) (0,14) circle (2pt)  (3,17) circle (2pt);
\end{scope}
\node [above] at (0,20.3)   {$1$};
\node [above] at (0.5,20.3) {$2$};
\node [above] at (1,20.3)   {$3$};
\node [above] at (1.5,20.3) {$4$};
\node [above] at (2,20.3)   {$5$};
\node [above] at (2.5,20.3) {$6$};
\node [above] at (3,20.3)   {$7$};
\node [above] at (3.5,20.3) {$8$};
\node [above] at (4,20.3)   {$9$};
\node [above] at (4.5,20.3) {$10$};
\node [above] at (5,20.3)   {$11$};
\node [left] at (-0.8,20)   {$0$};
\node [left] at (-0.8,19.5)   {$1$};
\node [left] at (-0.8,19)     {$2$};
\node [left] at (-0.8,18.5)   {$3$};
\node [left] at (-0.8,18)     {$4$};
\node [left] at (-0.8,17.5)   {$5$};
\node [left] at (-0.8,17)      {$6$};
\node [left] at (-0.8,16.5)   {$7$};
\node [left] at (-0.8,16)     {$8$};
\node [left] at (-0.5,15.5)   {$9(1)$};
\node [left] at (-0.5,15)     {$10(2)$};
\node [left] at (-0.5,14.5)   {$11(3)$};
\node [left] at (-0.5,14)     {$12(4)$};
\node [left] at (-0.5,13.5)   {$13(5)$};
\node [left] at (-0.5,13)     {$14(6)$};
\node [left] at (-0.5,12.5)   {$15(7)$};
\node [left] at (-0.5,12)     {$16(8)$};
\node [left] at (-0.5,11.5)   {$17(1)$};
\node [left] at (-0.5,11)     {$18(2)$};
\node [left] at (-0.5,10.5)   {$19(3)$};
\node [left] at (-0.5,10)     {$20(4)$};
\node [left] at (-0.5,9.5)   {$21(5)$};
\node [left] at (-0.5,9)     {$22(6)$};
\node [left] at (-0.5,8.5)    {$23(7)$};
\node [left] at (-0.5,8)      {$24(8)$};
\node [left] at (-0.5,7.5)    {$25(1)$};
\node [left] at (-0.5,7)      {$26(2)$};
\node [left] at (-0.5,6.5)    {$27(3)$};
\node [left] at (-0.5,6)      {$28(4)$};
\node [left] at (-0.5,5.5)    {$29(5)$};
\node [above] at (3,17) {$(7,6)$};
\node [right] at (2.5,11.5) {$(i,k'_2)$};
\node [above] at (0,14) {$(a_2, \alpha_2)$};
\node at (2.5,5) {$(c)$};
\end{tikzpicture}}
\end{minipage}}
\caption{For the $U_{\varepsilon}^{\res}({L\mathfrak{sl}_{12}})$-module $L(Y_{6,1}Y_{7,6})$, where $\varepsilon^{2 \ell}=1$, $\ell=4$. (a) $\mathscr{P}_{6,1}$ and $\mathscr{P}_{7,6}$. (b) The translation of paths in $\mathscr{P}_{6,1}$ in $(a)$ to paths in $\mathscr{P}_{6,9}$ in (b) with respect to $\mathbf{PS}(7,6)$. We obtain the dominant monomial $Y_{5,8}Y_{8,7}$. Here, we have $(a_1, \alpha_1)=(5,8)$ and $(b_1, \beta_1)=(8,7)$. (c) The translation of paths $\mathscr{P}_{6,1}$ in (a) to paths in $\mathscr{P}_{6,17}$ in (c) with respect to $\mathbf{PS}(7,6)$. We obtain the dominant monomial $Y_{1,12}=Y_{1,4}$. Here, we have $(a_2, \alpha_2)=(1,12)$.}\label{F: tranlation}
\end{figure}

\begin{lemma}\label{Le: II translation of degree 2}
Let $\varepsilon^{2 \ell}=1$ and let $L(Y_{i,k}Y_{j,v})$, $k<v$, be a $U_{\varepsilon}^{\res}({L\mathfrak{sl}_{n+1}})$ snake module. Assume that $h(i,j)\equiv -h_{0}(i,j)-2s\,(\text{mod } 2 \ell)$, where $0\leq s<|\mathbf{PS}^{j}(i,k)|$. Then there is a translation of paths in $\mathscr{P}_{i,k}$ to paths in $\mathscr{P}_{i,k'}$ with respect to $\mathbf{PS}(j,v)$, where $k'\equiv k \, (\text{mod } 2 \ell)$ and $(i,k')\in \mathbf{PS}(j,v)$ such that the following properties hold.  For $i\leq j$ (resp. $i>j$), we denote $m=Y_{i-r,k'-r}Y_{j+r,v+r}$ (resp. $m=Y_{j-r,v+r}Y_{i+r,k'-r}$), where $r=\frac{k'-v-h_0(i,j)}2+1$. Then
\begin{enumerate}
    \item the dominant monomial $m$ is in $S_{(i,k)(j,v)}$;
    \item there exist paths $\widetilde{p}_1 \in \mathscr{P}_{i,k}$, $\widetilde{p}_2\in \mathscr{P}_{j,v}$ such that $\mathfrak{m}(\widetilde{p}_1)\mathfrak{m}(\widetilde{p}_2)$ is the lowest $\ell$-weight monomial of $L(m)$ and $\mathfrak{m}(\widetilde{p}_1)\mathfrak{m}(\widetilde{p}_2)$ is in $S_{(i,k)(j,v)}$;
    \item the monomials of $\chi_\varepsilon(L(m))$ are contained in $S_{(i,k)(j,v)}$.
\end{enumerate}
\end{lemma}

\begin{proof}
Assume that $h(i,j)\equiv -h_{0}(i,j)-2s\,(\text{mod } 2 \ell)$, where $0\leq s<|\mathbf{PS}^{j}(i,k)|$. Then there exists $(i,k')$ such that $k'\equiv k\,(\text{mod } 2 \ell)$ and $(i,k')\in \mathbf{PS}(j,v)$. By Definition \ref{Def: translation}, there is a translation of paths in $\mathscr{P}_{i,k}$ to paths in $\mathscr{P}_{i,k'}$ with respect to $\mathbf{PS}(j,v)$, where $k'=k+2a \ell$ for some $a\in \ZZ_{\geq 0}$. In the following, we prove (1), (2), and (3) for the case of $i\leq j$, the proof for the case of $i>j$ is similar.

(1) Let $p_1$ be the highest path in $\mathscr{P}_{i,k'}$ with no lower corners and $p_2 \in \mathscr{P}_{j,v}$ be the path which has exactly one lower corner $(i, k')$. Since $i\leq j$, then $\mathfrak{m}(p_1)=Y_{i,k'}$, $\mathfrak{m}(p_2)=Y_{i-r,k'-r}Y_{i,k'}^{-1}Y_{j+r,v+r}$, where $r=\frac{k'-v-h_0(i,j)}2+1$. Let $p'_1\in\mathscr{P}_{i,k}$ denote the path that corresponds to $p_1$. Then $\mathfrak{m}(p'_1)=\mathfrak{m}(p_1)$. Clearly the paths $p'_1$ and $p_2$ are non-overlapping, so the dominant monomial $m=Y_{i-r,k'-r}Y_{j+r,v+r}=\mathfrak{m}(p'_1)\mathfrak{m}(p_2)$ is in $S_{(i,k)(j,v)}$.

(2) Following (1), $m=Y_{i-r,k'-r}Y_{j+r,v+r}$. Let $\widetilde{p}'_1 \in\mathscr{P}_{i,k'}$ be the path which has exactly one upper corner $(n+1-j, n+1+v)$, and let $\widetilde{p}_2$ be the lowest path in $\mathscr{P}_{j,v}$ with no upper corners. Let $\widetilde{p}_1$ be the path in $\mathscr{P}_{i,k}$ that corresponds to $\widetilde{p}'_1$. Then $\mathfrak{m}(\widetilde{p}'_1)=\mathfrak{m}(\widetilde{p}_1)$. Since $i\leq j$, we have $\mathfrak{m}(\widetilde{p}_1)=Y_{n+1-j-r,n+1+v+r-2a \ell}^{-1}Y_{n+1-j,n+1+v-2a \ell}Y_{n+1-i+r,n+1+k'-r-2a \ell}^{-1}$, $\mathfrak{m}(\widetilde{p}_2)=Y_{n+1-j,n+1+v}^{-1}$. Clearly the paths $\widetilde{p}_1$ and $\widetilde{p}_2$ are non-overlapping. Therefore, the lowest $l$-weight monomial of the module $L(m)$ is given by
\[
 \mathfrak{m}(\widetilde{p}_1)\mathfrak{m}(\widetilde{p}_2)=Y_{n+1-j-r,n+1+v+r}^{-1}Y_{n+1-i+r,n+1+k'-r}^{-1}, 
 \]
which is in $S_{(i,k)(j,v)}$.

(3) According to (1), using the translation of paths in $\mathscr{P}_{i,k}$ to paths in $\mathscr{P}_{i,k'}$ with respect to $\mathbf{PS}(j,v)$, where $k'=k+2a \ell$ for some $a\in\ZZ_{>0}$, we obtain a dominant monomial $Y_{i-r,k'-r}Y_{j+r,v+r}$. For any path $p$ in $\mathscr{P}_{i-r,k'-r}$, the starting point is $(0,j+v)$ and the ending point is $(n+1,n+1-i+k')$. Similarly, for any path $p'$ in $\mathscr{P}_{j+r,v+r}$, the starting point is $(0,i+k')$ and the ending point is $(n+1,n+1-j+v)$.

Therefore, the paths $p$ and $p'$ must intersect at some points. This implies there are paths from $(0,j+v)$ to $(n+1,n+1-j+v)$ and from $(0,i+k')$ to $(n+1,n+1-i+k')$. We choose $\hat{p}_1\in \mathscr{P}_{j,v}$, which is the path from $(0,j+v)$ to $(n+1,n+1-j+v)$ that lies below all other paths connecting these two points (sometimes, there is only one such path). The remaining path is denoted as $\hat{p}_2\in\mathscr{P}_{i,k'}$, which is the path from $(0,i+k')$ to $(n+1,n+1-i+k')$. Let $\hat{p}'_2\in\mathscr{P}_{i,k}$ denote the path corresponding to $\hat{p}_2$. Then $\mathfrak{m}(\hat{p}'_2)=\mathfrak{m}(\hat{p}_2)$.

Now, we convert $L(Y_{i-r,k'-r}Y_{j+r,v+r})$ into the snake module $L(Y_{i-r,k'-r}Y_{j+r,v+r+2a \ell})$. Let $p''\in\mathscr{P}_{j+r,v+r+2a \ell}$ denote the path corresponding to $p'$. If $p$ and $p''$ intersect at some point $(c_1,d_1)$, where $c_1\in [0,n+1]$ and $d_1\in \ZZ$, then we have $(c_1,d_1)\in \hat{p}_2$ and $(c_1,d_1-2a \ell) \in \hat{p}_1$. Furthermore, we have $(c_1,d_1-2a \ell)\in \hat{p}_2'$. Therefore, $\hat{p}_1$ and $\hat{p}'_2$ intersect at the point $(c_1,d_1-2a\ell)$.

Assume that $\hat{p}_1$, $\hat{p}'_2$ and $p$, $p''$ satisfy the correspondence relations mentioned above, respectively. In the following, we prove that if $p$ and $p''$ do not intersect, then $\hat{p}_1$ and $\hat{p}'_2$ also do not intersect. Suppose that $\hat{p}_1$ and $\hat{p}'_2$ intersect at some point $(c_2,d_2)$, where $c_2\in [0,n+1]$ and $d_2\in \ZZ$. Then we have $(c_2,d_2)\in p'$ and $(c_2,d_2+2a \ell)\in p$. Furthermore, we have $(c_2,d_2+2a \ell)\in p''$. So $p$ and $p''$ intersect at the point $(c_2,d_2+2a \ell)$. It is a contradiction. Therefore, we conclude that if $p$ and $p''$ do not intersect, then $\hat{p}_1$ and $\hat{p}'_2$ also do not intersect. That is, if $\mathfrak{m}(p)\mathfrak{m}(p'')\in S_{(i-r,k'-r)(j+r,v+r+2a \ell)}$, then $\mathfrak{m}(p)\mathfrak{m}(p'')= \mathfrak{m}(\hat{p}_1)\mathfrak{m}(\hat{p}'_2)\in S_{(i,k)(j,v)}$. Furthermore, $S_{(i-r,k'-r)(j+r,v+r+2a \ell)}$ is a part of $S_{(i,k)(j,v)}$.

By Theorem \ref{Th:simple modules are snake modules}, we have $\chi_\varepsilon(L(Y_{i-r,k'-r}Y_{j+r,v+r}))=\chi_\varepsilon(L(Y_{i-r,k'-r}Y_{j+r,v+r+2a \ell}))$. We know that $\chi_\varepsilon(L(Y_{i-r,k'-r}Y_{j+r,v+r+2a \ell}))$ is a part of $S_{(i-r,k'-r)(j+r,v+r+2a \ell)}$. Therefore, it follows that $\chi_\varepsilon(L(Y_{i-r,k'-r}Y_{j+r,v+r}))$ is also a part of $S_{(i-r,k'-r)(j+r,v+r+2a \ell)}$. Thus, we conclude that the assertion is valid.
\end{proof}

In Lemma \ref{Le: II translation of degree 2}, if $h(i,j)\equiv -h_{0}(i,j)-2s\,(\text{mod } 2 \ell)$, where $0\leq s<|\mathbf{PS}^{j}(i,k)|$, then there are translations of paths in $\mathscr{P}_{i,k}$ to paths in
\begin{align} \label{eq:def of rij}
\mathscr{P}_{i,k'_1}, \mathscr{P}_{i,k'_2}, \ldots, \mathscr{P}_{i,k'_{r(i,j)}}
\end{align}
with respect to $\mathbf{PS}(j,v)$, respectively, for some $r(i,j) \in \ZZ_{\ge 0}$, $k'_1<k'_2<\cdots<k'_{r(i,j)} \in \ZZ$. The number of all such translations is $r(i,j)$. We use the convention that $r(i,j)=0$ if there is no such translation. We denote the dominant monomials obtained by translating paths in $\mathscr{P}_{i,k}$ to paths in $\mathscr{P}_{i,k'_1}$, $\mathscr{P}_{i,k'_2}$, $\ldots$, $\mathscr{P}_{i,k'_{r(i,j)}}$ with respect to $\mathbf{PS}(j,v)$ as
\begin{align}\label{Eq:a_tb_t}
Y_{a_1,\alpha_1}Y_{b_1,\beta_1}, Y_{a_2,\alpha_2}Y_{b_2,\beta_2}, \ldots, Y_{a_{r(i,j)},\alpha_{r(i,j)}}Y_{b_{r(i,j)},\beta_{r(i,j)}},
\end{align}
respectively, where $a_{t}<b_{t}\in \hat{I}\cup \{n+1\}$, $\alpha_t, \beta_t \in \mathbb{Z}$, $1\leq t\leq r(i,j)$, see Example \ref{example:type ii translation of paths} and Figure \ref{F: tranlation}.

\begin{example} \label{example:type ii translation of paths}
Let $\varepsilon^{2 \ell}=1$ with $\ell=4$ and let $L(Y_{6,1}Y_{7,6})$ be a $U_{\varepsilon}^{\res}({L\mathfrak{sl}_{12}})$-module. There are two points $(6,9)$, $(6,17) \in {\bf PS}(7,6)$ and $9\equiv 17\equiv 1 \, (\text{mod } 8)$. That is, there are translations of paths in $\mathscr{P}_{6,1}$ to paths in $\mathscr{P}_{6,9}$ or to paths in $\mathscr{P}_{6,17}$ with respect to $\mathbf{PS}(7,6)$. Translating paths in $\mathscr{P}_{6,1}$ to paths in $\mathscr{P}_{6,9}$, see Figure \ref{F: tranlation} $(b)$, we obtain the dominant monomial $Y_{5,8}Y_{8,7}$, which is the product of $Y_{5,8}Y_{6,9}^{-1}Y_{8,7}$ and $Y_{6,9}$. The monomials of $\chi_\varepsilon(L(Y_{5,8}Y_{8,7}))$ are contained in $S_{(6,1)(7,6)}$. In addition, translating paths in $\mathscr{P}_{6,1}$ to paths in $\mathscr{P}_{6,17}$, see Figure \ref{F: tranlation} $(c)$, we obtain the dominant monomial $Y_{1,12}=Y_{1,4}$, which is the product of $Y_{1,12}Y_{6,17}^{-1}$ and $Y_{6,17}$. The monomials of $\chi_\varepsilon(L(Y_{1,4}))$ are contained in $S_{(6,1)(7,6)}$.
\end{example}

\begin{lemma}\label{Le: two translations}
Let $\varepsilon^{2 \ell}=1$ and let $L(Y_{i,k}Y_{j,v})$, $i\leq j$, $k<v$, be a $U_{\varepsilon}^{\res}({L\mathfrak{sl}_{n+1}})$ snake module. Assume that $Y_{a,\alpha}Y_{b,\beta}$, $a< b \in I$, $\alpha, \beta \in \mathbb{Z}$, is a dominant monomial which can be obtained by a translation of paths in $\mathscr{P}_{i,k}$ to paths in $\mathscr{P}_{i,k'}$ with respect to $\mathbf{PS}(j,v)$ and set $\gamma=\frac{(|j-i|+h(i,j))\,(\text{mod } 2 \ell)}2$. For $|j-i|+h(i,j)>2 \ell$, if $\gamma>0$, $a-\gamma$, $b+\gamma$ are in $\hat{I}\cup\{n+1\}$, then there is a translation of paths in $\mathscr{P}_{a,\alpha}$ to paths in $\mathscr{P}_{a, \overline{\alpha}}$ with respect to $\mathbf{PS}(b,\beta)$, where $\overline{\alpha}\equiv \alpha \, (\text{mod } 2 \ell)$ and $\beta+h_0(a,b)\leq \overline{\alpha}< \beta+h_0(a,b)+2 \ell$, such that the following properties hold:
\begin{enumerate}
    \item the dominant monomial $m=Y_{a-\gamma,\overline{\alpha}-\gamma}Y_{b+\gamma,\beta+\gamma}$ is in $S_{(i,k)(j,v)}$;
    \item the monomials of $\chi_\varepsilon(L(m))$ are contained in $S_{(i,k)(j,v)}$.
\end{enumerate}
\end{lemma}

\begin{proof}
Suppose that $|j-i|+h(i,j)>2 \ell$. If $i\leq j$, $\gamma>0$, and $a-\gamma$, $b+\gamma$ are in $\hat{I}\cup\{n+1\}$, then there exists a point $(a,\overline{\alpha})$ which satisfies $\beta+h_0(a,b)\leq \overline{\alpha}< \beta+h_0(a,b)+2 \ell$ and $\alpha \equiv \overline{\alpha} (\text{mod } 2 \ell)$. By Definition \ref{Def: translation}, there is a translation of paths in $\mathscr{P}_{a,\alpha}$ to paths in $\mathscr{P}_{a,\overline{\alpha}}$ with respect to $\mathbf{PS}(b,\beta)$.

(1) Let $p'_1$ be the highest path in $\mathscr{P}_{a,\overline{\alpha}}$ with no lower corners and $p'_2 \in \mathscr{P}_{b,\beta}$ be the path which has exactly one lower corner $(a, \overline{\alpha})$. That is, $\mathfrak{m}(p'_1)=Y_{a,\overline{\alpha}}$, $\mathfrak{m}(p'_2)=Y_{a-\gamma,\overline{\alpha}-\gamma}Y_{a,\overline{\alpha}}^{-1}Y_{b+\gamma,\beta+\gamma}$. Then we obtain a dominant monomial $m=Y_{a-\gamma,\overline{\alpha}-\gamma}Y_{b+\gamma,\beta+\gamma}=\mathfrak{m}(p'_1)\mathfrak{m}(p'_2)$.

Let $p''_1 \in \mathscr{P}_{a,{\alpha}}$ be the path that corresponds to $p'_1$. Then $\mathfrak{m}(p''_1)=\mathfrak{m}(p'_1)$. By assumption, the dominant monomial $Y_{a,\alpha}Y_{b,\beta}$ can be obtained by a translation of paths in $\mathscr{P}_{i,k}$ to paths in $\mathscr{P}_{i,k'}$ with respect to $\mathbf{PS}(j,v)$, where $k'=k+2c \ell$ for some $c\in\ZZ_{\geq 1}$.Thus, the starting point of the path $p''_1$ is $(0,j+v)$, the endpoint is $(n+1,n+1-i+k')$. In addition, the starting point of the path $p'_2$ is $(0,i+k')$, and the endpoint is $(n+1,n+1-j+v)$. Furthermore, the paths $p''_1$ and $p'_2$ intersect at the point $(i+\gamma,k'+\gamma)$. We choose $p_1\in \mathscr{P}_{i,k'}$ to be the path from $(0,i+k')$ to $(n+1,n+1-i+k')$. That is, $\mathfrak{m}(p_1)=Y_{a-\gamma,\overline{\alpha}-\gamma}Y_{a,\overline{\alpha}}^{-1}Y_{i+\gamma,k'+\gamma}$. The remaining path is denoted as $p_2\in \mathscr{P}_{j,v}$, which is the path from $(0,j+v)$ to $(n+1,n+1-j+v)$. That is, $\mathfrak{m}(p_2)=Y_{a,\alpha}Y_{i+\gamma,k'+\gamma}^{-1}Y_{b+\gamma,\beta+\gamma}$. Therefore, $\mathfrak{m}(p_1)\mathfrak{m}(p_2)=\mathfrak{m}(p''_1)\mathfrak{m}(p'_2)$. Let $\hat{p}_1\in \mathscr{P}_{i,k}$ be the path that corresponds to $p_1$. Then $\mathfrak{m}(\hat{p}_1)=\mathfrak{m}(p_1)=Y_{a-\gamma,\overline{\alpha}-\gamma-2c \ell }Y_{a,\overline{\alpha}-2c \ell }^{-1}Y_{i+\gamma,k'+\gamma-2c \ell}$.
Since $|j-i|+h(i,j)>2 \ell$, we have that $\overline{\alpha}-2c \ell<\alpha$. Therefore, the paths $\hat{p}_1$ and $p_2$ are non-overlapping. Hence, the dominant monomial
$\mathfrak{m}(\hat{p}_1)\mathfrak{m}(p_2)=\mathfrak{m}(p'_1)\mathfrak{m}(p'_2)=Y_{a-\gamma,\overline{\alpha}-\gamma}Y_{b+\gamma,\beta+\gamma}$
is in $S_{(i,k)(j,v)}$.

(2) Using the same method as in (2) and (3) of Lemma \ref{Le: II translation of degree 2}, we can obtain that the lowest $l$-weight monomial of $L(m)$ is in $S_{(i,k)(j,v)}$. Furthermore, the monomials of the $\varepsilon$-character of the module $L(m)$ are contained in $S_{(i,k)(j,v)}$.
\end{proof}

\begin{example}\label{Ex:two translations}
Let $\varepsilon^{2 \ell}=1$ with $\ell=2$ and let $L(Y_{3,0}Y_{6,7})$ be a $U_{\varepsilon}^{\res}({L\mathfrak{sl}_{8}})$-module.
There is a point $(3,12) \in {\bf PS}(6,7)$ and $12 \equiv 0 \, (\text{mod } 4)$. That is, there is a translation of paths in $\mathscr{P}_{3,0}$ to paths in $\mathscr{P}_{3,12}$ with respect to ${\bf PS}(6,7)$, see Figure \ref{F:  two translations} $(b)$. We obtain the dominant monomial $Y_{2,11}Y_{7,8}$, which is the product of $Y_{2,11}Y_{3,12}^{-1}Y_{7,8}$ and $Y_{3,12}$. The monomials of $\chi_\varepsilon(L(Y_{2,11}Y_{7,8}))$ are contained in $S_{(3,0)(6,7)}$.

In addition, there exists $(2,15)\in {\bf PS}(7,8)$ and $15\equiv 11 \, (\text{mod } 4)$. That is, there is a translation of paths in $\mathscr{P}_{2,11}$ to paths in $\mathscr{P}_{2,15}$ with respect to ${\bf PS}(7,8)$. We obtain the dominant monomial $Y_{1,14}=Y_{1,2}$, which is the product of $Y_{1,14}Y_{2,15}^{-1}$ and $Y_{2,15}$, see Figure \ref{F:  two translations} $(c)$. The monomials of $\chi_\varepsilon(L(Y_{1,2}))$ are contained in $S_{(3,0)(6,7)}$.
\end{example}

\begin{figure}
\resizebox{0.6\width}{0.6\height}{
\begin{minipage}[b]{0.4\linewidth}
\centerline{
\begin{tikzpicture}
\draw[step=.5cm,gray,thin] (-0.5,5.5) grid (3.5,17) (-0.5,5.5)--(3.5,5.5);
\draw[fill] (0,17.3) circle (2pt) -- (0.5,17.3) circle (2pt) --(1,17.3) circle (2pt) --(1.5,17.3) circle (2pt)--(2,17.3) circle (2pt) -- (2.5,17.3) circle (2pt) --(3,17.3) circle (2pt);
\begin{scope}[thick, every node/.style={sloped,allow upside down}]
\draw (-0.5,15.5)--node {\midarrow}(0,16)--node {\midarrow}(0.5,16.5)--node {\midarrow}(1,17)--node {\midarrow}(1.5,16.5)--node {\midarrow}(2,16)--node {\midarrow}(2.5,15.5)--node {\midarrow}(3,15)--node {\midarrow}(3.5,14.5);
\draw (-0.5,15.5)--node {\midarrow}(0,15)--node {\midarrow}(0.5,14.5)--node {\midarrow}(1,14)--node {\midarrow}(1.5,13.5)--node {\midarrow}(2,13)--node {\midarrow}(2.5,13.5)--node {\midarrow}(3,14)--node {\midarrow}(3.5,14.5);
\draw (0.5,16.5)--node {\midarrow}(1,16)--node {\midarrow}(1.5,15.5)--node {\midarrow}(2,15)--node {\midarrow}(2.5,14.5)--node {\midarrow}(3,14);
\draw (0,16)--node {\midarrow}(0.5,15.5)--node {\midarrow}(1,15)--node {\midarrow}(1.5,14.5)--node {\midarrow}(2,14)--node {\midarrow}(2.5,13.5);
\draw (0,15)--node {\midarrow}(0.5,15.5)--node {\midarrow}(1,16)--node {\midarrow}(1.5,16.5);
\draw (0.5,14.5)--node {\midarrow}(1,15)--node {\midarrow}(1.5,15.5)--node {\midarrow}(2,16);
\draw (1,14)--node {\midarrow}(1.5,14.5)--node {\midarrow}(2,15)--node {\midarrow}(2.5,15.5);
\draw (1.5,13.5)--node {\midarrow}(2,14)--node {\midarrow}(2.5,14.5)--node {\midarrow}(3,15);
\draw[orange](-0.5,10.5)--node {\midarrow}(0,11)--node {\midarrow}(0.5,11.5)--node {\midarrow}(1,12)--node {\midarrow}(1.5,12.5)--node {\midarrow}(2,13)--node {\midarrow}(2.5,13.5)--node {\midarrow}(3,13)--node {\midarrow}(3.5,12.5);
\draw[orange](-0.5,10.5)--node {\midarrow}(0,10)--node {\midarrow}(0.5,9.5)--node {\midarrow}(1,10)--node {\midarrow}(1.5,10.5)--node {\midarrow}(2,11)--node {\midarrow}(2.5,11.5)--node {\midarrow}(3,12)--node {\midarrow}(3.5,12.5);
\draw[orange](0,10)--node {\midarrow}(0.5,10.5)--node {\midarrow}(1,11)--node {\midarrow}(1.5,11.5)--node {\midarrow}(2,12)--node {\midarrow}(2.5,12.5)--node {\midarrow}(3,13);
\draw[orange](0,11)--node {\midarrow}(0.5,10.5)--node {\midarrow}(1,10);
\draw[orange](0.5,11.5)--node {\midarrow}(1,11)--node {\midarrow}(1.5,10.5);
\draw[orange](1,12)--node {\midarrow}(1.5,11.5)--node {\midarrow}(2,11);
\draw[orange](1.5,12.5)--node {\midarrow}(2,12)--node {\midarrow}(2.5,11.5);
\draw[orange](2,13)--node {\midarrow}(2.5,12.5)--node {\midarrow}(3,12);
\draw[thick,red](-0.5,15.55)--node {\midarrow}(0,16.05)--node {\midarrow}(0.5,15.55)--node {\midarrow}(1,15.05)--node {\midarrow}(1.5,14.55)--node {\midarrow}(2,15.05)--node {\midarrow}(2.5,15.55)--node {\midarrow}(3,15.05)--node {\midarrow}(3.5,14.55);
\draw[thick,red](-0.5,10.55)--node {\midarrow}(0,11.05)--node {\midarrow}(0.5,11.55)--node {\midarrow}(1,12.05)--node {\midarrow}(1.5,12.55)--node {\midarrow}(2,12.05)--node {\midarrow}(2.5,11.55)--node {\midarrow}(3,12.05)--node {\midarrow}(3.5,12.55);
\draw[fill](1,17) circle (1.5pt)  (2.5,13.5) circle (1.5pt);
\end{scope}
\node [above] at (0,17.3)   {$1$};
\node [above] at (0.5,17.3) {$2$};
\node [above] at (1,17.3)   {$3$};
\node [above] at (1.5,17.3) {$4$};
\node [above] at (2,17.3)   {$5$};
\node [above] at (2.5,17.3) {$6$};
\node [above] at (3,17.3)   {$7$};
\node [left] at (-0.8,17)   {$0$};
\node [left] at (-0.8,16.5) {$1$};
\node [left] at (-0.8,16)   {$2$};
\node [left] at (-0.8,15.5) {$3$};
\node [left] at (-0.8,15)   {$4$};
\node [left] at (-0.5,14.5) {$5(1)$};
\node [left] at (-0.5,14)   {$6(2)$};
\node [left] at (-0.5,13.5) {$7(3)$};
\node [left] at (-0.5,13)   {$8(4)$};
\node [left] at (-0.5,12.5) {$9(1)$};
\node [left] at (-0.5,12)   {$10(2)$};
\node [left] at (-0.5,11.5) {$11(3)$};
\node [left] at (-0.5,11)   {$12(4)$};
\node [left] at (-0.5,10.5) {$13(1)$};
\node [left] at (-0.5,10)   {$14(2)$};
\node [left] at (-0.5,9.5) {$15(3)$};
\node [left] at (-0.5,9)   {$16(4)$};
\node [left] at (-0.5,8.5) {$17(1)$};
\node [left] at (-0.5,8)   {$18(2)$};
\node [left] at (-0.5,7.5) {$19(3)$};
\node [left] at (-0.5,7)   {$20(4)$};
\node [left] at (-0.5,6.5)  {$21(1)$};
\node [left] at (-0.5,6)   {$22(2)$};
\node [left] at (-0.5,5.5)  {$23(3)$};
\node at (1.5,5) {$(a)$};
\node [right] at (1,17) {$(3,0)$};
\node [right] at (2.5,13.5) {$(6,7)$};
\end{tikzpicture}}
\end{minipage}
\begin{minipage}[b]{0.4\linewidth}
\centerline{
\begin{tikzpicture}
\draw[step=.5cm,gray,thin] (-0.5,5.5) grid (3.5,17) (-0.5,5.5)--(3.5,5.5);
\draw[fill] (0,17.3) circle (2pt) -- (0.5,17.3) circle (2pt) --(1,17.3) circle (2pt) --(1.5,17.3) circle (2pt)--(2,17.3) circle (2pt) -- (2.5,17.3) circle (2pt) --(3,17.3) circle (2pt);
\begin{scope}[thick, every node/.style={sloped,allow upside down}]
\draw (-0.5,9.45)--node {\midarrow}(0,9.95)--node {\midarrow}(0.5,10.45)--node {\midarrow}(1,10.95)--node {\midarrow}(1.5,10.45)--node {\midarrow}(2,10)--node {\midarrow}(2.5,9.5)--node {\midarrow}(3,9)--node {\midarrow}(3.5,8.5);
\draw (-0.5,9.5)--node {\midarrow}(0,9)--node {\midarrow}(0.5,8.5)--node {\midarrow}(1,8)--node {\midarrow}(1.5,7.5)--node {\midarrow}(2,7)--node {\midarrow}(2.5,7.5)--node {\midarrow}(3,8)--node {\midarrow}(3.5,8.5);
\draw (0.5,10.45)--node {\midarrow}(1,9.95)--node {\midarrow}(1.5,9.45)--node {\midarrow}(2,9)--node {\midarrow}(2.5,8.5)--node {\midarrow}(3,8);
\draw (0,9.95)--node {\midarrow}(0.5,9.45)--node {\midarrow}(1,9)--node {\midarrow}(1.5,8.5)--node {\midarrow}(2,8)--node {\midarrow}(2.5,7.5);
\draw (0,9)--node {\midarrow}(0.5,9.5)--node {\midarrow}(1,10)--node {\midarrow}(1.5,10.5);
\draw (0.5,8.5)--node {\midarrow}(1,9)--node {\midarrow}(1.5,9.5)--node {\midarrow}(2,10);
\draw (1,8)--node {\midarrow}(1.5,8.5)--node {\midarrow}(2,9)--node {\midarrow}(2.5,9.5);
\draw (1.5,7.5)--node {\midarrow}(2,8)--node {\midarrow}(2.5,8.5)--node {\midarrow}(3,9);
\draw[orange](-0.5,10.5)--node {\midarrow}(0,11)--node {\midarrow}(0.5,11.5)--node {\midarrow}(1,12)--node {\midarrow}(1.5,12.5)--node {\midarrow}(2,13)--node {\midarrow}(2.5,13.5)--node {\midarrow}(3,13)--node {\midarrow}(3.5,12.5);
\draw[orange](-0.5,10.5)--node {\midarrow}(0,10)--node {\midarrow}(0.5,9.5)--node {\midarrow}(1,10)--node {\midarrow}(1.5,10.5)--node {\midarrow}(2,11)--node {\midarrow}(2.5,11.5)--node {\midarrow}(3,12)--node {\midarrow}(3.5,12.5);
\draw[orange](0,10)--node {\midarrow}(0.5,10.5)--node {\midarrow}(1,11)--node {\midarrow}(1.5,11.5)--node {\midarrow}(2,12)--node {\midarrow}(2.5,12.5)--node {\midarrow}(3,13);
\draw[orange](0,11)--node {\midarrow}(0.5,10.5)--node {\midarrow}(1,10);
\draw[orange](0.5,11.5)--node {\midarrow}(1,11)--node {\midarrow}(1.5,10.5);
\draw[orange](1,12)--node {\midarrow}(1.5,11.5)--node {\midarrow}(2,11);
\draw[orange](1.5,12.5)--node {\midarrow}(2,12)--node {\midarrow}(2.5,11.5);
\draw[orange](2,13)--node {\midarrow}(2.5,12.5)--node {\midarrow}(3,12);
\draw[red](-0.5,10.55)--node {\midarrow}(0,11.05)--node {\midarrow}(0.5,11.55)--node {\midarrow}(1,11.05)--node {\midarrow}(1.5,11.55)--node {\midarrow}(2,12.05)--node {\midarrow}(2.5,12.55)--node {\midarrow}(3,13.05)--node {\midarrow}(3.5,12.55);
\draw[red](-0.5,10.55)--node {\midarrow}(0,10.05);
\draw[red](-0.5,9.55)--node {\midarrow}(0,10.05);
\draw[red] (-0.5,9.55)--node {\midarrow}(0,9.05)--node {\midarrow}(0.5,9.55)--node {\midarrow}(1,9.05)--node {\midarrow}(1.5,8.55)--node {\midarrow}(2,8.05)--node {\midarrow}(2.5,7.55)--node {\midarrow}(3,8.05)--node {\midarrow}(3.5,8.55);
\draw[red] (1.5,10.55)--node {\midarrow}(2,11.05)--node {\midarrow}(2.5,11.55)--node {\midarrow}(3,12.05)--node {\midarrow}(3.5,12.55);
\draw[red] (1.5,10.55)--node {\midarrow}(2,10.05)--node {\midarrow}(2.5,9.55)--node {\midarrow}(3,9.05)--node {\midarrow}(3.5,8.55);
\draw[fill] (0.5,11.5) circle (2pt) (3,13) circle (2pt) (1,11) circle (1.5pt);
\end{scope}
\node [above] at (0,17.3)   {$1$};
\node [above] at (0.5,17.3) {$2$};
\node [above] at (1,17.3)   {$3$};
\node [above] at (1.5,17.3) {$4$};
\node [above] at (2,17.3)   {$5$};
\node [above] at (2.5,17.3) {$6$};
\node [above] at (3,17.3)   {$7$};
\node [left] at (-0.8,17)   {$0$};
\node [left] at (-0.8,16.5) {$1$};
\node [left] at (-0.8,16)   {$2$};
\node [left] at (-0.8,15.5) {$3$};
\node [left] at (-0.8,15)   {$4$};
\node [left] at (-0.5,14.5) {$5(1)$};
\node [left] at (-0.5,14)   {$6(2)$};
\node [left] at (-0.5,13.5) {$7(3)$};
\node [left] at (-0.5,13)   {$8(4)$};
\node [left] at (-0.5,12.5) {$9(1)$};
\node [left] at (-0.5,12)   {$10(2)$};
\node [left] at (-0.5,11.5) {$11(3)$};
\node [left] at (-0.5,11)   {$12(4)$};
\node [left] at (-0.5,10.5) {$13(1)$};
\node [left] at (-0.5,10)   {$14(2)$};
\node [left] at (-0.5,9.5) {$15(3)$};
\node [left] at (-0.5,9)   {$16(4)$};
\node [left] at (-0.5,8.5) {$17(1)$};
\node [left] at (-0.5,8)   {$18(2)$};
\node [left] at (-0.5,7.5) {$19(3)$};
\node [left] at (-0.5,7)   {$20(4)$};
\node [left] at (-0.5,6.5)  {$21(1)$};
\node [left] at (-0.5,6)   {$22(2)$};
\node [left] at (-0.5,5.5)  {$23(3)$};
\node [above] at (0.5,11.5)  {$(2,11)$};
\node [right] at (3,13)  {$(7,8)$};
\node [right] at (1,11)  {$(3,12)$};
\node at (1.5,5) {$(b)$};
\end{tikzpicture}}
\end{minipage}
\begin{minipage}[b]{0.4\linewidth}
\centerline{
\begin{tikzpicture}
\draw[step=.5cm,gray,thin] (-0.5,5.5) grid (3.5,17) (-0.5,5.5)--(3.5,5.5);
\draw[fill] (0,17.3) circle (2pt) -- (0.5,17.3) circle (2pt) --(1,17.3) circle (2pt) --(1.5,17.3) circle (2pt)--(2,17.3) circle (2pt) -- (2.5,17.3) circle (2pt) --(3,17.3) circle (2pt);
\begin{scope}[thick, every node/.style={sloped,allow upside down}]
\draw[red] (-0.5,9.5)--node {\midarrow}(0,10)--node {\midarrow}(0.5,10.5)--node {\midarrow}(1,11)--node {\midarrow}(1.5,11.5)--node {\midarrow}(2,12)--node {\midarrow}(2.5,12.5)--node {\midarrow}(3,13)--node {\midarrow}(3.5,12.5);
\draw[red] (-0.5,9.5)--node {\midarrow}(0,9)--node {\midarrow}(0.5,9.5)--node {\midarrow}(1,10)--node {\midarrow}(1.5,10.5)--node {\midarrow}(2,11)--node {\midarrow}(2.5,11.5)--node {\midarrow}(3,12)--node {\midarrow}(3.5,12.5);
\draw[red](0,10)--node {\midarrow}(0.5,9.5);
\draw[red](0.5,10.5)--node {\midarrow}(1,10);
\draw[red](1,11)--node {\midarrow}(1.5,10.5);
\draw[red](1.5,11.5)--node {\midarrow}(2,11);
\draw[red](2,12)--node {\midarrow}(2.5,11.5);
\draw[red](2.5,12.5)--node {\midarrow}(3,12);
\draw[blue] (-0.5,8.5)--node {\midarrow}(0,9)--node {\midarrow}(0.5,9.5)--node {\midarrow}(1,9)--node {\midarrow}(1.5,8.5)--node {\midarrow}(2,8)--node {\midarrow}(2.5,7.5)--node {\midarrow}(3,7)--node {\midarrow}(3.5,6.5);
\draw[blue] (-0.5,8.5)--node {\midarrow}(0,8)--node {\midarrow}(0.5,7.5)--node {\midarrow}(1,7)--node {\midarrow}(1.5,6.5)--node {\midarrow}(2,6)--node {\midarrow}(2.5,5.5)--node {\midarrow}(3,6)--node {\midarrow}(3.5,6.5);
\draw[blue] (0,9)--node {\midarrow}(0.5,8.5)--node {\midarrow}(1,8)--node {\midarrow}(1.5,7.5)--node {\midarrow}(2,7)--node {\midarrow}(2.5,6.5)--node {\midarrow}(3,6);
\draw[blue](0,8)--node {\midarrow}(0.5,8.5)--node {\midarrow}(1,9);
\draw[blue](0.5,7.5)--node {\midarrow}(1,8)--node {\midarrow}(1.5,8.5);
\draw[blue](1,7)--node {\midarrow}(1.5,7.5)--node {\midarrow}(2,8);
\draw[blue](1.5,6.5)--node {\midarrow}(2,7)--node {\midarrow}(2.5,7.5);
\draw[blue](2,6)--node {\midarrow}(2.5,6.5)--node {\midarrow}(3,7);
\draw[green](-0.5,9.55)--node {\midarrow}(0,10.05)--node {\midarrow}(0.5,9.55)--node {\midarrow}(1,9.05)--node {\midarrow}(1.5,8.55)--node {\midarrow}(2,8.05)--node {\midarrow}(2.5,7.55)--node {\midarrow}(3,7.05)--node {\midarrow}(3.5,6.55);
\draw[green](-0.5,9.55)--node {\midarrow}(0,9.05)--node {\midarrow}(0.5,8.55)--node {\midarrow}(1,8.05)--node {\midarrow}(1.5,7.55)--node {\midarrow}(2,7.05)--node {\midarrow}(2.5,6.55)--node {\midarrow}(3,6.05)--node {\midarrow}(3.5,6.55);
\draw[fill] (0.5,9.5) circle (2pt) (0,10) circle (2pt);
\end{scope}
\node [above] at (0,17.3)   {$1$};
\node [above] at (0.5,17.3) {$2$};
\node [above] at (1,17.3)   {$3$};
\node [above] at (1.5,17.3) {$4$};
\node [above] at (2,17.3)   {$5$};
\node [above] at (2.5,17.3) {$6$};
\node [above] at (3,17.3)   {$7$};
\node [left] at (-0.8,17)   {$0$};
\node [left] at (-0.8,16.5) {$1$};
\node [left] at (-0.8,16)   {$2$};
\node [left] at (-0.8,15.5) {$3$};
\node [left] at (-0.8,15)   {$4$};
\node [left] at (-0.5,14.5) {$5(1)$};
\node [left] at (-0.5,14)   {$6(2)$};
\node [left] at (-0.5,13.5) {$7(3)$};
\node [left] at (-0.5,13)   {$8(4)$};
\node [left] at (-0.5,12.5) {$9(1)$};
\node [left] at (-0.5,12)   {$10(2)$};
\node [left] at (-0.5,11.5) {$11(3)$};
\node [left] at (-0.5,11)   {$12(4)$};
\node [left] at (-0.5,10.5) {$13(1)$};
\node [left] at (-0.5,10)   {$14(2)$};
\node [left] at (-0.5,9.5) {$15(3)$};
\node [left] at (-0.5,9)   {$16(4)$};
\node [left] at (-0.5,8.5) {$17(1)$};
\node [left] at (-0.5,8)   {$18(2)$};
\node [left] at (-0.5,7.5) {$19(3)$};
\node [left] at (-0.5,7)   {$20(4)$};
\node [left] at (-0.5,6.5)  {$21(1)$};
\node [left] at (-0.5,6)   {$22(2)$};
\node [left] at (-0.5,5.5)  {$23(3)$};
\node [right] at (0.5,9.5)  {$(2,15)$};
\node [above] at (0,10)  {$(1,14)$};
\node at (1.5,5) {$(c)$};
\end{tikzpicture}}
\end{minipage}}
\caption{For the $U_{\varepsilon}^{\res}({L\mathfrak{sl}_{8}})$-module $L(Y_{3,0}Y_{6,7})$, where $\varepsilon^{2 \ell}=1$, $\ell=2$. (a) $\mathscr{P}_{3,0}$ and $\mathscr{P}_{6,7}$. (b) The translation of paths in $\mathscr{P}_{3,0}$ in (a) to paths in $\mathscr{P}_{3,12}$ in (b) with respect to ${\bf PS}(6,7)$. We obtain the dominant monomial $Y_{2,11}Y_{7,8}$. (c) The translation of paths in $\mathscr{P}_{2,11}$ in $(b)$ to paths in $\mathscr{P}_{2,15}$ in $(c)$ with respect to ${\bf PS}(7,8)$. We obtain the dominant monomial $Y_{1,14}=Y_{1,2}$.}\label{F:  two translations}
\end{figure}
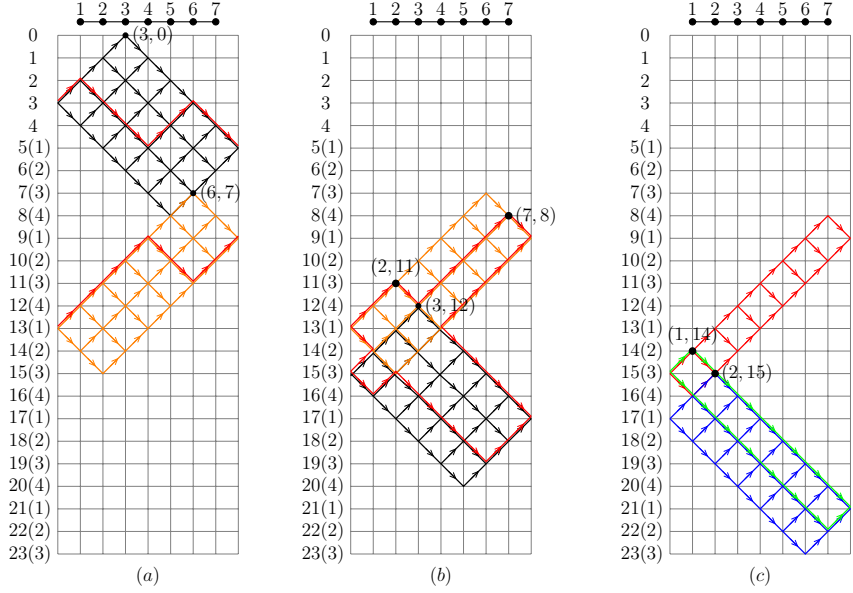

Dually, we have the following lemma.

\begin{lemma}\label{Le: dual two translations}
Let $\varepsilon^{2 \ell}=1$ and let $L(Y_{i,k}Y_{j,v})$, $i>j$, $k<v$, be a $U_{\varepsilon}^{\res}({L\mathfrak{sl}_{n+1}})$ snake module. Assume that $Y_{a,\alpha}Y_{b,\beta}$, $a< b \in I$, $\alpha, \beta \in \mathbb{Z}$, is a dominant monomial which can be obtained by a translation of paths in $\mathscr{P}_{i,k}$ to paths in $\mathscr{P}_{i,k'}$ with respect to ${\bf PS}(j,v)$ and set $\gamma=\frac{(|j-i|+h(i,j))\,(\text{mod } 2 \ell)}2$. For $|j-i|+h(i,j)>2 \ell$, if $\gamma>0$, and $a-\gamma$, $b+\gamma\in \hat{I}\cup\{n+1\}$, then there is a translation of paths in $\mathscr{P}_{b,\beta}$ to paths in $\mathscr{P}_{b,\overline{\beta}}$ with respect to ${\bf PS}(a,\alpha)$, where $\overline{\beta}\equiv \beta\, (\text{mod } 2 \ell)$ and $\alpha+h_0(a,b)\leq \overline{\beta}< \alpha+h_0(a,b)+2 \ell$, such that the following properties hold:
\begin{enumerate}
\item the dominant monomial $m=Y_{a-\gamma,\alpha+\gamma}Y_{b+\gamma,\overline{\beta}-\gamma}$ is in $S_{(i,k)(j,v)}$;
\item The monomials of $\chi_\varepsilon(L(m))$ are contained in $S_{(i,k)(j,v)}$.
\end{enumerate}
\end{lemma}

\section{\texorpdfstring{A path description for $\varepsilon$-characters of degree two modules and irreducibility of tensor products of fundamental modules}{A path description for epsilon-characters of degree two modules and irreducibility of tensor products of fundamental modules}}\label{section of path description of degree two}

In this section, using path translations or raising and lowering moves, we obtain an effective and concrete path description for the $\varepsilon$-character of any simple $U_{\varepsilon}^{\res}({L\mathfrak{sl}_{n+1}})$-module with highest $l$-weight of degree two. As an application of our path description, we obtain a necessary and sufficient condition for the tensor product of two fundamental representations for $U_{\varepsilon}^{\res}({L\mathfrak{sl}_{n+1}})$ to be irreducible, where $\varepsilon^{2 \ell}=1$, $\ell \geq 2$. Subsequently, we obtain a necessary condition for the tensor product of fundamental representations for $U_{\varepsilon}^{\res}({L\mathfrak{sl}_{n+1}})$ to be irreducible, where $\varepsilon^{2 \ell}=1$, $\ell \geq 2$.

\subsection{\texorpdfstring{A path description for the $\varepsilon$-character of any simple module with highest $l$-weight of degree two}{A path description for the epsilon-character of any simple module with highest l-weight of degree two}}\label{subsection of path description of degree two}
\begin{definition} \label{def:small values of indices}
Let $\varepsilon^{2 \ell}=1$. We say that a dominant monomial $Y_{i,k}Y_{j,v}$, where $i,j \in I$, $k, v\in \ZZ$, and $|j-i|\equiv |k-v|\ (\text{mod}\ 2)$, has small values of indices if $h_0(i,j)+k\leq v<h_0(i,j)+k+2 \ell$.
\end{definition}

\begin{remark} \label{remark:any degree 2 monomial satisfying certain condition can be converted to small value indices}
Let $\varepsilon^{2 \ell}=1$. For any dominant monomial $Y_{i,k}Y_{j,v}$, where $i,j \in I$, $k,v\in \ZZ$ and $|j-i|\equiv |k-v|\ (\text{mod}\ 2)$, we define $\overline{v}$ by
\begin{align}\label{Eq:small index}
\overline{v}\equiv v\,(\text{mod } 2 \ell)\quad \text{and}\quad h_0(i,j)+k\leq \overline{v}<h_0(i,j)+k+2 \ell.
\end{align}
Then $Y_{i,k}Y_{j,\overline{v}}$ has small values of indices and $L(Y_{i,k}Y_{j,\overline{v}})$ is a snake module.
\end{remark}

\begin{remark}\label{Re:non-snake module}
Let $Y_{i,k}Y_{j,v}$ be a dominant monomial, where $i,j \in I$, $k,v\in \ZZ$. If $|j-i|\equiv |k-v|+1\ (\text{mod}\ 2)$, then $L(Y_{i,k}Y_{j,v})$ cannot be converted to a snake module.
\end{remark}

We use the convention that $0!=1$ and $\prod_{t=s}^{s'} \phi(t) = 1$ for $s' < s$, where $\phi(t)$ is some polynomial in $t$.
\begin{definition}\label{Def: chi_(i,k,j,v)}
Let $\varepsilon^{2 \ell}=1$ and let $L(Y_{i,k}Y_{j,v})$, $i,j \in I$, $k<v\in \ZZ$, be a $U_{\varepsilon}^{\res}({L\mathfrak{sl}_{n+1}})$ snake module such that $Y_{i,k}Y_{j,v}$ has small values of indices. Let $r(i,j)$ be the integer defined by $(\ref{eq:def of rij})$ and let $a_t, \alpha_t, b_t, \beta_t$ be the integers defined in (\ref{Eq:a_tb_t}).
For $t, \rho \in \ZZ_{\ge 1}$, we denote
\begin{align*}
e(t) = \frac{\prod_{s=1}^{t-2} (t+s)}{(t-1)!}, \ f(t) = \frac{ \rho \prod_{s=1}^{t-1}(\rho+t+s)}{t!}, \ g(t) = \frac{ (\rho+1)\prod_{s=2}^t (\rho+t+s)}{t!},
\end{align*}
and we denote $g(0)=1$. For $r(i,j)=0$, we define $\chi_{(i,k),(j,v)}=0$. For $r(i,j) \ge 1$, we define $\chi_{(i,k),(j,v)}$ as follows.
\begin{enumerate}
\item For $|j-i|+h(i,j) < 2\,\ell$, $\chi_{(i,k),(j,v)}=\sum_{t=1}^{r(i,j)}e(t)\chi_\varepsilon(L(Y_{a_t, \alpha_t}Y_{b_t, \beta_t}))$.
\item For $|j-i|+h(i,j) =2\rho\,\ell$, $\rho\in\mathbb{Z}_{\geq 1}$, $\chi_{(i,k),(j,v)}=\sum_{t=1}^{r(i,j)} f(t)\chi_\varepsilon(L(Y_{a_t, \alpha_t}Y_{b_t, \beta_t}))$.
\item For $2\rho\,\ell<|j-i|+h(i,j) < 2(\rho+1)\,\ell$, $\rho\in\mathbb{Z}_{\geq1}$, $\gamma=\frac{(|j-i|+h(i,j))\,(\text{mod } 2 \ell)}2$,
\begin{align}\label{D:the sum 3}
\chi_{(i,k),(j,v)}= \sum_{t=1}^{b'}f(t)\chi_\varepsilon(L (m_{t}))+\sum_{t=0}^{r(i,j)-1}g(t)\chi_\varepsilon(L(Y_{a_{t+1}, \alpha_{t+1}}Y_{b_{t+1}, \beta_{t+1}})),
\end{align}
where if $a_{r(i,j)}-\gamma$, $b_{r(i,j)}+\gamma\in \hat{I}\cup\{n+1\}$, then $b'= r(i,j)$ and otherwise $b'= r(i,j)-1$.
If $i\leq j$, the dominant monomial $m_{t}$ is obtained by translating paths in $\mathscr{P}_{a_{t}, \alpha_{t}}$ to paths in $\mathscr{P}_{a_{t}, \overline{\alpha_t}}$ with respect to ${\bf PS}(b_t,\beta_t)$; if $i>j$, the dominant monomial $m_{t}$ is obtained by translating paths in $\mathscr{P}_{b_{t}, \beta_{t}}$ to paths in $\mathscr{P}_{b_{t}, \overline{\beta_t}}$ with respect to ${\bf PS}(a_t,\alpha_t)$, where $\overline{\alpha_t}, \overline{\beta_t}$ are defined in Equation $(\ref{Eq:small index})$.
\end{enumerate}
\end{definition}

Recall that for any path $p$, $\mathfrak{m}(p)$ is the monomial corresponding to $p$, see (\ref{map: path to monomial}).

\begin{theorem} \label{Th:path description of degree 2}
Let $\varepsilon^{2 \ell}=1$ and let $L(Y_{i,k}Y_{j,v})$ be a simple $U_{\varepsilon}^{\res}({L\mathfrak{sl}_{n+1}})$-module, where $i, j \in [1,n]$, $k, v \in \ZZ$. Then $L(Y_{i,k}Y_{j,v})$ is special and
\begin{enumerate}
\item if $|j-i|\equiv |k-v|+1\ (\text{mod}\ 2)$, then
\[
\chi_{\varepsilon}(L(Y_{i,k}Y_{j,v}))= \left(\sum_{p \in \mathscr{P}_{i,k}} \mathfrak{m}(p)\right) \left(\sum_{p \in \mathscr{P}_{j,v}} \mathfrak{m}(p)\right)=\chi_{\varepsilon}(L(Y_{i,k}))\chi_{\varepsilon}(L(Y_{j,v}));
\]
\item if $|j-i|\equiv |k-v|\ (\text{mod}\ 2)$, then
\begin{align*}
\chi_{\varepsilon}(L(Y_{i,k}Y_{j,v}))=\left(\sum_{(p_{1},p_{2}) \in \overline{\mathscr{P}}_{((i,k),(j,\overline{v}))}} \mathfrak{m}(p_1)\mathfrak{m}(p_2)\right)-\chi_{(i,k), (j,\overline{v})},
 \end{align*}
where $\overline{v}$ is defined in $(\ref{Eq:small index})$ and $\chi_{(i,k), (j,\overline{v})}$ is defined in Definition \ref{Def: chi_(i,k,j,v)}. 
\end{enumerate}
\end{theorem}

Theorem \ref{Th:path description of degree 2} will be proved in Section \ref{Sec: prove the th path description of degree 2}.

\begin{example}
Let $\varepsilon^{2 \ell}=1$ with $\ell=2$. For the $U_{\varepsilon}^{\res}({L\mathfrak{sl}_{3}})$-module $L(Y_{1,0}Y_{2,0})$, we have that $\chi_\varepsilon(L(Y_{1,0}Y_{2,0}))=\chi_\varepsilon(L(Y_{1,0})) \chi_\varepsilon(L(Y_{2,0}))$.
\end{example}

\begin{example}
Let $\varepsilon^{2 \ell}=1$ with $\ell=2$. We consider the $U_{\varepsilon}^{\res}({L\mathfrak{sl}_{8}})$-module $L(Y_{3,0}Y_{6,7})$.
Recall from Example \ref{Ex:two translations} that by translating paths in $\mathscr{P}_{3,0}$ to paths in $\mathscr{P}_{3,12}$ with respect to ${\bf PS}(6,7)$, we obtain the dominant monomial $Y_{2,11}Y_{7,8}$. The monomials of $\chi_\varepsilon(L(Y_{2,11}Y_{7,8}))$ are contained in $S_{(3,0)(6,7)}$. In addition, for the module $L(Y_{2,11}Y_{7,8})$, one can translate paths in $\mathscr{P}_{2,11}$ to paths in $\mathscr{P}_{2,15}$ with respect to ${\bf PS}(7,8)$, consequently obtain the dominant monomial $Y_{1,2}$. The monomials of $\chi_\varepsilon(L(Y_{1,2}))$ are contained in $S_{(3,0)(6,7)}$.

There is another copy of $Y_{1,2}$ which appears in $S_{(3,0)(6,7)}$. It is the product of $Y_{1,2}Y_{4,5}^{-1}Y_{6,3}$ and $Y_{4,9}Y_{6,11}^{-1}$, see the red lines in Figure \ref{F:  two translations} $(a)$. The lowest weight monomial $Y_{7,10}^{-1}$ of $\chi_{\varepsilon}(L(Y_{1,2}))$ is also in $S_{(3,0)(6,7)}$. Therefore,
\[
\chi_\varepsilon(L(Y_{3,0}Y_{6,7}))=S_{(3,0)(6,7)}-\chi_{\varepsilon}(L(Y_{7,8}Y_{2,11}))-2\chi_{\varepsilon}(L(Y_{1,2})).
\]
Using the same method as above, we have $\chi_\varepsilon(L(Y_{7,8}Y_{2,11}))=\chi_\varepsilon(L(Y_{7,0}Y_{2,7}))=S_{(7,0)(2,7)}$.
In conclusion,
\[
\chi_\varepsilon(L(Y_{3,0}Y_{6,7}))=S_{(3,0)(6,7)}-S_{(7,0)(2,7)}-2\chi_{\varepsilon}(L(Y_{1,2})).
\]
\end{example}

\begin{remark}
In the case of $q$-characters, for $U_q(L\mathfrak{sl}_{n+1})$-module $L(Y_{i,s})$ where $i \in I$, $s \in \ZZ$, we have that $\chi_q(L(Y_{i,s}^2))=\chi_q(L(Y_{i,s}))^2$. Unlike $q$-characters, for $U_{\varepsilon}^{\res}(L\mathfrak{sl}_{n+1})$-module $L(Y_{i,s})$, it is possible that $\chi_\varepsilon(L(Y_{i,s}^2)) \ne \chi_\varepsilon(L(Y_{i,s}))^2$. For example, let $\varepsilon^{2 \ell}=1$ with $\ell=2$ and consider $U_{\varepsilon}^{\res}({L\mathfrak{sl}_{4}})$-module $L(Y_{2,1})$. On the one hand, we have
\begin{align*}
&\chi_{\varepsilon}(L(Y_{2,1}^2))=\chi_{\varepsilon}(L(Y_{2,1}Y_{2,5}))=\big(\sum_{(p_{1},p_{2}) \in \overline{\mathscr{P}}_{((2,1),(2,5))}} \mathfrak{m}(p_1)\mathfrak{m}(p_2)\big)-\chi_{(2,1), (2,5)}=S_{(2,1)(2,5)}-1,
\end{align*}
where $S_{(2,1)(2,5)}$ is equal to
\begin{align*} 
&Y_{2,1}^2+Y_{1,2}^2Y_{3,4}^{-2}+Y_{3,2}^2Y_{1,4}^{-2}+ 2Y_{1,2}Y_{2,1}^{-1}Y_{3,4}^{-1} + 2Y_{3,2}Y_{1,4}^{-1}Y_{2,1}^{-1}+Y_{1,2}^2Y_{3,2}^2Y_{2,3}^{-2} \\
& + Y_{2,3}^2Y_{1,4}^{-2}Y_{3,4}^{-2} +2Y_{1,2}Y_{2,1}Y_{3,4}^{-1}+2Y_{2,1}Y_{3,2}Y_{1,4}^{-1}+2Y_{1,2}Y_{2,1}Y_{3,2}Y_{2,3}^{-1} \\
& +2Y_{1,2}Y_{2,3}Y_{1,4}^{-1}Y_{3,4}^{-2}+2Y_{1,2}Y_{3,2}Y_{2,1}^{-1}Y_{2,3}^{-1} +4Y_{1,2}Y_{3,2}Y_{1,4}^{-1}Y_{3,4}^{-1} + 2Y_{2,1}Y_{2,3}Y_{1,4}^{-1}Y_{3,4}^{-1} \\
& +2Y_{2,3}Y_{3,2}Y_{1,4}^{-2}Y_{3,4}^{-1} +2Y_{1,2}Y_{3,2}^2Y_{1,4}^{-1}Y_{2,3}^{-1} +2Y_{2,3}Y_{1,4}^{-1}Y_{2,1}^{-1}Y_{3,4}^{-1}+2Y_{1,2}^2Y_{3,2}Y_{2,3}^{-1}Y_{3,4}^{-1}+Y_{2,1}^{-2} + 1.
\end{align*}
On the other hand, $\chi_{\varepsilon}(L(Y_{2,1})) =Y_{2,1}+Y_{1,2}Y^{-1}_{2, 3}Y_{3,2}+Y_{1, 4}^{-1}Y_{3,2}+Y_{1,2}Y_{3, 4}^{-1}+Y_{1, 4}^{-1}Y_{2,3}Y_{3, 4}^{-1}+Y_{2, 1}^{-1}$. We have that $\chi_{\varepsilon}(L(Y_{2,1}^2))=\chi_{\varepsilon}(L(Y_{2,1}))^2-2$.
\end{remark}

\subsection{Tensor product of the fundamental representations of restricted quantum loop algebras at roots of unity}\label{tensor products}

\begin{theorem}\label{Th: tensor product}
Let $\varepsilon^{2 \ell}=1$ and $i_{1}, i_{2}\in I$, $\xi_{1}, \xi_{2}\in \mathbb{Z}$. The following conditions are equivalent.
\begin{enumerate}
\item The tensor product $L(Y_{i_1,\xi_1})\otimes L(Y_{i_2,\xi_2})$ is a simple module of $U_{\varepsilon}^{\res}(L\mathfrak{sl}_{n+1})$.
\item For $1\leq t\leq \min\{i_1, i_2, n+1-i_1, n+1-i_2\}$, $|\xi_{2}-\xi_1| \not\equiv \pm(2t+|i_2-i_1|)\,(\text{mod } 2 \ell)$.
\end{enumerate}
\end{theorem}

\begin{proof}
$(1)\Rightarrow(2)$. Suppose that $L(Y_{i_1,\xi_1})\otimes L(Y_{i_2,\xi_2})$ is a simple module. Then
\[
\chi_\varepsilon(L(Y_{i_1,\xi_1})) \chi_\varepsilon(L(Y_{i_2,\xi_2}))=\chi_\varepsilon(L(Y_{i_1,\xi_1}Y_{i_2,\xi_2})).
\]
If $|i_2-i_1|\equiv |\xi_2-\xi_1|+1\ (\text{mod}\ 2)$, then for any $t \in \ZZ$, $|\xi_{2}-\xi_1| \not\equiv \pm(2t+|i_2-i_1|)\,(\text{mod } 2 \ell)$. In particular, (2) is true. If $|i_2-i_1|\equiv |\xi_2-\xi_1|\ (\text{mod}\ 2)$, then by Remark \ref{remark:any degree 2 monomial satisfying certain condition can be converted to small value indices}, there exists $\overline{\xi}_2$ such that $Y_{i_1,\xi_1}Y_{i_2,\overline{\xi}_2}$ has small values of indices and $L(Y_{i_1,\xi_1}Y_{i_2,\overline{\xi}_2})$ is a snake module, where $\overline{\xi}_2\equiv \xi_2\,(\text{mod } 2 \ell)$. Since $L(Y_{i_1,\xi_1})\otimes L(Y_{i_2,\xi_2})$ is simple, we have that $L(Y_{i_1,\xi_1})\otimes L(Y_{i_2,\overline{\xi}_2})$ is simple. Following Theorem \ref{Th:path description of degree 2}, we obtain that $(i_2,\overline{\xi}_2) \notin {\bf PS}(i_1,\xi_1)$ and
$\chi_{(i_1,\xi_1),(i_{2},\overline{\xi}_2)}=0$. By Lemma \ref{Le: I translation of degree 2} and \ref{Le: II translation of degree 2}, we have
\begin{align*}
h(i_1,i_2) \not\equiv h_{0}(i_1,i_2)+2s\,(\text{mod } 2 \ell) \ \text{and } \  h(i_1,i_2) \not\equiv -h_{0}(i_1,i_2)-2s\,(\text{mod } 2 \ell),
\end{align*}
where $0\leq s<|\mathbf{PS}^{i_2}(i_1,\xi_1)|$. We know that $h_{0}(i_1,i_2)+2s=|i_2-i_1|+2(s+1)$, $h(i_1,i_2)=|\overline{\xi}_2-\xi_1|$, and $|\overline{\xi}_2-\xi_1|\equiv \pm|\xi_2-\xi_1|\,(\text{mod } 2 \ell)$. By the definition of $\mathbf{PS}^{i_2}(i_1,\xi_1)$, see Equation $(\ref{Prime snake 2})$, we have
\[
|\mathbf{PS}^{i_2}(i_1,\xi_1)|= \min\{i_1, i_2, n+1-i_1, n+1-i_2\}.
\]
Let $s+1=t$. Thus, we have that $|\xi_2-\xi_1| \not\equiv \pm(2t+|i_2-i_1|)\,(\text{mod } 2 \ell)$.

$(2)\Rightarrow(1)$. Suppose that $|\xi_2-\xi_1| \not\equiv \pm(2t+|i_2-i_1|)\,(\text{mod } 2 \ell)$, where $1\leq t\le \min\{i_1, i_2, n+1-i_1, n+1-i_2\}$. If $|\xi_2-\xi_1| \equiv |i_2-i_1|+1 \,(\text{mod } 2 )$, then by Theorem \ref{Th:path description of degree 2}, we have that $L(Y_{i_1,\xi_1})\otimes L(Y_{i_2,\xi_2})$ is simple. Assume that $|\xi_2-\xi_1| \equiv |i_2-i_1| \,(\text{mod } 2 )$. Then by Remark \ref{remark:any degree 2 monomial satisfying certain condition can be converted to small value indices}, there exists $\overline{\xi}_2$ such that $Y_{i_1,\xi_1}Y_{i_2,\overline{\xi}_2}$ has small values of indices and $L(Y_{i_1,\xi_1}Y_{i_2,\overline{\xi}_2})$ is a snake module, where $\overline{\xi}_2 \equiv \xi_2\,(\text{mod } 2 \ell)$. Since $|\xi_2-\xi_1| \not\equiv \pm(2t+|i_2-i_1|)\,(\text{mod } 2 \ell)$, where $1\leq t\leq \min\{i_1, i_2, n+1-i_1, n+1-i_2\}$, we have $|\overline{\xi}_2-\xi_1| \not\equiv \pm(2t+|i_2-i_1|)\,(\text{mod } 2 \ell)$. That is,
\begin{align*}
h(i_1,i_2) \not\equiv h_{0}(i_1,i_2)+2s\,(\text{mod } 2 \ell) \ \text{and } \  h(i_1,i_2) \not\equiv -h_{0}(i_1,i_2)-2s\,(\text{mod } 2 \ell),
\end{align*}
where $0\leq s<|\mathbf{PS}^{i_2}(i_1,\xi_1)|$. So $(i_2, \overline{\xi}_2) \notin {\bf PS}(i_1,\xi_1)$ and $\chi_{(i_1,\xi_1),(i_{2},\overline{\xi}_2)}=0$. By Theorem \ref{Th:path description of degree 2}, we obtain that $L(Y_{i_1,\xi_1})\otimes L(Y_{i_2, \overline{\xi}_2})$ is simple. Thus $L(Y_{i_1,\xi_1})\otimes L(Y_{i_2, \xi_2})$ is simple.
\end{proof}

We will need the following lemma. The proof of the lemma is similar to the proof of the corresponding result for $U_q(L\mathfrak{g})$-modules in \cite{H10}: for simple $U_q(L\mathfrak{g})$-modules $S_1, \ldots, S_N$, if the tensor product $S_1 \otimes \cdots \otimes S_N$ is simple, then for any $1\leq i\neq j\leq N$, $S_i \otimes S_j$ is simple.
\begin{lemma}\label{Le: simple tensor product}
Let $\mathfrak{g}$ be a simple Lie algebra over $\CC$ and let $\varepsilon$ be a root of unity. For simple $U_{\varepsilon}^{\res}({L\mathfrak{g}})$-modules $S_1,\ldots, S_N$, $N\in \ZZ_{\geq 2}$, if the tensor product $S_1 \otimes \cdots \otimes S_N$ is simple, then for any $1\leq i\neq j\leq N$, $S_i \otimes S_j$ is simple.
\end{lemma}

\begin{proof}
Since the $\varepsilon$-character morphism $\chi_{\varepsilon}: \mathcal{K}_0({\rm Rep} U_\varepsilon^{\res}({L\mathfrak{g}})) \rightarrow \mathbb{Z}[Y_{i,a}^{\pm1}]_{i\in I}^{a\in\mathbb{C^{\times}}}$ is injective \cite{FM02}, the ring $\mathcal{K}_0({\rm Rep} U_\varepsilon^{\res}({L\mathfrak{g}}))$ is commutative. So the irreducibility of $S_1 \otimes \cdots \otimes S_N$ is equivalent to the irreducibility of $S_\sigma =S_{\sigma_1}\otimes \cdots \otimes S_{\sigma_N}$ for any permutation $\sigma$ of $[1,N]$. Let $i\neq j$ and take $\sigma$ such that $\sigma(i)=1$ and $\sigma(j)=2$. If $S_i\otimes S_j$ is not simple, then there is a proper submodule $V \subset S_i \otimes S_j$. Therefore there is a proper submodule $V\otimes S_{\sigma_3}\otimes \cdots \otimes S_{\sigma_N}\subset S_\sigma$. Hence $S_1 \otimes \cdots \otimes S_N$ is not simple.
\end{proof}

\begin{corollary}\label{Co:necessary condition for the tensor product}
Let $m\in \mathbb{Z}_{\geq 2}$, $i_{1},\ldots, i_{m}\in I$, $\xi_{1},\ldots, \xi_{m}\in \mathbb{Z}$, and $\varepsilon^{2 \ell}=1$. Assume that the tensor product
$L(Y_{i_1,\xi_1})\otimes L(Y_{i_2,\xi_2})\otimes \cdots \otimes L(Y_{i_m,\xi_m})$ is a simple module of $U_{\varepsilon}^{\res}(L\mathfrak{sl}_{n+1})$.
Then for any $1\leq k\neq k'\leq m$ and $1\leq t\leq \min\{i_k, i_{k'}, n+1-i_k, n+1-i_{k'}\}$,
\[
 | \xi_{k'}-\xi_{k} | \not\equiv \pm(2t+|i_{k'}-i_k|)\,(\text{mod } 2 \ell).
\]
\end{corollary}

\section{\texorpdfstring{A path description for $\varepsilon$-characters of Kirillov-Reshetikhin modules}{A path description for epsilon-characters of Kirillov-Reshetikhin modules}}\label{Kirillov-Reshetikhin}
In this section, we introduce a concept of glue of the lines $x=0$ with $x=2 \ell$ in a lattice with paths, and then we give an effective and concrete path description for $\varepsilon$-characters of any Kirillov-Reshetikhin module of $U_{\varepsilon}^{\res}({L\mathfrak{sl}_{n+1}})$.

\subsection{A path description for $\varepsilon$-characters of Kirillov-Reshetikhin modules}

\begin{definition}\label{D:glue}
Let $\varepsilon^{2 \ell}=1$. For the lines $x=0$ with $x=2 \ell$ in a lattice with paths, the glue of these two lines is the identification of the lines $x=c$ and $x=2d \ell+c$, $0\leq c<2 \ell$, $d\in \mathbb{Z}_{\geq 1}$.
\end{definition}

Gluing the lines $x=0$ with $x=2 \ell$ of a lattice, the lattice becomes a tube.

Let $p_1, p_2$ be paths. Define $p_1\cap p_2= \{(a, b):(a, b)\in p_1, (a, b) \in p_2\}$.
We say that $p_1$ and $p_2$ are \textit{disjoint} if $p_1\cap p_2=\emptyset$. We say that a $z$-tuple of paths $(p_{1},\ldots,p_{z})$ is disjoint if $p_{s}\cap p_{t}=\emptyset$ for all $1\leq s\neq t\leq z$. Note that ``disjoint'' is slightly different from ``non-overlapping'' defined in Section \ref{sec:snake modules and three types of translations of paths}.

In a tube, for $\mathfrak{a} \in \{0,1\}$ and any snake $(i,k_t)\in \mathcal{X}_{\mathfrak{a}}$, $1\leq t\leq z$, $z\in \mathbb{Z}_{\geq1}$, we denote
\begin{align}\label{def: disjoint path in a tube}
\overline{\mathscr{P'}}_{{{(i,k_t)}_{1\leq t\leq z}}}=\{(p_{1},\ldots,p_{z}): p_{t}\in \mathscr{P}_{i,k_t}, 1\leq t\leq z, (p_{1},\ldots,p_{z})\text { is disjoint} \},
\end{align}
where $\mathscr{P}_{i,k_t}$ is defined in $(\ref{path})$.

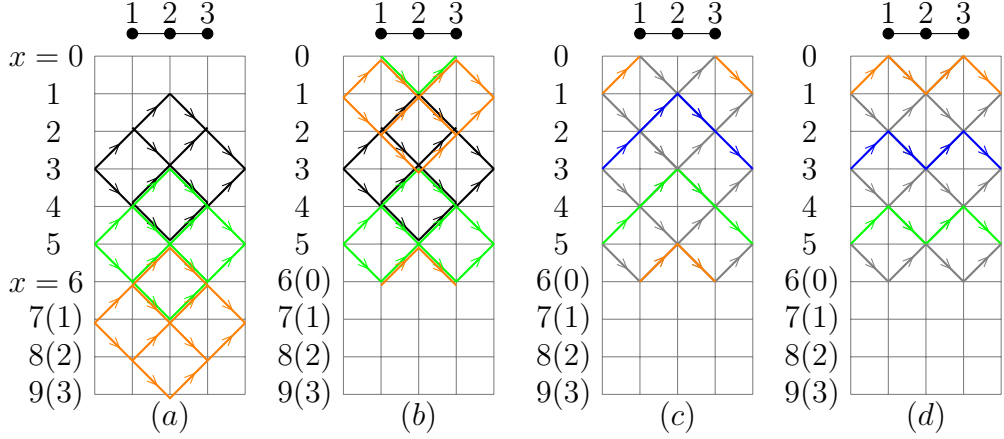
\begin{figure}
\resizebox{1.0\width}{1.0\height}{
\begin{minipage}[b]{0.2\linewidth}
\centerline{
\begin{tikzpicture}
\draw[step=.5cm,gray,thin] (-0.5,4.5) grid (1.5,9) (-0.5,4.5)--(1.5,4.5);
\draw[fill] (0,9.3) circle (2pt) -- (0.5,9.3) circle (2pt) --(1,9.3) circle (2pt);
\begin{scope}[thick, every node/.style={sloped,allow upside down}]
\draw (-0.5,7.5)--node {\midarrow}(0,8);
\draw (0,8)--node {\midarrow}(0.5,8.5);
\draw (0.5,8.5)--node {\midarrow}(1,8);
\draw (1,8)--node {\midarrow}(1.5,7.5);
\draw (-0.5,7.5)--node {\midarrow}(0,7);
\draw (0,7.05)--node {\midarrow}(0.5,6.55);
\draw (0.5,6.55)--node {\midarrow}(1,7.05);
\draw (1,7)--node {\midarrow}(1.5,7.5);
\draw (0,7.05)--node {\midarrow}(0.5,7.55);
\draw (0.5,7.55)--node {\midarrow}(1,8.05);
\draw (0,8.05)--node {\midarrow}(0.5,7.55);
\draw (0.5,7.55)--node {\midarrow}(1,7.05);
\textcolor{green}{
\draw (-0.5,6.5)--node {\midarrow}(0,7)--node {\midarrow}(0.5,7.5)--node {\midarrow}(1,7)--node {\midarrow}(1.5,6.5);
\draw (-0.5,6.5)--node {\midarrow}(0,6)--node {\midarrow}(0.5,5.5)--node {\midarrow}(1,6)--node {\midarrow}(1.5,6.5);
\draw (0,6)--node {\midarrow}(0.5,6.5)--node {\midarrow}(1,7);
\draw (0,7)--node {\midarrow}(0.5,6.5)--node {\midarrow}(1,6);}
\draw[orange] (-0.5,5.45)--node {\midarrow}(0,5.95)--node {\midarrow}(0.5,6.45)--node {\midarrow}(1,5.95)--node {\midarrow}(1.5,5.45);
\draw[orange] (-0.5,5.45)--node {\midarrow}(0,4.95)--node {\midarrow}(0.5,4.45)--node {\midarrow}(1,4.95)--node {\midarrow}(1.5,5.45);
\draw[orange](0,4.95)--node {\midarrow}(0.5,5.45)--node {\midarrow}(1,5.95);
\draw[orange] (0,5.95)--node {\midarrow}(0.5,5.45)--node {\midarrow}(1,4.95);
\end{scope}
\node [above] at (0,9.3) {$1$};
\node [above] at (0.5,9.3) {$2$};
\node [above] at (1,9.3) {$3$};
\node [left] at (-0.5,9) {$x=0$};
\node [left] at (-0.8,8.5) {$1$};
\node [left] at (-0.8,8) {$2$};
\node [left] at (-0.8,7.5) {$3$};
\node [left] at (-0.8,7) {$4$};
\node [left] at (-0.8,6.5) {$5$};
\node [left] at (-0.5,6) {$x=6$};
\node [left] at (-0.5,5.5) {$7(1)$};
\node [left] at (-0.5,5) {$8(2)$};
\node [left] at (-0.5,4.5) {$9(3)$};
\node at (0.5,4.2)  {$(a)$};
\end{tikzpicture}}
\end{minipage}
\begin{minipage}[b]{0.2\linewidth}
\centerline{
\begin{tikzpicture}
\draw[step=.5cm,gray,thin] (-0.5,4.5) grid (1.5,9) (-0.5,4.5)--(1.5,4.5);
\draw[fill] (0,9.3) circle (2pt) -- (0.5,9.3) circle (2pt) --(1,9.3) circle (2pt);
\begin{scope}[thick, every node/.style={sloped,allow upside down}]
\draw (-0.5,7.5)--node {\midarrow}(0,8);
\draw (0,8)--node {\midarrow}(0.5,8.5);
\draw (0.5,8.5)--node {\midarrow}(1,8);
\draw (1,8)--node {\midarrow}(1.5,7.5);
\draw (-0.5,7.5)--node {\midarrow}(0,7);
\draw (0,7.05)--node {\midarrow}(0.5,6.55);
\draw (0.5,6.55)--node {\midarrow}(1,7.05);
\draw (1,7)--node {\midarrow}(1.5,7.5);
\draw (0,7.05)--node {\midarrow}(0.5,7.55);
\draw (0.5,7.55)--node {\midarrow}(1,8.05);
\draw (0,8.05)--node {\midarrow}(0.5,7.55);
\draw (0.5,7.55)--node {\midarrow}(1,7.05);
\textcolor{green}{
\draw (-0.5,6.5)--node {\midarrow}(0,7);
\draw (0,7)--node {\midarrow}(0.5,7.5);
\draw (0.5,7.5)--node {\midarrow}(1,7);
\draw (1,7)--node {\midarrow}(1.5,6.5);
\draw (-0.5,6.5)--node {\midarrow}(0,6);
\draw (0,9)--node {\midarrow}(0.5,8.5);
\draw (0.5,8.5)--node {\midarrow}(1,9);
\draw (1,6)--node {\midarrow}(1.5,6.5);
\draw (0,6)--node {\midarrow}(0.5,6.5);
\draw (0.5,6.5)--node {\midarrow}(1,7);
\draw (0,7)--node {\midarrow}(0.5,6.5);
\draw (0.5,6.5)--node {\midarrow}(1,6);}
\draw[orange] (-0.5,8.45)--node {\midarrow}(0,8.95);
\draw[orange] (1,8.95)--node {\midarrow}(1.5,8.45);
\draw[orange] (-0.5,8.45)--node {\midarrow}(0,7.95)--node {\midarrow}(0.5,7.45)--node {\midarrow}(1,7.95)--node {\midarrow}(1.5,8.45);
\draw[orange](0,7.95)--node {\midarrow}(0.5,8.45)--node {\midarrow}(1,8.95);
\draw[orange] (0,8.95)--node {\midarrow}(0.5,8.45)--node {\midarrow}(1,7.95);
\draw[orange](0,5.95)--node {\midarrow}(0.5,6.45)--node {\midarrow}(1,5.95);
\end{scope}
\node [above] at (0,9.3) {$1$};
\node [above] at (0.5,9.3) {$2$};
\node [above] at (1,9.3) {$3$};
\node [left] at (-0.8,9) {$0$};
\node [left] at (-0.8,8.5) {$1$};
\node [left] at (-0.8,8) {$2$};
\node [left] at (-0.8,7.5) {$3$};
\node [left] at (-0.8,7) {$4$};
\node [left] at (-0.8,6.5) {$5$};
\node [left] at (-0.5,6) {$6(0)$};
\node [left] at (-0.5,5.5) {$7(1)$};
\node [left] at (-0.5,5) {$8(2)$};
\node [left] at (-0.5,4.5) {$9(3)$};
\node at (0.5,4.2)  {$(b)$};
\end{tikzpicture}}
\end{minipage}
\begin{minipage}[b]{0.2\linewidth}
\centerline{
\begin{tikzpicture}
\draw[step=.5cm,gray,thin] (-0.5,4.5) grid (1.5,9) (-0.5,4.5)--(1.5,4.5);
\draw[fill] (0,9.3) circle (2pt) -- (0.5,9.3) circle (2pt) --(1,9.3) circle (2pt);
\begin{scope}[thick, every node/.style={sloped,allow upside down}]
\draw [gray](-0.5,7.5)--node {\midarrow}(0,7);
\draw [gray](0,7)--node {\midarrow}(0.5,6.5);
\draw [gray](0.5,6.5)--node {\midarrow}(1,7);
\draw [gray](1,7)--node {\midarrow}(1.5,7.5);
\draw [gray](0,7)--node {\midarrow}(0.5,7.5);
\draw [gray](0,8)--node {\midarrow}(0.5,7.5);
\draw [gray](0.5,7.5)--node {\midarrow}(1,7);
\draw [green](-0.5,6.5)--node {\midarrow}(0,7)--node {\midarrow}(0.5,7.5)--node {\midarrow}(1,7)--node {\midarrow}(1.5,6.5);
\draw [gray](-0.5,6.5)--node {\midarrow}(0,6);
\draw [gray](1,6)--node {\midarrow}(1.5,6.5);
\draw [gray](0,6)--node {\midarrow}(0.5,6.5);
\draw [gray](0.5,6.5)--node {\midarrow}(1,6);
\draw [orange] (-0.5,8.5)--node {\midarrow}(0,9);
\draw [orange] (0,6)--node {\midarrow}(0.5,6.5)--node {\midarrow}(1,6);
\draw [orange] (1,9)--node {\midarrow}(1.5,8.5);
\draw [gray](-0.5,8.5)--node {\midarrow}(0,8)--node {\midarrow}(0.5,7.5)--node {\midarrow}(1,8)--node {\midarrow}(1.5,8.5);
\draw[gray](0,8)--node {\midarrow}(0.5,8.5)--node {\midarrow}(1,9);
\draw [gray](0,9)--node {\midarrow}(0.5,8.5)--node {\midarrow}(1,8);
\draw [blue] (-0.5,7.5)--node {\midarrow}(0,8)--node {\midarrow}(0.5,8.5)--node {\midarrow}(1,8)--node {\midarrow}(1.5,7.5);
\end{scope}
\node [above] at (0,9.3) {$1$};
\node [above] at (0.5,9.3) {$2$};
\node [above] at (1,9.3) {$3$};
\node [left] at (-0.8,9) {$0$};
\node [left] at (-0.8,8.5) {$1$};
\node [left] at (-0.8,8) {$2$};
\node [left] at (-0.8,7.5) {$3$};
\node [left] at (-0.8,7) {$4$};
\node [left] at (-0.8,6.5) {$5$};
\node [left] at (-0.5,6) {$6(0)$};
\node [left] at (-0.5,5.5) {$7(1)$};
\node [left] at (-0.5,5) {$8(2)$};
\node [left] at (-0.5,4.5) {$9(3)$};
\node at (0.5,4.2)  {$(c)$};
\end{tikzpicture}}
\end{minipage}\begin{minipage}[b]{0.2\linewidth}
\centerline{
\begin{tikzpicture}
\draw[step=.5cm,gray,thin] (-0.5,4.5) grid (1.5,9) (-0.5,4.5)--(1.5,4.5);
\draw[fill] (0,9.3) circle (2pt) -- (0.5,9.3) circle (2pt) --(1,9.3) circle (2pt);
\begin{scope}[thick, every node/.style={sloped,allow upside down}]
\draw [gray](-0.5,7.5)--node {\midarrow}(0,7);
\draw [gray](0,7)--node {\midarrow}(0.5,6.5);
\draw [gray](0.5,6.5)--node {\midarrow}(1,7);
\draw [gray](1,7)--node {\midarrow}(1.5,7.5);
\draw [gray](0,7)--node {\midarrow}(0.5,7.5);
\draw [gray](0,8)--node {\midarrow}(0.5,7.5);
\draw [gray](0.5,7.5)--node {\midarrow}(1,7);
\draw [green](-0.5,6.5)--node {\midarrow}(0,7)--node {\midarrow}(0.5,6.5)--node {\midarrow}(1,7)--node {\midarrow}(1.5,6.5);
\draw [gray](-0.5,6.5)--node {\midarrow}(0,6);
\draw [gray](1,6)--node {\midarrow}(1.5,6.5);
\draw [gray](0,6)--node {\midarrow}(0.5,6.5);
\draw [gray](0.5,6.5)--node {\midarrow}(1,6);
\draw [gray](-0.5,8.5)--node {\midarrow}(0,8)--node {\midarrow}(0.5,7.5)--node {\midarrow}(1,8)--node {\midarrow}(1.5,8.5);
\draw[gray](0,8)--node {\midarrow}(0.5,8.5)--node {\midarrow}(1,9);
\draw [gray](0,9)--node {\midarrow}(0.5,8.5)--node {\midarrow}(1,8);
\draw [blue] (-0.5,7.5)--node {\midarrow}(0,8)--node {\midarrow}(0.5,7.5)--node {\midarrow}(1,8)--node {\midarrow}(1.5,7.5);
\draw [orange] (-0.5,8.5)--node {\midarrow}(0,9)--node {\midarrow}(0.5,8.5)--node {\midarrow}(1,9)--node {\midarrow}(1.5,8.5);
\end{scope}
\node [above] at (0,9.3) {$1$};
\node [above] at (0.5,9.3) {$2$};
\node [above] at (1,9.3) {$3$};
\node [left] at (-0.8,9) {$0$};
\node [left] at (-0.8,8.5) {$1$};
\node [left] at (-0.8,8) {$2$};
\node [left] at (-0.8,7.5) {$3$};
\node [left] at (-0.8,7) {$4$};
\node [left] at (-0.8,6.5) {$5$};
\node [left] at (-0.5,6) {$6(0)$};
\node [left] at (-0.5,5.5) {$7(1)$};
\node [left] at (-0.5,5) {$8(2)$};
\node [left] at (-0.5,4.5) {$9(3)$};
\node at (0.5,4.2)  {$(d)$};
\end{tikzpicture}}
\end{minipage}}
\caption{For the $U_{\varepsilon}^{\res}({L\mathfrak{sl}_4})$-module $L(Y_{2,1}Y_{2,3}Y_{2,5})$, where $\varepsilon^{2 \ell}=1$, $\ell=3$. (a) $\mathscr{P}_{2,1}$, $\mathscr{P}_{2,3}$ and $\mathscr{P}_{2,5}$, respectively. (b) The tube obtained by gluing the lines $x=0$ with $x=6$ in figure (a). (c) Paths corresponding to monomials $Y_{2,1}$, $Y_{2,3}$, and $Y_{2,5}$ are drawn with blue, green, and orange lines, respectively. (d) Paths corresponding to monomials $Y_{1,2}Y_{2,3}^{-1}Y_{3,2}$, $Y_{1,4}Y_{2,5}^{-1}Y_{3,4}$, and $Y_{1,0}Y_{2,1}^{-1}Y_{3,0}$ are drawn with blue, green, and orange lines, respectively.}\label{F: k-r module}
\end{figure}

We have the following result.
\begin{theorem}\label{Th: the path description of K-R module}
Let $\varepsilon^{2 \ell}=1$ and let $L(Y_{i,k_1}\cdots Y_{i,k_z})$ be a Kirillov-Reshetikhin module of $U_{\varepsilon}^{\res}({L\mathfrak{sl}_{n+1}})$, where $Y_{i,k_1}\cdots Y_{i,k_z}$ has small values of indices, $i \in I=[1,n]$, $z \in [1,\ell]$, $k_t \in \ZZ$, $t\in [1,z]$. Then
\begin{align*}
\chi_{\varepsilon}(L(Y_{i,k_1}\cdots Y_{i,k_z}))=\sum_{(p_{1},\ldots,p_{z}) \in \overline{\mathscr{P'}}_{(i,k_{t})_{1 \leq t \leq z}}} \prod_{t=1}^{z}\mathfrak{m}(p_{t}),
\end{align*}
where $\mathfrak{m}(p_t)$ is the monomial of the path $p_t$, $1\leq t\leq z$, which is given in $(\ref{map: path to monomial})$, and $\overline{\mathscr{P'}}_{{(i,k_{t})_{ 1\leq t\leq z}}}$ is defined in $(\ref{def: disjoint path in a tube})$.
\end{theorem}

Theorem \ref{Th: the path description of K-R module} will be proved in Sections \ref{subsec:lemmas for proving theorem of path description of KR modules} and \ref{subsec: prove Th K-R module}.

\begin{remark}\label{remark:path description for KR modules with large degree}
Let $\varepsilon^{2 \ell}=1$ and let $L(Y_{i,k_1}\cdots Y_{i,k_z})$ be a Kirillov-Reshetikhin module of $U_{\varepsilon}^{\res}({L\mathfrak{sl}_{n+1}})$, where $Y_{i,k_1}\cdots Y_{i,k_z}$ has small values of indices. In the case of $z>\ell$, $z=a \ell+b$ for some $a\in \mathbb{Z}_{\geq 1}$, $0\leq b< \ell$. By Theorem \ref{Th:decomposition}, we have
\[
\chi_{\varepsilon}(L(Y_{i,k_1}\cdots Y_{i,k_z}))=\chi_{\varepsilon}(L( (Y_{i,k_1}\cdots Y_{i,k_{\ell}})^a )) \chi_{\varepsilon}(L(Y_{i,k_1}\cdots Y_{i,k_b})).
\]
The $\varepsilon$-characters $\chi_{\varepsilon}(L(Y_{i,k_1}\cdots Y_{i,k_{\ell}}))$ and $\chi_{\varepsilon}(L(Y_{i,k_1}\cdots Y_{i,k_b}))$ can be obtained using Theorem \ref{Th: the path description of K-R module}.
\end{remark}

\begin{example}\label{Ex:2_12_32_5}
Let $\varepsilon^{2 \ell}=1$ with $\ell=3$. We consider the $U_{\varepsilon}^{\res}({L\mathfrak{sl}_{4}})$-module $L(Y_{2,1}Y_{2,3}Y_{2,5})$. In Figure \ref{F: k-r module} $(a)$, gluing the lines $x=0$ with $x=6$, we obtain a tube in Figure \ref{F: k-r module} $(b)$. We have
\begin{align}\label{Eq: glue}
&\chi_{\varepsilon}(L(Y_{2,1}Y_{2,3}Y_{2,5}))=Y_{2,1}Y_{2,3}Y_{2,5}+Y_{1,0}Y_{1,2}Y_{1,4}Y_{2,1}^{-1}Y_{2,3}^{-1}Y_{2,5}^{-1}Y_{3,0}Y_{3,2}Y_{3,4}+Y_{2,1}^{-1}Y_{2,3}^{-1}Y_{2,5}^{-1}+\nonumber\\
&Y_{1,0}Y_{1,2}Y_{1,4}Y_{3,0}^{-1}Y_{3,2}^{-1}Y_{3,4}^{-1}+Y_{1,0}^{-1}Y_{1,2}^{-1}Y_{1,4}^{-1}Y_{3,0}Y_{3,2}Y_{3,4}
+Y_{1,0}^{-1}Y_{1,2}^{-1}Y_{1,4}^{-1}Y_{2,1}Y_{2,3}Y_{2,5}Y_{3,0}^{-1}Y_{3,2}^{-1}Y_{3,4}^{-1}.
\end{align}
The first term on the right-hand side of the Equation $(\ref{Eq: glue})$ is $\mathfrak{m}(p_1)\mathfrak{m}(p_2)\mathfrak{m}(p_3)=Y_{2,1}Y_{2,3}Y_{2,5}$, where $p_1 \in \mathscr{P}_{2,1}$, $\mathfrak{m}(p_1)=Y_{2,1}$, $p_2 \in \mathscr{P}_{2,3}$, $\mathfrak{m}(p_2)=Y_{2,3}$, and $p_3 \in \mathscr{P}_{2,5}$, $\mathfrak{m}(p_3)=Y_{2,5}$. Paths $p_1$, $p_2$, and $p_3$ drawn with blue, green, and orange lines, respectively, in Figure \ref{F: k-r module} $(c)$ are disjoint. The second term on the right-hand side of the Equation $(\ref{Eq: glue})$ is $\mathfrak{m}(p_1)\mathfrak{m}(p_2)\mathfrak{m}(p_3)=Y_{1,0}Y_{1,2}Y_{1,4}Y_{2,1}^{-1}Y_{2,3}^{-1}Y_{2,5}^{-1}Y_{3,0}Y_{3,2}Y_{3,4}$, where $p_1 \in \mathscr{P}_{2,1}$, $\mathfrak{m}(p_1)=Y_{1,2}Y_{2,3}^{-1}Y_{3,2}$, $p_2\in \mathscr{P}_{2,3}$, $\mathfrak{m}(p_2)=Y_{1,4}Y_{2,5}^{-1}Y_{3,4}$, and $p_3\in \mathscr{P}_{2,5}$, $\mathfrak{m}(p_3)=Y_{1,0}Y_{2,1}^{-1}Y_{3,0}$. Paths $p_1$, $p_2$, and $p_3$ drawn with blue, green, and orange lines, respectively, in Figure \ref{F: k-r module} (d) are disjoint. The other terms are similarly obtained.
\end{example}

\begin{example}
Let $\varepsilon^{2 \ell}=1$ with $\ell=3$. For the $U_{\varepsilon}^{\res}({L\mathfrak{sl}_{4}})$-module $L(Y_{2,1}Y_{2,3}Y_{2,5}Y_{2,7})$, we have that
\[
\chi_{\varepsilon}(L(Y_{2,1}Y_{2,3}Y_{2,5}Y_{2,7}))=\chi_{\varepsilon}(L(Y_{2,1}^2Y_{2,3}Y_{2,5}))=\chi_{\varepsilon}(L(Y_{2,1}Y_{2,3}Y_{2,5}))\chi_{\varepsilon}(L(Y_{2,1})).
\]
The $\varepsilon$-character $\chi_{\varepsilon}(L(Y_{2,1}Y_{2,3}Y_{2,5}))$ is computed in Example \ref{Ex:2_12_32_5}.
\end{example}

\subsection{Preparation for proving Theorem \ref{Th: the path description of K-R module}} \label{subsec:lemmas for proving theorem of path description of KR modules}

For $a \leq b \in \mathbb{Z}$, we denote $[a,b] = \{a, a+1, a+2, \ldots, b\}$ and $(a,b] = \{a+1, a+2, \ldots, b\}$. When $a=b$, $(a,b] = \emptyset$. For $a\in \mathbb{R}$, denote $\lceil a\rceil$ the smallest integer greater than or equal to $a$, and denote $\lfloor a\rfloor$ the largest integer less than or equal to $a$.

\begin{lemma}\label{Lem: the disjoint path for z small}
Let $\varepsilon^{2 \ell}=1$ and let $L(Y_{i,k_1}\cdots Y_{i,k_z})$ be a $U_{\varepsilon}^{\res}({L\mathfrak{sl}_{n+1}})$ Kirillov-Reshetikhin module, where $Y_{i,k_1}\cdots Y_{i,k_z}$ has small values of indices, $i\in I$, $k_t\in\ZZ$, $t\in[1,z]$, $1\leq z< \ell$. If
\begin{align*}
\begin{cases} k_z+2i-k_1<2 \ell, & {\hskip 2em} i\in [1, \lfloor \frac{n+1}2 \rfloor],\\
              k_z+2(n+1-i)-k_1<2 \ell, & {\hskip 2em} i\in (\lfloor \frac{n+1}2 \rfloor, n],
\end{cases}
\end{align*}
then
\[
\sum_{(p_{1},\ldots,p_{z}) \in \overline{\mathscr{P'}}_{(i,k_t)_{1 \leq t \leq z}}} \prod_{t=1}^{z}\mathfrak{m}(p_{t})\quad=\sum_{(p_{1},\ldots,p_{z}) \in \overline{\mathscr{P}}_{(i,k_t)_{1 \leq t \leq z}}} \prod_{t=1}^{z}\mathfrak{m}(p_{t}).
\]
\end{lemma}

\begin{proof}
We prove the lemma for the case of $k_z+2i-k_1< 2 \ell$, $i\in [1, \lfloor \frac{n+1}2 \rfloor ]$, the proof for the case of $k_z+2(n+1-i)-k_1< 2 \ell$, $i\in (\lfloor \frac{n+1}2 \rfloor, n]$ is similar. 

According to Equation $(\ref{Non-overlapping paths})$, $\overline{\mathscr{P}}_{(i,k_{t})_{1\leq t\leq z}}$ is defined as the set of $(p_{1},\ldots,p_{z})$, where $p_{t}\in \mathscr{P}_{i,k_{t}}$, $1\leq t\leq z$, and $(p_{1},\ldots,p_{z})$ is non-overlapping. By gluing the lines $x=0$ and $x=2 \ell$, we obtain a tube (see Definition \ref{D:glue}). In this tube, $\overline{\mathscr{P'}}_{{{(i,k_t)}_{1\leq t\leq z}}}$ is defined as the set of $(p_{1},\ldots,p_{z})$, where $p_{t}\in \mathscr{P}_{i,k_t}$, $1\leq t\leq z$, and $(p_{1},\ldots,p_{z})$ is disjoint, see Equation (\ref{def: disjoint path in a tube}). Since $k_z+2i<k_1+2 \ell$, $i\in [1, \lfloor \frac{n+1}2 \rfloor ]$, the highest path in $\mathscr{P}_{i,k_1}$ and any path in $\mathscr{P}_{i,k_z}$ within this tube are disjoint. Therefore, the condition of $(p_{1},\ldots,p_{z})$ being disjoint is equivalent to being non-overlapping. Thus, 
\[
\overline{\mathscr{P'}}_{{{(i,k_t)}_{1\leq t\leq z}}} = \overline{\mathscr{P}}_{(i,k_{t})_{1\leq t\leq z}}.
\]
Therefore, the conclusion follows.
\end{proof}

Recall that $\hat{I}=I \cup \{0\}$ and we identify $Y_{0,s}$ and $Y_{n+1,s}$ with $1$. For convenience, we denote
\begin{align}
S_{(i,k_t)_{1 \leq t \leq z}}=\sum_{(p_{1},\ldots,p_{z}) \in \overline{\mathscr{P}}_{(i,k_t)_{1 \leq t \leq z}}} \prod_{t=1}^{z}\mathfrak{m}(p_{t}).
\end{align}

\begin{lemma}\label{Lem: dominant monomial for k-r with z great }
Let $\varepsilon^{2 \ell}=1$ and let $L(Y_{i,k_1}\cdots Y_{i,k_z})$ be a $U_{\varepsilon}^{\res}({L\mathfrak{sl}_{n+1}})$ Kirillov-Reshetikhin module, where $Y_{i,k_1}\cdots Y_{i,k_z}$ has small values of indices, $i\in I$, $k_t\in\ZZ$, $t\in[1,z]$, $1\leq z\leq \ell$. If
\begin{align*}
\begin{cases} k_z+2i-k_1\geq 2 \ell, & {\hskip 2em} i\in [1, \lfloor \frac{n+1}2 \rfloor],\\
              k_z+2(n+1-i)-k_1\geq 2 \ell, & {\hskip 2em} i\in (\lfloor \frac{n+1}2 \rfloor, n],
\end{cases}
\end{align*}
then the modules $L(m)$ whose $\varepsilon$-characters are contained in $S_{(i,k_t)_{1 \leq t \leq z}}$ are those for which the highest $l$-weights $m$ can be obtained by a series of translations of paths in $\mathscr{P}_{i,k_j}$ to paths in $\mathscr{P}_{i,k_j+2x \ell}$ with respect to ${\bf PS}(i,k_{\mu+j-1})$, where $1\leq j\leq z-\mu+1$, $x\in[1, \lfloor\frac{k_\mu-k_1+2i}{2 \ell}\rfloor]$, and $\xi \leq \mu\leq z$. Here, $\xi$ is the smallest integer satisfying the inequalities
\[
k_\xi+2i-k_1\geq 2 \ell \quad\text{and}\quad \xi \geq \lceil \frac{z+2}2 \rceil.
\]
\end{lemma}

\begin{proof}
We prove the lemma for the case of $k_z+2i-k_1\geq 2 \ell$, $i\in [1, \lfloor \frac{n+1}2 \rfloor ]$, the proof for the case of $k_z+2(n+1-i)-k_1\geq 2 \ell$, $i\in (\lfloor \frac{n+1}2 \rfloor, n ]$ is similar. To find out all modules $L(m)$ whose $\varepsilon$-characters are contained in $S_{(i,k_t)_{1 \leq t \leq z}}$, we need to find all dominant monomials $m=\mathfrak{m}(p_1)\cdots \mathfrak{m}(p_z)$ in $S_{(i,k_t)_{1 \leq t \leq z}}$ such that the monomials of $\chi_{\varepsilon}(L(m))$ are contained in $S_{(i,k_t)_{1 \leq t \leq z}}$, where $p_t \in \mathscr{P}_{i,k_t}$, $t\in[1,z]$.

By assumption, $L(Y_{i,k_1}\cdots Y_{i,k_z})$ is a Kirillov-Reshetikhin module, where $Y_{i,k_1}\cdots Y_{i,k_z}$ has small values of indices. If $p_1 \in \mathscr{P}_{i,k_1}$ has at least one lower corner, then by the definition of $\overline{\mathscr{P}}_{(i,k_{t})_{1\leq t\leq N}}$ (see Equation $(\ref{Non-overlapping paths})$), the path $p_t \in \mathscr{P}_{i,k_t}$, $t\in(1,z]$, also has at least one lower corner. So $\mathfrak{m}(p_1)\cdots \mathfrak{m}(p_z)$ is not a dominant monomial. Consequently, we have $\mathfrak{m}(p_1)=Y_{i,k_1}$.

Assume that $k_z+2i-k_1\geq 2 \ell$, where $i\in [1, \lfloor \frac{n+1}2 \rfloor]$. Then, there exists an integer $\xi\in[1,z]$ that satisfies the inequalities
\[
k_\xi+2i-k_1\geq 2 \ell \quad\text{and}\quad \xi \geq \lceil \frac{z+2}2 \rceil.
\]
We define $\xi$ as the smallest integer that satisfies these inequalities. For $1<t<\mu$, $\xi \leq \mu\leq z$, we take $\mathfrak{m}(p_t)=Y_{i,k_t}$. That is,
\[
\mathfrak{m}(p_1)\cdots \mathfrak{m}(p_{\mu-1})=Y_{i,k_1}Y_{i,k_2}\cdots Y_{i,k_{\mu-1}}.
\]

{\bf Case 1}. Assume that $\mathfrak{m}(p_{\mu})=Y_{u_0,\eta_0}Y_{i,k_1+2x \ell}^{-1}Y_{v_0,\eta_0}$, where $x\in[1,\lfloor\frac{k_\mu-k_1+2i}{2 \ell}\rfloor]$, $\eta_0\in \mathbb{Z}$, $u_0<v_0\in \hat{I} \cup \{n+1\}$. If $\mu=z$, then we obtain a dominant monomial $m=Y_{i,k_2}\cdots Y_{i,k_{z-1}}Y_{u_0,\eta_0}Y_{v_0,\eta_0}$. 
In fact, $m$ can be obtained by a translation of paths in $\mathscr{P}_{i,k_1}$ to $\mathscr{P}_{i,k_1+2x \ell}$ with respect to ${\bf PS}(i,k_z)$.

Similar to (2) of Lemma \ref{Le: II translation of degree 2}, let $\widetilde{p}_1\in \mathscr{P}_{i,k_1}$ be the path that has exactly one upper corner $(n+1-i,n+1+k_z-2x \ell)$. Let $\widetilde{p}_t$, $t\in[2, z]$, denote the lowest path in $\mathscr{P}_{i,k_t}$ with no upper corners. Then we have $\mathfrak{m}(\widetilde{p}_{1})=Y_{n+1-v_0,n+1+\eta_0-2x \ell}^{-1}Y_{n+1-i,n+1+k_z-2x \ell}Y_{n+1-u_0,n+1+\eta_0-2x \ell}^{-1}$ and for $t\in[2, z]$, $\mathfrak{m}(\widetilde{p}_{t})=Y_{n+1-i,n+1+k_t}^{-1}$. Thus, the paths $\widetilde{p}_{1}, \widetilde{p}_{2}, \ldots, \widetilde{p}_{z}$ are non-overlapping. Therefore, the lowest $l$-weight monomial of the module $L(m)$ is given by
\[
\mathfrak{m}(\widetilde{p}_{1})\cdots \mathfrak{m}(\widetilde{p}_{z})=
Y_{n+1-v_0,n+1+\eta_0-2x \ell}^{-1}Y_{n+1-u_0,n+1+\eta_0-2x \ell}^{-1}Y_{n+1-i,n+1+k_2}^{-1}\cdots Y_{n+1-i,n+1+k_{z-1}}^{-1},
\]
which is in $S_{(i,k_t)_{1 \leq t \leq z}}$. Utilizing a similar approach to that in part (3) of Lemma \ref{Le: II translation of degree 2}, we conclude that the monomials of $\chi_\varepsilon(L(m))$ are contained in $S_{(i,k_t)_{1 \leq t \leq z}}$. If $\mu<z$, then by the definition of $\overline{\mathscr{P}}_{(i,k_{t})_{1\leq t\leq N}}$, there are three possible values for $\mathfrak{m}(p_{\mu+1})$.

Subcase 1.1. Suppose that $\mathfrak{m}(p_{\mu+1})=Y_{u_1,\eta_1}Y_{i,k_2+2x\ell}^{-1}Y_{v_1,\eta_1}$, where $u_1<v_1 \in \hat{I} \cup \{n+1\}$, and $\eta_1\in \mathbb{Z}$. Then by Equation $(\ref{Non-overlapping paths})$, we take $\mathfrak{m}(p_{\mu+t})=Y_{u_t,\eta_t}Y_{i,k_{t+1}+2x\ell}^{-1}Y_{v_t,\eta_t}$, where $u_t<v_t \in \hat{I}\cup \{n+1\}$, $\eta_t\in \mathbb{Z}$, $2\leq t\leq z-\mu$. Therefore, we have
\[
m=Y_{i,k_{z-\mu+2}}Y_{i,k_{z-\mu+3}}\cdots Y_{i,k_{\mu-1}}Y_{u_0,\eta_0}Y_{v_0,\eta_0}Y_{u_1,\eta_1}Y_{v_1,\eta_1}\cdots Y_{u_{z-\mu},\eta_{z-\mu}}Y_{v_{z-\mu},\eta_{z-\mu}}.
\]
In fact, $m$ can be obtained by translations of paths in $\mathscr{P}_{i,k_j}$ to $\mathscr{P}_{i,k_j+2x \ell}$ with respect to ${\bf PS}(i,k_{\mu+j-1})$, respectively, where $j\in[1,z-\mu+1]$.

Similar to (2) of Lemma \ref{Le: II translation of degree 2}, let $\widetilde{p}_t\in \mathscr{P}_{i,k_t}$ for $t \in [1, z-\mu+1]$ be the path that has exactly one upper corner $(n+1-i,n+1+k_{\mu+t-1}-2x \ell)$. Let $\widetilde{p}_t$ for $t\in[z-\mu+2, z]$ denote the lowest path in $\mathscr{P}_{i,k_t}$ with no upper corners. That is, for $t \in [1, z-\mu+1]$, $\mathfrak{m}(\widetilde{p}_{t})=Y_{n+1-v_{t-1},n+1+\eta_{t-1}-2x \ell}^{-1}Y_{n+1-i,n+1+k_{\mu+t-1}-2x \ell}Y_{n+1-u_{t-1}, n+1+\eta_{t-1}-2x \ell}^{-1}$, and for $t\in[z-\mu+2, z]$, $\mathfrak{m}(\widetilde{p}_{t})=Y_{n+1-i,n+1+k_t}^{-1}$. Thus, the paths $\widetilde{p}_{1}, \widetilde{p}_{2}, \ldots, \widetilde{p}_{z}$ are non-overlapping. Therefore, the lowest $l$-weight monomial of the module $L(m)$ is given by 
\begin{align*}
\mathfrak{m}(\widetilde{p}_{1})\cdots \mathfrak{m}(\widetilde{p}_{z})=\left(\prod_{t=0}^{z-\mu}Y_{n+1-v_t,n+1+\eta_t-2x \ell}^{-1}Y_{n+1-u_t,n+1+\eta_t-2x \ell}^{-1}\right)\left(\prod_{t=z-\mu+2}^{\mu-1}Y_{n+1-i,n+1+k_{t}}^{-1}\right),
\end{align*}
which is in $S_{(i,k_t)_{1 \leq t \leq z}}$. Using a method analogous to that in (3) of Lemma \ref{Le: II translation of degree 2}, we find that the monomials of $\chi_\varepsilon(L(m))$ are contained in $S_{(i,k_t)_{1 \leq t \leq z}}$.

Subcase 1.2. Assume that $\mathfrak{m}(p_{\mu+1})=Y_{u_1,\eta_1}Y_{i,k_2+2x_1\ell}^{-1}Y_{v_1,\eta_1}$, where $x_1\in (x,\lfloor\frac{k_\mu-k_1+2i}{2 \ell}\rfloor]$, $u_1<v_1 \in \hat{I} \cup \{n+1\}$, and $\eta_1\in \mathbb{Z}$. Consequently, $\mathfrak{m}(p_1)\cdots \mathfrak{m}(p_{\mu+1})$ is the dominant monomial
\[
m'=Y_{i,k_{3}}Y_{i,k_4}\cdots Y_{i,k_{\mu-1}} Y_{u_0,\eta_0}Y_{v_0,\eta_0} Y_{u_1,\eta_1}Y_{v_1,\eta_1}.
\]
In fact, $m'$ can be obtained by translating paths in $\mathscr{P}_{i,k_1}$ to paths in $\mathscr{P}_{i,k_1+2x \ell}$ with respect to ${\bf PS}(i,k_{\mu})$ and paths in $\mathscr{P}_{i,k_2}$ to paths in $\mathscr{P}_{i,k_2+2x_1 \ell}$ with respect to ${\bf PS}(i,k_{\mu+1})$.

Similar to (2) of Lemma \ref{Le: II translation of degree 2}, let $\widetilde{p}_1\in \mathscr{P}_{i,k_1}$ be the path that has exactly one upper corner $(n+1-i,n+1+k_{\mu}-2x \ell)$. Let $\widetilde{p}_2\in \mathscr{P}_{i,k_2}$ be the path that has exactly one upper corner $(n+1-i,n+1+k_{\mu+1}-2x_1 \ell)$. For $t\in[3, \mu+1]$, let $\widetilde{p}_t$ denote the lowest path in $\mathscr{P}_{i,k_t}$ with no upper corners. That is, we have 
\[\mathfrak{m}(\widetilde{p}_{1})=Y_{n+1-v_0,n+1+\eta_0-2x \ell}^{-1}Y_{n+1-i,n+1+k_{\mu}-2x \ell}Y_{n+1-u_0, n+1+\eta_0-2x \ell}^{-1},
\]
\[
\mathfrak{m}(\widetilde{p}_{2})=Y_{n+1-v_1,n+1+\eta_1-2x_1 \ell}^{-1}Y_{n+1-i,n+1+k_{\mu+1}-2x_1 \ell}Y_{n+1-u_1, n+1+\eta_1-2x_1 \ell}^{-1},
\]
and for $t\in[3, \mu+1]$, $\mathfrak{m}(\widetilde{p}_{t})=Y_{n+1-i,n+1+k_t}^{-1}$. Since $x_1\in (x,\lfloor\frac{k_\mu-k_1+2i}{2 \ell}\rfloor]$, we conclude that the paths $\widetilde{p}_{1}$ and $\widetilde{p}_{2}$ must interest at some points. Thus, the lowest $l$-weight monomial of $L(m')$ is given by 
\begin{align*}
 \mathfrak{m}(\widetilde{p}_{1})\cdots \mathfrak{m}(\widetilde{p}_{\mu+1})=&\big( \prod_{t=3}^{\mu-1} Y_{n+1-i,n+1+k_{t}}^{-1} \big)
 Y_{n+1-v_0,n+1+\eta_0-2x \ell}^{-1} Y_{n+1-u_0,n+1+\eta_0-2x \ell}^{-1} \times \\
 & \times Y_{n+1-v_1,n+1+\eta_1-2x_1 \ell}^{-1}Y_{n+1-u_1,n+1+\eta_1-2x_1 \ell}^{-1}, 
\end{align*}
which is not in $S_{(i,k_t)_{1 \leq t \leq \mu+1}}$.

Therefore, for $1\leq t<\mu$, we have $\mathfrak{m}(p_t)=Y_{i,k_t}$. For $t=\mu$, $\mathfrak{m}(p_{\mu})=Y_{u_0,\eta_0}Y_{i,k_1+2x\ell}^{-1}Y_{v_0,\eta_0}$, where $x\in[1,\lfloor\frac{k_\mu-k_1+2i}{2 \ell}\rfloor]$, $u_0<v_0 \in \hat{I} \cup \{n+1\}$, and $\eta_0\in \mathbb{Z}$. For $t=\mu+1$, $\mathfrak{m}(p_{\mu+1})=Y_{u_1,\eta_1} Y_{i,k_2+2x_1\ell}^{-1}Y_{v_1,\eta_1}$, where $x_1\in (x,\lfloor\frac{k_{\mu}-k_1+2i}{2 \ell}\rfloor]$, $u_1<v_1\in \hat{I} \cup \{n+1\}$, and $\eta_1\in \mathbb{Z}$. For $\mu+1<t\leq z$, we choose appropriate $p_t\in\mathscr{P}_{i,k_t}$ such that $m=\mathfrak{m}(p_1)\cdots \mathfrak{m}(p_z)$ is a dominant monomial. Then, it always holds that the lowest $l$-weight monomial of $L(m)$ is not in $S_{(i,k_t)_{1 \leq t \leq z}}$.

Subcase 1.3. Suppose that there exists $\theta$ such that $k_\theta+2x_1\ell\leq k_{\mu+1}+2i$ and $2<\theta \leq 2\mu-z$, where $x_1\in [x,\lfloor\frac{k_{\mu+1}-k_\theta+2i}{2 \ell}\rfloor]$. We take $\mathfrak{m}(p_{\mu+1})=Y_{u'_1,\eta'_1} Y_{i,k_\theta+2x_1\ell}^{-1}Y_{v'_1,\eta'_1}$, $u'_1<v'_1\in \hat{I} \cup \{n+1\}$, and $\eta'_1\in \mathbb{Z}$. Consequently, $\mathfrak{m}(p_1)\cdots \mathfrak{m}(p_{\mu+1})$ is the dominant monomial
\[
m'=Y_{i,k_{2}}\cdots Y_{i,k_{\theta-1}}Y_{i,k_{\theta+1}}\cdots Y_{i,k_{\mu-1}} Y_{u_0,\eta_0}Y_{v_0,\eta_0} Y_{u'_1,\eta'_1}Y_{v'_1,\eta'_1}.
\]
In fact, $m'$ can be obtained by translating paths in $\mathscr{P}_{i,k_1}$ to paths in $\mathscr{P}_{i,k_1+2x \ell}$ with respect to ${\bf PS}(i,k_{\mu})$ and paths in $\mathscr{P}_{i,k_\theta}$ to $\mathscr{P}_{i,k_\theta+2x_1 \ell}$ with respect to ${\bf PS}(i,k_{\mu+1})$.

Similar to (2) of Lemma \ref{Le: II translation of degree 2}, let $\widetilde{p}_1\in \mathscr{P}_{i,k_1}$ be the path that has exactly one upper corner $(n+1-i,n+1+k_{\mu}-2x \ell)$. Let $\widetilde{p}_{\theta}\in \mathscr{P}_{i,k_{\theta}}$ be the path that has exactly one upper corner $(n+1-i,n+1+k_{\mu+1}-2x_1 \ell)$. For $t\in[2, \mu+1]$ and $t\neq \theta$, let $\widetilde{p}_t$ denote the lowest path in $\mathscr{P}_{i,k_t}$ with no upper corners. That is, we have 
\[\mathfrak{m}(\widetilde{p}_{1})=Y_{n+1-v_0,n+1+\eta_0-2x \ell}^{-1}Y_{n+1-i,n+1+k_{\mu}-2x \ell}Y_{n+1-u_0, n+1+\eta_0-2x \ell}^{-1},
\]
\[
\mathfrak{m}(\widetilde{p}_{\theta})=Y_{n+1-v'_1,n+1+\eta'_1-2x_1 \ell}^{-1}Y_{n+1-i,n+1+k_{\mu+1}-2x_1 \ell}Y_{n+1-u'_1, n+1+\eta'_1-2x_1 \ell}^{-1}.
\]
For $t\in[2, \mu+1]$ and $t\neq \theta$, $\mathfrak{m}(\widetilde{p}_{t})=Y_{n+1-i,n+1+k_t}^{-1}$. Since $\mathfrak{m}(\widetilde{p}_{\theta-1})=Y_{n+1-i,n+1+k_{\theta-1}}^{-1}$, we conclude that the paths $\widetilde{p}_{\theta-1}$ and $\widetilde{p}_{\theta}$ must interest at some points. Thus, the lowest $l$-weight monomial of $L(m')$ is given by 
\begin{align*}
& \mathfrak{m}(\widetilde{p}_{1})\cdots \mathfrak{m}(\widetilde{p}_{\mu+1})= \left( \prod_{t=2}^{\theta-1}  Y_{n+1-i,n+1+k_{t}}^{-1} \right)\left( \prod_{t=\theta+1}^{\mu-1}Y_{n+1-i,n+1+k_{t}}^{-1} \right) \times \\
& \quad \times Y_{n+1-v_0,n+1+\eta_0-2x \ell}^{-1} Y_{n+1-u_0,n+1+\eta_0-2x \ell}^{-1}  Y_{n+1-v'_1,n+1+\eta'_1-2x_1 \ell}^{-1}Y_{n+1-u'_1,n+1+\eta'_1-2x _1 \ell}^{-1},
\end{align*}
which is not in $S_{(i,k_t)_{1 \leq t \leq \mu+1}}$.

Therefore, for $1\leq t<\mu$, we have $\mathfrak{m}(p_t)=Y_{i,k_t}$. For $t=\mu$, $\mathfrak{m}(p_{\mu})=Y_{u_0,\eta_0}Y_{i,k_1+2x \ell}^{-1}Y_{v_0,\eta_0}$, where $x\in[1,\lfloor\frac{k_\mu-k_1+2i}{2 \ell}\rfloor]$, $u_0<v_0 \in \hat{I} \cup \{n+1\}$, and $\eta_0\in \mathbb{Z}$. For $t=\mu+1$, $\mathfrak{m}(p_{\mu+1})=Y_{u'_1,\eta'_1} Y_{i,k_\theta+2x_1\ell}^{-1}Y_{v'_1,\eta'_1}$, where $2<\theta \leq 2\mu-z$, $x_1\in [x,\lfloor\frac{k_{\mu+1}-k_\theta+2i}{2 \ell}\rfloor]$, $u'_1<v'_1\in \hat{I} \cup \{n+1\}$, and $\eta'_1\in \mathbb{Z}$. For $\mu+1<t\leq z$, we choose appropriate $p_t\in\mathscr{P}_{i,k_t}$ such that $m=\mathfrak{m}(p_1)\cdots \mathfrak{m}(p_z)$ is a dominant monomial. Then, it always holds that the lowest $l$-weight monomial of $L(m)$ is not in $S_{(i,k_t)_{1 \leq t \leq z}}$.

{\bf Case 2}. Assume that $\mathfrak{m}(p_{\mu})=Y_{u,\eta}Y_{i,k_s+2x \ell}^{-1}Y_{v,\eta}$, where $2\leq s<\mu$, $x\in[1,\lfloor\frac{k_\mu-k_s+2i}{2 \ell}\rfloor]$, $u<v \in \hat{I} \cup \{n+1\}$, and $\eta\in \mathbb{Z}$. Then $\mathfrak{m}(p_1)\cdots \mathfrak{m}(p_{\mu})$ is the dominant monomial
\[
m''=Y_{i,k_{1}}\cdots Y_{i,k_{s-1}}Y_{i,k_{s+1}}\cdots Y_{i,k_{\mu-1}} Y_{u,\eta}Y_{v,\eta}.
\]
In fact, $m'$ can be obtained by translating paths in $\mathscr{P}_{i,k_s}$ to paths in $\mathscr{P}_{i,k_s+2x \ell}$ with respect to ${\bf PS}(i,k_{\mu})$.

Similar to (2) of Lemma \ref{Le: II translation of degree 2}, for $t\in [1,\mu]$ and $t\neq s$, let $\widetilde{p}_t$ denote the lowest path in $\mathscr{P}_{i,k_t}$ with no upper corners. Let $\widetilde{p}_{s}\in \mathscr{P}_{i,k_{s}}$ be the path that has exactly one upper corner $(n+1-i,n+1+k_{\mu}-2x\ell)$. That is, $\mathfrak{m}(\widetilde{p}_{s})= Y_{n+1-v,n+1+\eta-2x \ell}^{-1}Y_{n+1-i,n+1+k_\mu -2x \ell} Y_{n+1-u,n+1+\eta-2x \ell}^{-1}$. For $t\in[1, \mu]$ and $t\neq s$, $\mathfrak{m}(\widetilde{p}_{t})=Y_{n+1-i,n+1+k_t}^{-1}$. Since $\mathfrak{m}(\widetilde{p}_{s-1})=Y_{n+1-i,n+1+k_{s-1}}^{-1}$, we conclude that the paths $\widetilde{p}_{s-1}$ and $\widetilde{p}_{s}$ must interest at some points. Thus, the lowest $l$-weight monomial of $L(m')$ is given by
\begin{align*}
 \mathfrak{m}(\widetilde{p}_{1})\cdots \mathfrak{m}(\widetilde{p}_{\mu})=\left( \prod_{t=1}^{s-1} Y_{n+1-i,n+1+k_{t}}^{-1} \right) \left( \prod_{t=s+1}^{\mu-1}Y_{n+1-i,n+1+k_{t}}^{-1} \right) Y_{n+1-v,n+1+\eta-2x \ell}^{-1} Y_{n+1-u,n+1+\eta-2x \ell}^{-1},
\end{align*}
which is not in $S_{(i,k_t)_{1 \leq t \leq \mu}}$.

Therefore, for $1\leq t<\mu$, we have $\mathfrak{m}(p_t)=Y_{i,k_t}$. For $t=\mu$, $\mathfrak{m}(p_{\mu})=Y_{u,\eta}Y_{i,k_s+2x \ell}^{-1}Y_{v,\eta}$, where $2\leq s<\mu$, $x\in[1,\lfloor\frac{k_\mu-k_s+2i}{2 \ell}\rfloor]$, $u<v \in \hat{I} \cup \{n+1\}$, and $\eta\in \mathbb{Z}$. For $\mu<t\leq z$, we choose appropriate $p_t\in\mathscr{P}_{i,k_t}$ such that $m=\mathfrak{m}(p_1)\cdots \mathfrak{m}(p_z)$ is a dominant monomial. Then, it always holds that the lowest $l$-weight monomial of $L(m)$ is not in $S_{(i,k_t)_{1 \leq t \leq z}}$.

In conclusion, the modules $L(m)$ whose $\varepsilon$-characters are contained in $S_{(i,k_t)_{1 \leq t \leq z}}$ are those for which the highest $l$-weights $m$ can be obtained by a series of translations of paths in $\mathscr{P}_{i,k_j}$ to paths in $\mathscr{P}_{i,k_j+2x \ell}$ with respect to ${\bf PS}(i,k_{\mu+j-1})$, respectively, where $1\leq j\leq z-\mu+1$, $x\in [1,\lfloor\frac{k_\mu-k_1+2i}{2 \ell}\rfloor$], $\xi \leq \mu\leq z$.

\end{proof}

\subsection{Proof of Theorem \ref{Th: the path description of K-R module}}\label{subsec: prove Th K-R module}

Let $\varepsilon^{2 \ell}=1$ and let $L(Y_{i,k_1}\cdots Y_{i,k_z})$ be a $U_{\varepsilon}^{\res}({L\mathfrak{sl}_{n+1}})$ Kirillov-Reshetikhin module, where $Y_{i,k_1}\cdots Y_{i,k_z}$ has small values of indices, $1\leq z\leq \ell$.

On the one hand, if
\begin{align*}
\begin{cases} k_z+2i-k_1<2 \ell, & {\hskip 2em} i\in [1, \lfloor \frac{n+1}2 \rfloor],\\
              k_z+2(n+1-i)-k_1<2 \ell, & {\hskip 2em} i\in (\lfloor \frac{n+1}2 \rfloor, n],
\end{cases}
\end{align*}
then, by Lemma \ref{Lem: the disjoint path for z small}, there is exactly one dominant monomial $Y_{i,k_1}\cdots Y_{i,k_z}$ in $S_{(i,k_t)_{1 \leq t \leq z}}$. Therefore, we have
\begin{align*}
\chi_{\varepsilon}(L(Y_{i,k_1}\cdots Y_{i,k_z}))=\sum_{(p_{1},\ldots, p_{z}) \in \overline{\mathscr{P}}_{(i,k_t)_{1 \leq t \leq z}}} \prod_{t=1}^{z}\mathfrak{m}(p_{t})=\sum_{(p_{1},\ldots,p_{z}) \in \overline{\mathscr{P'}}_{(i,k_t)_{1 \leq t \leq z}}} \prod_{t=1}^{z}\mathfrak{m}(p_{t}).
\end{align*}

On the other hand, if
\begin{align*}
\begin{cases} k_z+2i-k_1\geq 2 \ell, & {\hskip 2em} i\in [1, \lfloor \frac{n+1}2 \rfloor],\\
              k_z+2(n+1-i)-k_1\geq 2 \ell, & {\hskip 2em} i\in (\lfloor \frac{n+1}2 \rfloor, n],
\end{cases}
\end{align*}
then, by Lemma \ref{Lem: dominant monomial for k-r with z great }, we have found all modules $L(m)$ whose $\varepsilon$-characters are contained in $S_{(i,k_t)_{1 \leq t \leq z}}$. Furthermore, we proved that $m$ can be obtained by a series of translations of paths in $\mathscr{P}_{i,k_j}$ to paths in $\mathscr{P}_{i,k_j+2x \ell}$ with respect to ${\bf PS}(i,k_{\mu+j-1})$, where $1\leq j\leq z-\mu+1$, $x\in[1,\lfloor\frac{k_\mu-k_1+2i}{2 \ell}\rfloor]$, $\xi \leq \mu\leq z$, $\xi$ is the smallest integer that satisfies the inequalities $k_\xi+2i-k_1\geq 2 \ell$ and $\xi \geq \lceil \frac{z+2}2 \rceil$.
Since the module $L(Y_{i,k_1}\cdots Y_{i,k_z})$ is irreducible, we obtain that
\[
\chi_{\varepsilon}(L(Y_{i, k_1}\cdots Y_{i, k_z}))=S_{(i,k_t)_{1 \leq t \leq z}}-\sum\chi_{\varepsilon}(L(m)).
\]
According to $(\ref{def: disjoint path in a tube})$, we know that
$\sum_{(p_{1}, \ldots, p_{z}) \in \overline{\mathscr{P'}}_{(i, k_t)_{1 \leq t \leq z}}} \prod_{t=1}^{z}\mathfrak{m}(p_{t})=S_{(i,k_t)_{1 \leq t \leq z}}-\sum\chi_{\varepsilon}(L(m))$.
Therefore, we have
\begin{align*}
\chi_{\varepsilon}(L(Y_{i, k_1}\cdots Y_{i, k_z}))=\sum_{(p_{1}, \ldots, p_{z}) \in \overline{\mathscr{P'}}_{(i, k_{t})_{1 \leq t \leq z}}} \prod_{t=1}^{z}\mathfrak{m}(p_{t}).
\end{align*}

\section{Proof of Theorem \ref{Th:path description of degree 2}}\label{Sec: prove the th path description of degree 2}
In this section, we prove Theorem \ref{Th:path description of degree 2}.

\subsection{Proof of Theorem \ref{Th:path description of degree 2} (1)}

Let $\varepsilon^{2\ell}=1$ and let $L(Y_{i,k}Y_{j,v})$ be a $U_{\varepsilon}^{\res}({L\mathfrak{sl}_{n+1}})$ simple module, where $i, j \in I$, $k, v \in \ZZ$. Suppose that $|j-i|\equiv |k-v|+1\ (\text{mod}\ 2)$.
\begin{lemma}[{\cite[Theorem 5.8]{A07}}, {\cite[Theorem 2]{C02}}]\label{Le:tensor product q-character}
Let $i, j\in I$, $k, v\in \mathbb{Z}$ such that $|j-i|\equiv |k-v|+1\ (\text{mod}\ 2)$. Then the tensor product $L(Y_{i,k})\otimes L(Y_{j,v})$ is a simple module of $U_{q}(L\mathfrak{sl}_{n+1})$.
\end{lemma}
By Lemma \ref{Le:tensor product q-character}, since $|j-i|\equiv |k-v|+1\ (\text{mod}\ 2)$, we have that $\chi_{q}(L(Y_{i,k}Y_{j,v}))=\chi_{q}(L(Y_{i,k}))\chi_{q}(L(Y_{j,v}))$.
Take any path $p \in \mathscr{P}_{i,k}$ and $p'\in \mathscr{P}_{j,v}$. Suppose that at least one of $p, p'$ is not the highest path. Let $C$ be an upper (resp. lower) corner of $p$ at some position $(a, s)$ and let $C'$ be a lower (resp. upper) corner of $p'$ at some position $(a',s')$. For $a=a'$, we have that $s \not\equiv s' \pmod {2 \ell}$. Therefore when $q \mapsto \varepsilon$, $\mathfrak{m}(p)\mathfrak{m}(p')$ is not a dominant monomial. It follows that when sending $q \mapsto \varepsilon$ in the $q$-character of $L(Y_{i,k}Y_{j,v})$, the only dominant monomial is the highest weight monomial $Y_{i,k}Y_{j,v}$. Thus, we have
\[
\chi_{\varepsilon}(L(Y_{i,k}Y_{j,v}))=\left(\sum_{p \in \mathscr{P}_{i,k}} \mathfrak{m}(p)\right) \left(\sum_{p \in \mathscr{P}_{j,v}} \mathfrak{m}(p)\right)=\chi_{\varepsilon}(L(Y_{i,k}))\chi_{\varepsilon}(L(Y_{j,v})).
\]
Theorem \ref{Th:path description of degree 2} (1) is proved.

For convenience, in the following, we identify a path $p$ with the monomial $\mathfrak{m}(p)$. Moreover, we define the product of paths $p$ and $\widetilde{p}$ as the product of $\mathfrak{m}(p)$ and $\mathfrak{m}(\widetilde{p})$. Recall that in Equation (\ref{Eq:abbreviation for non-overlapping paths}), we denote $S_{(i,k)(j,v)} = \sum_{(p_{1},p_{2}) \in \overline{\mathscr{P}}_{((i,k),(j,v))}}\mathfrak{m}(p_1)\mathfrak{m}(p_2)$.
\subsection{Strategy of proof of Theorem \ref{Th:path description of degree 2} (2)}
From now on to the end of Section \ref{Sec: prove the th path description of degree 2}, we assume that $\varepsilon^{2 \ell}=1$, $L(Y_{i,k}Y_{j,v})$ is a simple module of $U_{\varepsilon}^{\res}({L\mathfrak{sl}_{n+1}})$, and $|j-i| \equiv |k-v| \pmod {2}$.

For a monomial $\mathfrak{m}(p)\mathfrak{m}(p')$ corresponding to a pair of paths $p$, $p'$, where $p$ is strictly above $p'$, we say that a monomial $m$ is obtained from $\mathfrak{m}(p)\mathfrak{m}(p')$ by raising and lowering moves if $m=\mathfrak{m}(p'')\mathfrak{m}(p''')$, where $p''$ is obtained from $p$ by several lowering moves of width $\ell$, $p'''$ is obtained from $p'$ by several raising moves of width $\ell$, see Section \ref{subsec:lowdef}.

Firstly, we prove that all dominant monomials except the highest $l$-weight monomial in $S_{(i, k)(j, \overline{v})}$ can be obtained by path translations or raising and lowering moves from dominant monomials obtained by translations of paths. Then we describe dominant monomials $m$ in $S_{(i, k)(j, \overline{v})}$ such that all monomials of $\chi_{\varepsilon}(L(m))$ are contained in $S_{(i, k)(j, \overline{v})}$. We prove that these dominant monomials are exactly all dominant monomials except the highest $l$-weight monomial in $S_{(i, k)(j, \overline{v})}$. The $\varepsilon$-character of $L(Y_{i,k}Y_{j,v})$ is obtained from $S_{(i, k)(j, \overline{v})}$ by removing all these $\chi_{\varepsilon}(L(m))$.

\subsection{Preparation for proving Theorem \ref{Th:path description of degree 2} (2)}
Since $|j-i| \equiv |k-v| \pmod {2}$, there is some $\mathfrak{a} \in \{0,1\}$ such that $(i,k),(j,v)\in \mathcal{X}_{\mathfrak{a}}$. The dominant monomials in $S_{(i, k)(j, v)}$ can be obtained by raising and lowering moves of width $\ell$ (see Section \ref{subsec:lowdef}) from dominant monomials which are obtained by translations of paths. To clarify the description of these dominant monomials and simplify the calculation of their multiplicities, we introduce the following sets of paths.  

Recall that (below Formula (\ref{path})), for a path $((a, b), (a+1, b+1), \ldots, (a+r, b+r))$, $r \in \ZZ_{\ge 1}$, we also write it as $((a, b), (a+r, b+r))$. Similarly, for a path $((a, b), (a+1, b-1), \ldots, (a+r, b-r))$, $r \in \ZZ_{\ge 1}$, we also write it as $((a, b), (a+r, b-r))$. Also we use the notation $P \uplus P'$ to concatenate paths in two sets of paths $P, P'$ in (\ref{eq:concatenation of two sets of paths}). 

\begin{figure}
\resizebox{0.6\width}{0.6\height}{
\begin{minipage}[b]{0.2\linewidth}
\centerline{
\begin{tikzpicture}
\begin{scope}[thick, every node/.style={sloped,allow upside down}]
\draw (-0.5,5.5)--node {\midarrow}(0,6)--node {\midarrow}(1,7)--node {\midarrow}(2,8)--node {\midarrow}(3,9);
\draw (0,6)--node {\midarrow}(0.5,5.5);
\draw (1,7)--node {\midarrow}(1.5,6.5);
\draw (2,8)--node {\midarrow}(2.5,7.5);
\draw (3,9)--node {\midarrow}(3.5,8.5);
\draw [fill] (-0.5,5.5) circle (1pt) (0,6) circle (1pt) (1,7) circle (1pt) (3,9) circle (1pt) (0.5,5.5) circle (1pt) (1.5,6.5) circle (1pt) (3.5,8.5) circle (1pt) (2,8) circle (1pt) (2.5,7.5) circle (1pt);
\end{scope}
\node [left] at (-0.5,5.5) {\tiny $A~(0,i+k)$};
\node [left] at (0,6) {\tiny $(\sigma_0, \varsigma_0)$};
\node [left] at (1,7) {\tiny $(\sigma_2, \varsigma_2)$};
\node [left] at (2,8) {$\cdots$};
\node [left] at (3,9) {\tiny $(\sigma_{2r}, \varsigma_{2r})$};
\node [right] at (0.5,5.5) {\tiny $A_1~(\sigma_1, \varsigma_1)$};
\node [right] at (1.5,6.5) {\tiny $A_2~(\sigma_3, \varsigma_3)$};
\node [right] at (2.5,7.5) {$\cdots$};
\node [right] at (3.5,8.5) {\tiny $A_{r+1}$};
\node at (0.5,5)  {$(a)$};
\end{tikzpicture}}
\end{minipage}
\begin{minipage}[b]{0.8\linewidth}
\centerline{
\begin{tikzpicture}
\begin{scope}[thick, every node/.style={sloped,allow upside down}]
\draw (0,9)--node {\midarrow}(1,8)--node {\midarrow}(2,7)--node {\midarrow}(3,6)--node {\midarrow}(3.5,5.5);
\draw (-0.5,8.5)--node {\midarrow}(0,9);
\draw (0.5,7.5)--node {\midarrow}(1,8);
\draw (1.5,6.5)--node {\midarrow}(2,7);
\draw (2.5,5.5)--node {\midarrow}(3,6);
\draw [fill] (-0.5,8.5) circle (1pt) (0,9) circle (1pt) (1,8) circle (1pt) (3,6) circle (1pt) (3.5,5.5) circle (1pt) (0.5,7.5) circle (1pt) (2.5,5.5) circle (1pt) (1.5,6.5) circle (1pt) (2,7) circle (1pt);
\end{scope}
\node [left] at (-0.5,8.5) {\tiny $B_1~(\sigma_1, \varsigma_1)$};
\node [left] at (0.5,7.5) {\tiny $B_2~(\sigma_3, \varsigma_3)$};
\node [left] at (1.5,6.5) {$\cdots$};
\node [left] at (2.5,5.5) {\tiny$B_{r+1}~(\sigma_{2r+1}, \varsigma_{2r+1})$};
\node [right] at (0,9) {\tiny$(\sigma_0, \varsigma_0)$};
\node [right] at (1,8) {\tiny$(\sigma_2, \varsigma_2)$};
\node [right] at (2,7) {$\cdots$};
\node [right] at (3,6) {\tiny$(\sigma_{2r}, \varsigma_{2r})$};
\node [right] at (3.5,5.5) {\tiny$B~(n+1, n+1-i+k)$};
\node at (0.5,5)  {$(b)$};
\end{tikzpicture}}
\end{minipage}
\begin{minipage}[b]{0.2\linewidth}
\centerline{
\begin{tikzpicture}
\begin{scope}[thick, every node/.style={sloped,allow upside down}]
\draw (-0.5,7.9)--node {\midarrow}(0.3,8.7)--node {\midarrow}(1.1,9.5)--node {\midarrow}(1.9,10.3)--node {\midarrow}(2.7,9.5)--node {\midarrow}(3.5,8.7)--node {\midarrow}(4.3,7.9);
\draw (-0.5,7.9)--node {\midarrow}(0.3,7.1)--node {\midarrow}(1.1,6.3)--node {\midarrow}(1.9,5.5);
\draw (0.3,8.7)--node {\midarrow}(1.1,7.9)--node {\midarrow}(1.9,7.1)--node {\midarrow}(2.7,6.3);
\draw (1.1,9.5)--node {\midarrow}(1.9,8.7)--node {\midarrow}(2.7,7.9)--node {\midarrow}(3.5,7.1);
\draw (0.3,7.1)--node {\midarrow}(1.1,7.9)--node {\midarrow}(1.9,8.7)--node {\midarrow}(2.7,9.5);
\draw (1.1,6.3)--node {\midarrow}(1.9,7.1)--node {\midarrow}(2.7,7.9)--node {\midarrow}(3.5,8.7);
\draw (1.9,5.5)--node {\midarrow}(2.7,6.3)--node {\midarrow}(3.5,7.1)--node {\midarrow}(4.3,7.9);
\draw [fill] (-0.5,7.9) circle (1pt) (0.3,8.7) circle (1pt) (1.1,9.5) circle (1pt) (1.9,10.3) circle (1pt) (2.7,9.5) circle (1pt) (3.5,8.7) circle (1pt) (4.3,7.9) circle (1pt);
\end{scope}
\node [left] at (-0.5,7.9) {\tiny $C_1(\sigma_0, \varsigma_0)$};
\node [left] at (0.3,8.7) {\tiny $C_2$};
\node [left] at (1.1,9.5) {\tiny $C_3$};
\node [left] at (1.9,10.3) {\tiny $C_4$};
\node [right] at (2.7,9.5) {\tiny $C_5$};
\node [right] at (3.5,8.7) {\tiny $C_6$};
\node [right] at (4.3,7.9) {\tiny $C_7(\sigma_6, \varsigma_6)$};
\node at (1,5.5)  {$(c)$};
\end{tikzpicture}}
\end{minipage}}
\caption{(a) The set $\mathscr{G}_{\sigma_0, \varsigma_0, r, \gamma}(i,k)$ consists of all paths from $A$ to $A_1$, $A_2$, $\ldots$, $A_{r+1}$, respectively. In Figure \ref{figure grid 1}, the $CF_1F_2F_3F_4F_1$ area corresponds to this set. (b) The set $\mathscr{G}^{\sigma_0, \varsigma_0, r,\gamma}(i,k)$ consists of all paths from $B_1,B_2,\ldots,B_{r+1}$ to $B$, respectively. In Figure \ref{figure grid 1}, the $F_6F_7F_3F_5F_6D$ area corresponds to this set. (c) The set $\mathfrak{G}_{\sigma_0, \varsigma_0, r, r'}(i,k)$, where $r=3,r'=3$, consists of all paths from $C_1,C_2,C_3,C_4$ to $C_4,C_5,C_6,C_7$, respectively. In Figure \ref{figure grid 1}, the $F_4F_3F_7F_8F_4$ area corresponds to this set.} 
\label{Fig: uplus}
\end{figure}
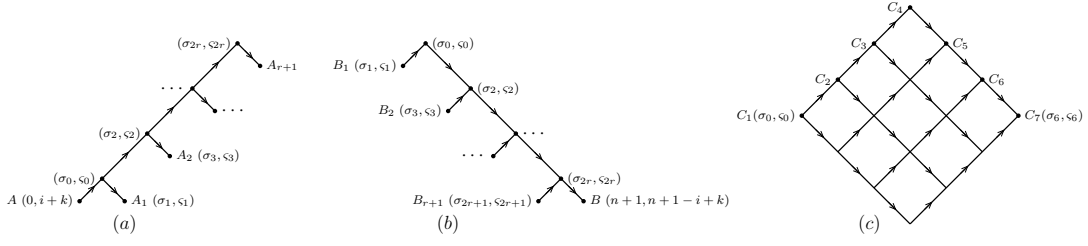

For $(\sigma_0, \varsigma_0)\in p$, $p\in \mathscr{P}_{i,k}$, $r, r', \gamma \in \mathbb{Z}_{\geq 0}$, we denote
\begin{align*}
\mathscr{G}_{\sigma_0, \varsigma_0, r, \gamma}(i,k)=\{ & ((0,i+k), (\sigma_0,\varsigma_{0}))\}\\
&\uplus \{((\sigma_0,\varsigma_0),(\sigma_2,\varsigma_2), \ldots, (\sigma_{2x},\varsigma_{2x}) ): x\in [0,r], \sigma_{2t}-\sigma_{2t-2}= \ell,\\
&\quad \quad \varsigma_{2t}-\varsigma_{2t-2}=- \ell, 1\leq t\leq x \}\\
&\uplus \{((\sigma_{2t},\varsigma_{2t}),(\sigma_{2t+1},\varsigma_{2t+1})):
\sigma_{2t+1}-\sigma_{2t}= \gamma, \varsigma_{2t+1}-\varsigma_{2t}=\gamma, 0\leq t\leq r\},
\end{align*}
where $\varsigma_{0}=i+k-\sigma_0$, see Figure \ref{Fig: uplus} (a) (note that the $y$-coordinate is decreasing when going up, see Figure \ref{F: figure 1});

\begin{align*}
\mathscr{G}^{\sigma_0, \varsigma_0, r,\gamma}(i,k)=\{ & ((\sigma_{2t+1},\varsigma_{2t+1}),(\sigma_{2t},\varsigma_{2t})): \sigma_{2t+1}-\sigma_{2t}= -\gamma, \varsigma_{2t+1}-\varsigma_{2t}=\gamma, 0\leq t\leq r\} \\
&\uplus\{((\sigma_{2x},\varsigma_{2x}), (\sigma_{2x+2},\varsigma_{2x+2}), \ldots, (\sigma_{2r},\varsigma_{2r})):  x\in [0,r], \sigma_{2t}-\sigma_{2t-2}= \ell, \\
&\quad\quad \varsigma_{2t}-\varsigma_{2t-2}= \ell, x\leq t\leq r\} \\
&\uplus\{((\sigma_{2r},\varsigma_{2r}), (n+1,n+1-i+k))\},
\end{align*}
see Figure \ref{Fig: uplus} (b); and
\begin{align*}
\mathfrak{G}_{\sigma_0, \varsigma_0, r, r'}(i,k)=\{ & ((\sigma_x,\varsigma_{x}), (\sigma_{x+1}, \varsigma_{x+1}), (\sigma_{x+ 2}, \varsigma_{x+2}), \ldots, (\sigma_y,\varsigma_{y})): x\in[0,r], y\in[r,r+r'], \\
&\sigma_x= \sigma_0+x \ell, \varsigma_{x}=\varsigma_0-x \ell,\sigma_y= \sigma_0+y \ell, \varsigma_{y}=\varsigma_0+(y-2r) \ell, \\
& \sigma_t-\sigma_{t-1}=\ell \text{ and } \varsigma_t-\varsigma_{t-1} \in \{ \ell, -\ell\}, t \in [x+1, y] \},
\end{align*}
see Figure \ref{Fig: uplus} (c).

By Remark \ref{Re:non-snake module}, since $\varepsilon^{2\ell}=1$ and $|j-i|\equiv |k-v|\ (\text{mod}\ 2)$, the monomial $Y_{i,k}Y_{j,v}$ can be converted to $Y_{i,k}Y_{j,\overline{v}}$, $k <\overline{v}$, which has small values of indices (see Definition \ref{def:small values of indices}), and $L(Y_{i, k}Y_{j, \overline{v}})$ is a snake module.

\begin{lemma}\label{Le:dominant is obtained by path translations or pairs of lowering and raising moves}
If $m \neq Y_{i,k}Y_{j,\overline{v}}$ in $S_{(i, k)(j, \overline{v})}$ is a dominant monomial, then $m$ can be obtained by path translations or raising and lowering moves of width $\ell$ from dominant monomials obtained by translations of paths.
\end{lemma}

\begin{proof}
Since $Y_{i,k}Y_{j,\overline{v}}$ has small values of indices, there is no translation of paths in $\mathscr{P}_{j,\overline{v}}$ to paths in $\mathscr{P}_{j,v'}$ with respect to ${\bf PS}(i,k)$ for any $v'<\overline{v}$. If $m\in S_{(i, k)(j, \overline{v})}$, then $m=\mathfrak{m}(p)\mathfrak{m}(\widetilde{p})$ for some $p\in \mathscr{P}_{i, k}$, $\widetilde{p}\in \mathscr{P}_{j, \overline{v}}$. Based on the values of $|j-i|+h(i,j)$, the proof of the lemma is divided into three cases. Without loss of generality, we may assume that $i\leq j$.

For convenience, we denote $\gamma=\frac{(|j-i|+h(i,j))\, (\text{mod} 2 \ell)}2$, $\gamma'=\ell-\gamma$, $\gamma''=\frac{-|j-i|+h(i,j)}2$, and $\gamma_0=\ell-\gamma''$. Since $Y_{i,k}Y_{j,\overline{v}}$ has small values of indices, we have that $\gamma,\gamma'' \in \ZZ$, $0\leq \gamma<\ell$, and $1\leq \gamma''\leq \ell$. Moreover, $\gamma', \gamma_0 \in \ZZ_{\geq 0}$.

\textbf{Case 1}. $|j-i|+h(i,j)< 2 \ell$. Assume that $i=a \ell+a_0$ for some $a, a_0\in \mathbb{Z}_{\geq 0}$, $0\leq a_0< \ell$, and $n+1-i=a' \ell+a'_0$ for some $a', a_0'\in \mathbb{Z}_{\geq 0}$, $0\leq a'_0< \ell$. Let $y_0=i+k-a_0$ and
\begin{align*}
{\bf G}(i,k)=&\{((0,i+k),(a_0,y_0))\}\uplus \mathfrak{G}_{a_0,y_0,a,a'}(i,k)\\
&\uplus \{((n+1-a'_0, n+1-i+k-a'_0),(n+1,n+1-i+k))\} \subseteq \mathscr{P}_{i, k}.
\end{align*}
For example, in Figure \ref{figure grid 1}, ${\bf G}(i,k)$ is the set of all black or red paths from $A$ to $B$ and the sizes of these square boxes are $\sqrt{2}\ell\times \sqrt{2}\ell$.
By the definition of $r(i,j)$ in $(\ref{eq:def of rij})$, if $r(i,j)=0$, then there is no translation of paths. Furthermore, if $r(i, j) > 0$, there exist $b_0, b_0' \in \ZZ_{\ge 0}$ such that $j-\gamma_0=b \ell+b_0$ and $n+1-j-\gamma'=b \ell+b'_0$, where $b=r(i,j)-1$. In Figure \ref{figure grid 1}, we have $r(i,j)=4$. Denote $E=(j,\overline{v})$. We have $|E_1E_2|=b_0$, $|E_3E|=\gamma_0$, $|EE_4|=\gamma'$, and $|E_5E_6|=b'_0$. Let
\begin{align*}
{\bf G}(j,\overline{v})=\mathscr{G}_{b_0, j+\overline{v}-b_0, b, \gamma'}(j,\overline{v})\uplus \mathfrak{G}_{b_0+\gamma',j+\overline{v}-b_0+ \gamma', b, b}(j,\overline{v}) \uplus \mathscr{G}^{j+\gamma', \overline{v}+\gamma', b, \gamma_0 }(j,\overline{v}) \subseteq \mathscr{P}_{j, \overline{v}}.
\end{align*}
For example, in Figure \ref{figure grid 1}, ${\bf G}(j,\overline{v})$ is the set of all red or black paths from $C$ to $D$ and the sizes of these square boxes are $\sqrt{2}\ell\times \sqrt{2}\ell$.

Subcase 1.1. Assume that $p\notin{\bf G}(i,k)$. Then there is at least one lower or upper corner that is an interior point in either a $\sqrt{2}\ell \times \sqrt{2}\ell$ box or a $\sqrt{2}a_0 \times \sqrt{2}\ell$ box.
In the following, we prove that if $p\notin{\bf G}(i,k)$, then there exists no path $\widetilde{p}\in \mathscr{P}_{j, \overline{v}}$ such that $\mathfrak{m}(p)\mathfrak{m}(\widetilde{p})$ in $S_{(i, k)(j, \overline{v})}$ is a dominant monomial.

Let us consider a pair of points $(c_1,d_1)$ and $(c_2,d_2)$, $c_2<c_1 \in I$, $d_1, d_2 \in \ZZ$. Suppose that $(c_1,d_1)$ is an interior point in either a $\sqrt{2}\ell \times \sqrt{2}\ell$ or $\sqrt{2}a_0 \times \sqrt{2}\ell$ box, and $(c_2,d_2)$ lies on the boundary of a $\sqrt{2}\ell \times \sqrt{2}\ell$ or $\sqrt{2}a_0 \times \sqrt{2}\ell$ box. Additionally, let $Y_{c_2, d_2}Y_{c_1, d_1}^{-1}$ or $Y_{c_2, d_2}^{-1}Y_{c_1, d_1}$ be a subpath of $p$. Without loss of generality, we may assume that $(c_1, d_1)\in C^{-}_p$. Suppose that there exists a point $(c_1, d_1+2z_1 \ell)\in C^{+}_{\widetilde{p}}$ for some $z_1 \in\ZZ_{>0}$, where $\widetilde{p}\in \mathscr{P}_{j,\overline{v}}$. Then $Y_{c_3,d_3}^{-1}Y_{c_1,d_1+2z_1 \ell}$, $c_3<c_1 \in I$, $d_3 \in \ZZ$, is a subpath of $\widetilde{p}$. If $(c_3,d_3-2z_3 \ell) \not\in C^{+}_p$ for any $z_3 \in\ZZ_{>0}$, then $\mathfrak{m}(p)\mathfrak{m}(\widetilde{p})$ is not a dominant monomial. Conversely, if $(c_3, d_3-2z_3 \ell) \in C^{+}_p$ for some $z_3 \in\ZZ_{>0}$, then $c_3\leq c_2$. If $c_3=c_2$, then $d_3 \not\equiv d_2\ (\text{mod}\ 2 \ell)$. That is, $\mathfrak{m}(p)\mathfrak{m}(\widetilde{p})$ is not a dominant monomial. If $c_3<c_2$, then the points $(c_2,d_2)$ and $(c_3,d_3-2z_3 \ell)$ are in $C^{+}_p$ and there must exist a point $(c_4,d_4)\in C^{-}_p$, where $c_3<c_4<c_2\in I$, $d_4 \in \ZZ$. Since $Y_{c_3,d_3}^{-1}Y_{c_1,d_1+2z_1 \ell}$ is a subpath of $\widetilde{p}$, it follows that $\mathfrak{m}(p)\mathfrak{m}(\widetilde{p})$ is not a dominant monomial. Therefore, the conclusion is valid.

Subcase 1.2. Suppose that $p \in {\bf G}(i,k)$. Then $p=Y_{i,k}$ or $p$ has at least one lower corner. In the following, we prove that if $p \in {\bf G}(i,k)$, then there exists no path $\widetilde{p}\in \mathscr{P}_{j, \overline{v}}\setminus{\bf G}(j,\overline{v})$ such that $\mathfrak{m}(p)\mathfrak{m}(\widetilde{p}) \in S_{(i, k)(j, \overline{v})}$ is a dominant monomial.

If $p=Y_{i,k}$ and there exists some $z\in \ZZ_{>0}$ such that $\widetilde{p}\in \mathscr{P}_{j, \overline{v}}$ has exactly one lower corner $(i,k+2z \ell)$, then, by the Definition of ${\bf G}(j,\overline{v})$, we conclude that $\widetilde{p}\in {\bf G}(j,\overline{v})$. If $p \in {\bf G}(i,k)$ has at least one lower corner $(c_1,d_1)$, $c_1\in I$, $d_1\in \ZZ$, and $\widetilde{p}\in \mathscr{P}_{j, \overline{v}}$ has the upper corner $(c_1,d_1+2z_1 \ell)$ for some $z_1 \in \ZZ_{>0}$, then $Y_{c_0, d_0}^{-1}Y_{c_1,d_1+2z_1 \ell}Y_{c_2,d_2}^{-1}$ is a subpath of $\widetilde{p}$, where $c_0<c_1<c_2 \in \hat{I}\cup \{n+1\}$, $d_0,d_2\in \ZZ$. To obtain a dominant monomial, we need to remove the monomial $Y_{c_0, d_0}^{-1}Y_{c_2,d_2}^{-1}$. If for any $z_0,z_2 \in \ZZ_{>0}$, $(c_0, d_0-2z_0 \ell)$ and $(c_2,d_2-2z_2 \ell)$ are not in $C^{+}_{p}$, then $\mathfrak{m}(p)\mathfrak{m}(\widetilde{p})$ is not a dominant monomial. If for some $z_0,z_2 \in \ZZ_{>0}$, $(c_0, d_0-2z_0 \ell)$ and $(c_2,d_2-2z_2 \ell)$ are in $C^{+}_{p}$, then $\widetilde{p}\in {\bf G}(j,\overline{v})$. Therefore, the conclusion is valid.

\begin{figure}
\centering
\resizebox{0.55\width}{0.55\height}{
\begin{minipage}[b]{0.7\textwidth}
\begin{tikzpicture}
\begin{scope}[thick, every node/.style={sloped,allow upside down}]
\draw (-0.5,13.5)--node {\midarrow}(0.5,14.5);
\draw (2,16)--node {\midarrow}(3.5,17.5)--node {\midarrow}(5,19)--node {\midarrow}(6.5,17.5)--node {\midarrow}(8,16);
\draw (-0.5,13.5)--(6,7)--(11.5,12.5)--(9.5,14.5);
\draw [dotted,thick](0.5,14.5)--node {\midarrow}(2,16);
\draw [dotted,thick](8,16)--node {\midarrow}(9.5,14.5);
\draw [dotted,thick](9.5,14.5)--node {\midarrow}(11.5,12.5);
\draw [black](3.5,17.5)--node {\midarrow}(5,16)--node {\midarrow}(6.5,14.5)--node {\midarrow}(8,13)--node {\midarrow}(9.5,11.5);
\draw [black](2,16)--node {\midarrow}(3.5,14.5)--node {\midarrow}(5,13)--node {\midarrow}(6.5,11.5)--node {\midarrow}(8,10);
\draw [black](3.5,14.5)--node {\midarrow}(5,16)--node {\midarrow}(6.5,17.5);
\draw [black](2,10)--node {\midarrow}(3.5, 11.5)--node {\midarrow}(5,13);
\draw [black](3.5,8.5)--node {\midarrow}(5,10);
\draw [black](5,10)--node {\midarrow}(6.5,8.5);
\draw [black](0.5,14.5)--node {\midarrow}(2,13)--node {\midarrow}(3.5,14.5);
\draw [black](7.5,15.5)--(8,16);
\draw [gray](9,14)--(9.5,14.5);
\draw [gray](10.5,12.5)--(11,13);
\draw [gray](0.5,11.5)--node{\midarrow}(2,13);
\node [above] at (5,19) {$(i,k)$};
\draw [fill] (5,19) circle (1pt);
\draw [green](-0.5,11.55)--(0,12.05);
\draw [green,dotted](0,12.05)--node {\midarrow}(1.5,13.55);
\draw [green](1.5,13.55)--node {\midarrow}(3,15.05)--node {\midarrow}(4.5,16.55)--(5.5,17.55) --(6,17.05)--node {\midarrow}(7.5, 15.55)-- node {\midarrow}(9,14.05);
\draw [green](10.5,12.55)--node {\midarrow}(11.5,11.55)--(5.5,5.55)--(-0.5,11.55);
\draw [red,dotted,thick](9,14)--node {\midarrow}(10.5,12.5);
\draw [red](-0.5,11.5)--node {\midarrow}(0,12);
\draw [red](1.5,13.5)--node {\midarrow}(3,15)--node {\midarrow}(4.5,16.5)--node {\midarrow}(5,16)--node {\midarrow}(6,17)--node {\midarrow}(7.5,15.5)--node {\midarrow}(9,14);
\draw [red](10.5,12.5)--node {\midarrow}(11.5,11.5);
\draw [red,dotted,thick](0,12)--node {\midarrow}(1.5,13.5);
\draw [red,dotted,thick](9,14)--node {\midarrow}(10.5,12.5);
\draw [red](1.5,13.5)--node {\midarrow}(3,15)--node {\midarrow}(5,13)--node {\midarrow}(7.5,15.5)--node {\midarrow}(9,14);
\draw [red](1.5,13.5)--node {\midarrow}(3.5, 11.5)--node {\midarrow}(5,10)--node {\midarrow}(6.5,11.5)--node {\midarrow}(9,14);
\draw[red](-0.5,11.5)--node {\midarrow}(0,12)--node {\midarrow}(2,10)--node {\midarrow}(3.5,8.5)--node {\midarrow}(5,7)--node {\midarrow}(6.5,8.5)--node {\midarrow}(8,10)--node {\midarrow}(10.5,12.5)--node {\midarrow}(11.5,11.5);
\draw [red](6.5,14.5)--node {\midarrow}(7.5,15.5);
\draw [red,dotted,thick](0,12)--node {\midarrow}(1.5,13.5);
\draw [red,dotted,thick](9,14)--node {\midarrow}(10.5,12.5);
\draw [fill](5,16) circle (2pt) (5,13) circle (2pt)(5,10) circle (2pt)(5,7) circle (2pt) (5.5,17.5) circle (1pt);
\draw [fill](-0.5,13.5) circle (1.5pt) (-0.5,11.5) circle (1.5pt)(11.5,12.5) circle (1.5pt)(11.5,11.5) circle (1.5pt);

\draw [blue, thick] (-0.5,17.5)--(11.5,17.5);
\draw [fill,blue](-0.5,17.5) circle (1pt) (0,17.5) circle (1pt) (1.5,17.5) circle (1pt) (3,17.5) circle (1pt) (4.5,17.5) circle (1pt) (6,17.5) circle (1pt)  (7.5,17.5) circle (1pt) (9,17.5) circle (1pt)  (10.5,17.5) circle (1pt) (11.5,17.5) circle (1pt);
\draw [fill] (0,12) circle (1pt) (4.5,16.5) circle (1pt)   (0.5,11.5) circle (1pt)  (6,17)  circle (1pt) (10.5,12.5)  circle (1pt) (9.5,11.5) circle (1pt);
\end{scope}
\node [above] at (5.4,17.5) {$(j,\overline{v})$};
\node [left] at (4.4,18.5) {$\sqrt{2}\ell$};
\node [left] at (2.9,17) {$\sqrt{2}\ell$};
\node [right] at (5.8,18.4) {$\sqrt{2}\ell$};
\node [right] at (7.3,16.9) {$\sqrt{2}\ell$};
\node at (5.8,17.3) {\tiny $\sqrt{2}\gamma'$};
\node at (4,17) {\scriptsize $\sqrt{2}\gamma$};
\node at (4.9,17.2) {\tiny $\sqrt{2}\gamma_0$};
\node [left] at (2.3,14.5) {$\sqrt{2}\ell$};
\node [left] at (1.2,15) {$\cdots$};
\node [right] at (8.3,14.8) {$\sqrt{2}\ell$};
\node [right] at (9,15) {$\cdots$};
\node [left] at (1,13) {$\cdots$};
\node [right] at (10,13.1) {$\cdots$};
\node [left] at (-0.5,13.5) {$A$};
\node [left] at (-0.5,11.5) {$C$};
\node [right] at (11.5,12.5) {$B$};
\node [right] at (11.5,11.5) {$D$};
\node[right] at (5,16) {$(i,k'_1)$};
\node[right] at (5,13) {$(i,k'_2)$};
\node[right] at (5,7) {\tiny$F_8 (i,k'_{r(i,j)})$};
\node [above] at (-0.5,17.5) {\scriptsize $E_1$};
\node [above] at (0,17.5) {\scriptsize $E_2$};
\node [above] at (4.5,17.5) {\scriptsize $E_3$};
\node [above] at (6,17.5)   {\scriptsize $E_4$};
\node [above] at (10.5,17.5)  {\scriptsize $E_5$};
\node [above] at (11.5,17.5) {\scriptsize $E_6$};
\node [left] at (0,12) {\tiny$F_1$};
\node[left] at (4.5,16.5) {\tiny$F_2$};
\node[above] at (5,16) {\tiny$F_3$};
\node [left] at (0.5,11.5) {\tiny$F_4$};
\node [right] at (6,17) {\tiny$F_5$};
\node [left] at (10.5,12.5) {\tiny$F_6$};
\node [left] at (9.5,11.5) {\tiny$F_7$};
\end{tikzpicture}
\end{minipage}}
\caption{$|j-i|+h(i,j)<2 \ell$. The black and green rectangles are ${\bf P}_{i,k}$ and ${\bf P}_{j,\overline{v}}$, respectively. The product of $Y_{i,k}$ and any one of the red paths is a dominant monomial, which can be obtained by translating paths in $\mathscr{P}_{i,k}$ to paths in $\mathscr{P}_{i,k'_t}$ with respect to ${\bf PS}(j,v)$, where $(i,k'_t)$ is one of the black bullet points.}\label{figure grid 1}
\end{figure}

Subcase 1.3. Assume that $p \in {\bf G}(i,k)$ and there exists a path $\widetilde{p}\in {\bf G}(j,\overline{v})$ such that $\mathfrak{m}(p)\mathfrak{m}(\widetilde{p})$ in $S_{(i, k)(j, \overline{v})}$ is a dominant monomial. Then either the path $p$ is $Y_{i,k}$ or it can be convert to $Y_{i,k}$ by raising moves of width $\ell$. Moreover, either the path $\widetilde{p}$ is one of the paths in $R_1$, where (we identify a path with its corresponding monomial)
\[
R_1=\{Y_{j-\gamma_0-(t-1) \ell,\overline{v}+\gamma_0+(t-1) \ell}Y_{i,k+2t \ell}^{-1}Y_{j+\gamma'+(t-1) \ell,\overline{v}+\gamma'+(t-1) \ell}: 1\leq t\leq r(i,j)\},
\]
see red paths in Figure \ref{figure grid 1}, or the path $\widetilde{p}$ can be converted to one of the paths in $R_1$ by lowering moves of width $\ell$. In fact, the product of $Y_{i,k}$ and any one of the paths in $R_1$ is a dominant monomial $m_t$, $1\leq t\leq r(i,j)$, which can be obtained by a translation of paths in $\mathscr{P}_{i,k}$ to paths in $\mathscr{P}_{i,k'_t}$ with respect to ${\bf PS}(j,v)$, where $k'_t$ is defined in (\ref{eq:def of rij}). That is, $m_t=Y_{a_t,\alpha_t}Y_{b_t,\beta_t}$, where $Y_{a_t,\alpha_t}Y_{b_t,\beta_t}$ is defined in (\ref{Eq:a_tb_t}).

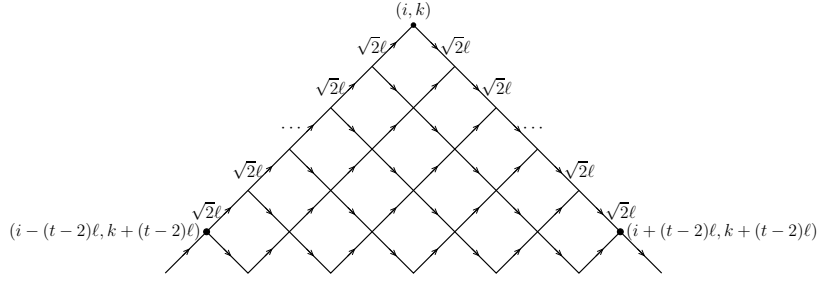
\begin{figure}
\centering
\resizebox{0.55\width}{0.55\height}{
\begin{minipage}[b]{0.95\textwidth}
\begin{tikzpicture}
\begin{scope}[thick, every node/.style={sloped,allow upside down}]
\draw (-0.5,5.5)--node {\midarrow}(0.5,6.5)--node {\midarrow}(1.5,7.5)--node {\midarrow}(2.5,8.5)--node {\midarrow}(3.5,9.5)--node {\midarrow}(4.5,10.5)--node {\midarrow}(5.5,11.5)--node {\midarrow}(6.5,10.5)--node {\midarrow}(7.5,9.5)--node {\midarrow}(8.5,8.5)--node {\midarrow}(9.5,7.5)--node {\midarrow}(10.5,6.5)--node {\midarrow}(11.5,5.5);
\draw (0.5,6.5)--node {\midarrow}(1.5,5.5)--node {\midarrow}(2.5,6.5)--node {\midarrow}(3.5,7.5)--node {\midarrow}(4.5,8.5)--node {\midarrow}(5.5,9.5)--node {\midarrow}(6.5,10.5);
\draw (1.5,7.5)--node {\midarrow}(2.5,6.5)--node {\midarrow}(3.5,5.5)--node {\midarrow}(4.5,6.5)--node {\midarrow}(5.5,7.5)--node {\midarrow}(6.5,8.5)--node {\midarrow}(7.5,9.5);
\draw (2.5,8.5)--node {\midarrow}(3.5,7.5)--node {\midarrow}(4.5,6.5)--node {\midarrow}(5.5,5.5)--node {\midarrow}(6.5,6.5)--node {\midarrow}(7.5,7.5)--node {\midarrow}(8.5,8.5);
\draw (3.5,9.5)--node {\midarrow}(4.5,8.5)--node {\midarrow}(5.5,7.5)--node {\midarrow}(6.5,6.5)--node {\midarrow}(7.5,5.5)--node {\midarrow}(8.5,6.5)--node {\midarrow}(9.5,7.5);
\draw (4.5,10.5)--node {\midarrow}(5.5,9.5)--node {\midarrow}(6.5,8.5)--node {\midarrow}(7.5,7.5)--node {\midarrow}(8.5,6.5)--node {\midarrow}(9.5,5.5)--node {\midarrow}(10.5,6.5);
\draw[fill] (5.5,11.5) circle (1.5pt) (0.5,6.5) circle (2pt) (10.5,6.5) circle (2pt);
\end{scope}
\node[above] at (5.5,11.5)  {$(i,k)$};
\node[left] at (5,11)   {$\sqrt{2}\ell$};
\node[left] at (4,10)   {$\sqrt{2}\ell$};
\node[left] at (3,9)   {$\cdots$};
\node[left] at (2,8)   {$\sqrt{2}\ell$};
\node[left] at (1,7)   {$\sqrt{2}\ell$};
\node[right] at (6,11)   {$\sqrt{2}\ell$};
\node[right] at (7,10)   {$\sqrt{2}\ell$};
\node[right] at (8,9)   {$\cdots$};
\node[right] at (9,8)   {$\sqrt{2}\ell$};
\node[right] at (10,7)   {$\sqrt{2}\ell$};
\node[left] at (0.5,6.5) {$(i-(t-2)\ell, k+(t-2) \ell)$};
\node[right] at (10.5,6.5) {$(i+(t-2)\ell, k+(t-2) \ell)$};
\end{tikzpicture}
\end{minipage}}
\caption{The paths obtained by lowering moves of $Y_{i,k}$.}\label{f:the paths obtained by lowering moves}
\end{figure}

In the following, we calculate the multiplicity of the dominant monomial $m_t$, $1\leq t\leq r(i,j)$. If $t=1$, then $\mathfrak{m}(p)=Y_{i,k}$ and $\mathfrak{m}(\widetilde{p})=Y_{j-\gamma,\overline{v}+\gamma}Y_{i,k+2 \ell}^{-1}Y_{j+\gamma',\overline{v}+\gamma'}$. Thus, the multiplicity of the dominant monomial $Y_{j-\gamma,\overline{v}+\gamma}Y_{j+\gamma',\overline{v}+\gamma'}$ is $1$. If $t=2$, then $\mathfrak{m}(p)=Y_{i,k}$ and $\mathfrak{m}(\widetilde{p})=Y_{j-\gamma- \ell,\overline{v}+\gamma+ \ell}Y_{i,k+4 \ell}^{-1}Y_{j+\gamma'+ \ell,\overline{v}+\gamma'+ \ell}$. Thus, the multiplicity of the dominant monomial $Y_{j-\gamma- \ell,\overline{v}+\gamma+ \ell}Y_{j+\gamma'+ \ell,\overline{v}+\gamma'+ \ell}$ is $1$. Assume that $t\geq 3$. Then the path $Y_{a_t,\alpha_t}Y^{-1}_{i,k'_t}Y_{b_t,\beta_t}\in \mathscr{P}_{j,\overline{v}}$ can be raised at points
\[
(i\pm s_u \ell, k'_t-(s_u+2u-1) \ell ), \ s_u \in \ZZ, \  0\leq s_u\leq t-2u-1, \ 1\leq u\leq \lceil \frac{t}2 \rceil-1,
\]
respectively. Moreover, the path $Y_{i,k}$ can be lowered at points
\[
(i\pm s_u \ell, k+(s_u+2u-1) \ell ), \ s_u \in \ZZ, \ 0\leq s_u\leq t-2u-1, \ 1\leq u\leq \lceil \frac{t}2 \rceil -1,
\]
respectively.

Therefore, for any $\widetilde{p}'$ obtained by raising moves of $Y_{a_t,\alpha_t}Y^{-1}_{i,k'_t}Y_{b_t,\beta_t}$, there exists a unique path $p'$ which is obtained by lowering moves of $Y_{i,k}$ such that $\mathfrak{m}(\widetilde{p}')\mathfrak{m}(p')=Y_{a_t,\alpha_t}Y_{b_t,\beta_t}=m_t$. To obtain the multiplicity of $m_t$, we only need to calculate the number of paths which is obtained by lowering moves of $Y_{i,k}$. That is, we need to calculate the number of paths from the point $(i-(t-2)\ell, k+(t-2) \ell)$ to the point $(i+(t-2)\ell, k+(t-2) \ell)$, see Figure \ref{f:the paths obtained by lowering moves} for example. This number is equal to $e(t) = \frac{\prod_{s=1}^{t-2} (t+s)}{(t-1)!}$. Hence, the multiplicity of the dominant monomial $m_t$, $1\leq t\leq r(i,j)$, is equal to $e(t)$. The sum of $\varepsilon$-characters of the modules $L(m_t)$, $1\leq t\leq r(i,j)$, is $\chi_{(i,k),(j,\overline{v})}$, as defined in Definition \ref{Def: chi_(i,k,j,v)}.

\textbf{Case 2}. $|j-i|+h(i,j)= 2\rho \ell$, $\rho \in \ZZ_{>0}$. We have that $\gamma=0$, $\gamma'=\ell$. Assume that $i=a \ell+a_0$ for some $a, a_0\in \mathbb{Z}_{\geq 0}$, $0\leq a_0< \ell$, and $n+1-i=a' \ell+a'_0$ for some $a', a_0'\in \mathbb{Z}_{\geq 0}$, $0\leq a'_0< \ell$. Let $y_0=i+k-a_0$ and
\begin{align*}
{\bf G}(i,k)=&\{((0,i+k),(a_0,y_0))\}\uplus \mathfrak{G}_{a_0,y_0,a, a'}(i,k)\\
&\uplus \{((n+1-a'_0, n+1-i+k-a'_0),(n+1,n+1-i+k))\} \subseteq \mathscr{P}_{i, k}.
\end{align*}
For example, in Figure \ref{F: figure grid 2}, ${\bf G}(i,k)$ is the set of all black or red paths from $A$ to $B$ and the sizes of these square boxes are $\sqrt{2}\ell\times \sqrt{2}\ell$. By the definition of $r(i,j)$ in (\ref{eq:def of rij}), if $r(i,j)=0$, then there is no translation of paths. Furthermore, if $r(i, j) > 0$, there exist $b_0, b'_0\in \mathbb{Z}_{\geq 0}$ such that $j-\gamma_0=b \ell+b_0$, where $b=\rho+r(i,j)-1$, and $n+1-j=b' \ell+b'_0$, where $b'=r(i,j)$. In Figure \ref{F: figure grid 2}, $r(i,j)=4$, $\rho=3$. Denote $E=(j,\overline{v})$. We have $|E_1E_2|=b_0$, $|E_3E|=\gamma_0$, and $|E_4E_5|=b'_0$. Let
\begin{align*}
{\bf G}(j,\overline{v})=\{((0,j+\overline{v}),(b_0,j+\overline{v}-b_0))\}\uplus\mathfrak{G}_{b_0,j+\overline{v}-b_0, b, b'}(j,\overline{v})\uplus \mathscr{G}^{j, \overline{v}, b', \gamma_0 }(j,\overline{v})
\subseteq \mathscr{P}_{j, \overline{v}}.
\end{align*}
For example, in Figure \ref{F: figure grid 2}, ${\bf G}(j,\overline{v})$ is the set of all red or black paths from $C$ to $D$ and the sizes of these square boxes are $\sqrt{2}\ell\times \sqrt{2}\ell$.

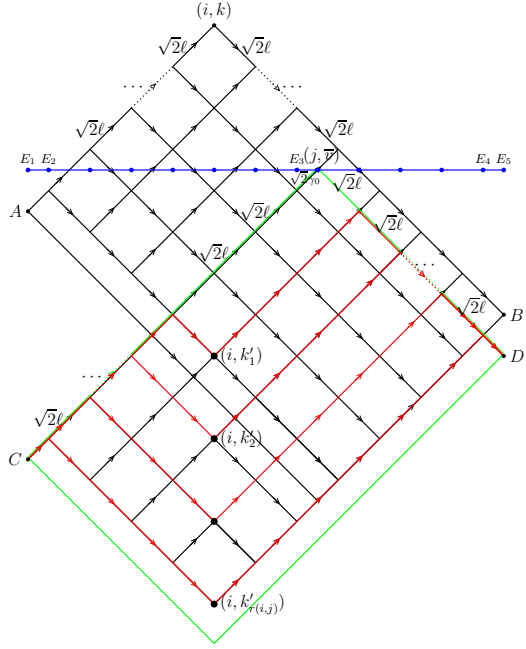
\begin{figure}
\centering
\resizebox{0.55\width}{0.55\height}{
\begin{minipage}[b]{0.7\textwidth}
\begin{tikzpicture}
\begin{scope}[thick, every node/.style={sloped,allow upside down}]
\draw[fill] (4,20.5) circle (1pt);
\draw (-0.5,16)--node {\midarrow}(1,17.5)--node {\midarrow}(2,18.5);
\draw[thick,dotted] (2,18.5)--node {\midarrow}(3,19.5);
\draw(3,19.5)--node {\midarrow}(4,20.5)--node {\midarrow}(5,19.5);
\draw[thick,dotted](5,19.5)--node {\midarrow}(6,18.5);
\draw(6,18.5)--node {\midarrow}(7,17.5)--node {\midarrow}(8,16.5)--node {\midarrow}(9,15.5)--node {\midarrow}(10,14.5)--node {\midarrow}(11,13.5);
\draw (-0.5,16)--node {\midarrow}(6.5,9) -- node {\midarrow}(11,13.5);
\draw[black](3,19.5)--node {\midarrow}(4,18.5)--node {\midarrow}(5,17.5)--node {\midarrow}(6,16.5)--node {\midarrow}(7,15.5)--node {\midarrow}(8,14.5)--node {\midarrow}(9,13.5)--node {\midarrow}(10,12.5);
\draw (8,14.5)--node {\midarrow}(9,15.5);
\draw (9,13.5)--node {\midarrow}(10,14.5);
\draw[black](2,18.5)--node {\midarrow}(3,17.5)--node {\midarrow}(4,16.5)--node {\midarrow}(5,15.5)--node {\midarrow}(6,14.5)--node {\midarrow}(7,13.5);
\draw[black](1,17.5)--node {\midarrow}(2,16.5)--node {\midarrow}(3,15.5)--node {\midarrow}(4,14.5)--node {\midarrow}(5,13.5)--node {\midarrow}(6,12.5)--node {\midarrow}(7,11.5)--node {\midarrow}(8,10.5);
\draw[black](0,16.5)--node {\midarrow}(1,15.5)--node {\midarrow}(2,14.5)--node {\midarrow}(3,13.5)--node {\midarrow}(4,12.5)--node {\midarrow}(5,11.5)--node {\midarrow}(6,10.5)--node {\midarrow}(7,9.5);
\draw[black](1,15.5)--node {\midarrow}(2,16.5)--node {\midarrow}(3,17.5)--node {\midarrow}(4,18.5)--node {\midarrow}(5,19.5);
\draw[black](2,14.5)--node {\midarrow}(3,15.5)--node {\midarrow}(4,16.5)--node {\midarrow}(5,17.5)--node {\midarrow}(6,18.5);
\draw[black](-0.5,10)--node {\midarrow}(0,10.5)--node {\midarrow}(1,11.5)--node {\midarrow}(2,12.5)--node {\midarrow}(3,13.5)--node {\midarrow}(4,14.5)--node {\midarrow}(5,15.5)--node {\midarrow}(6,16.5)--node {\midarrow}(7,17.5);
\draw[black](1,9.5)--node {\midarrow}(2,10.5)--node {\midarrow}(3,11.5)--node {\midarrow}(4,12.5)--node {\midarrow}(5,13.5)--node {\midarrow}(6,14.5)--node {\midarrow}(7,15.5)--node {\midarrow}(8,16.5);
\draw[black](2,8.5)--node {\midarrow}(3,9.5)--node {\midarrow}(4,10.5)--node {\midarrow}(5,11.5)--node {\midarrow}(6,12.5)--node {\midarrow}(7,13.5)--node {\midarrow}(8.5,15);
\draw[black](3,7.5)--node {\midarrow}(4,8.5)--node {\midarrow}(5,7.5);
\draw[black](4,6.5)--node {\midarrow}(5,7.5)--node {\midarrow}(6,8.5)--node {\midarrow}(7,9.5)--node {\midarrow}(8,10.5)--node {\midarrow}(9,11.5)--node {\midarrow}(10.5,13);
\node [above] at (4,20.5) {$(i,k)$};
\draw[fill] (6.5,17) circle (1pt);
\draw[green](-0.5,10.05)--(4,5.55) --(11,12.55);
\draw[green] (-0.5,10.05)--(1,11.55);
\draw[green,dotted,thick] (1,11.55)--node {\midarrow}(2,12.55);
\draw[green] (2,12.55)--(6.5,17.05)--(8.5,15.05);
\draw[green,dotted,thick](8.5,15.05)--(9.5,14.05);
\draw[green](9.5,14.05)--(11,12.55);
\draw[black](0,10.5)--(4,6.5);
\draw[black](1,11.5)--(5,7.5);
\draw[black](4,10.5)--node {\midarrow}(5,9.5)--node {\midarrow}(6,8.5);
\draw[black](5,11.5)--node {\midarrow}(6,10.5)--node {\midarrow}(7,9.5);
\draw[black](7,13.5)--node {\midarrow}(8,12.5)--node {\midarrow}(9,11.5);
\draw[red](-0.5,10)--(1,11.5);
\draw[red,dotted,thick] (1,11.5)--node {\midarrow}(2,12.5);
\draw[red](2,12.5)--node {\midarrow}(3,13.5)--node {\midarrow}(4,12.5)--node {\midarrow}(5,13.5)--node {\midarrow}(6,14.5)--node {\midarrow}(7,15.5)--node {\midarrow}(7.5,16)--(8.5,15);
\draw[red,dotted,thick](8.5,15)--node {\midarrow}(9.5,14);
\draw[red](9.5,14)--node {\midarrow}(10.5,13)--node {\midarrow}(11,12.5);
\draw[red](-0.5,10)--node {\midarrow}(1,11.5);
\draw[red](2,12.5)--node {\midarrow}(3,11.5)--node {\midarrow}(4,10.5)--node {\midarrow}(5,11.5)--node {\midarrow}(6,12.5)--node {\midarrow}(7,13.5)--node {\midarrow}(8.5,15);
\draw[red](9.5,14)--node {\midarrow}(11,12.5);
\draw[red,dotted,thick](8.5,15)--node {\midarrow}(9.5,14);
\draw[red,dotted,thick] (1,11.5)--node {\midarrow}(2,12.5);
\draw[red](-0.5,10)--node {\midarrow}(1,11.5)--node {\midarrow}(2,10.5)--node {\midarrow}(3,9.5)--node {\midarrow}(4,8.5)--node {\midarrow}(5,9.5)--node {\midarrow}(6,10.5)--node {\midarrow}(7,11.5)--node {\midarrow}(8,12.5)--node {\midarrow}(9.5,14)--node {\midarrow}(11,12.5);
\draw[red](-0.5,10)--node {\midarrow}(0,10.5)--node {\midarrow}(1,9.5)--node {\midarrow}(2,8.5)--node {\midarrow}(3,7.5)--node {\midarrow}(4,6.5)--node {\midarrow}(5,7.5)--node {\midarrow}(6,8.5)--node {\midarrow}(7,9.5)--node {\midarrow}(8,10.5)--node {\midarrow}(9,11.5)--node {\midarrow}(10.5,13)--node {\midarrow}(11,12.5);
\draw[fill](4,6.5) circle (2pt) (4,8.5) circle (2pt)(4,10.5) circle (2pt)(4,12.5) circle (2pt) (-0.5,16) circle (1pt) (11,13.5) circle (1pt)  (-0.5,10) circle (1pt)  (11,12.5) circle (1pt);
\draw[blue,thick] (-0.5,17)--(11,17);
\draw[fill,blue](-0.5,17) circle (1.2pt) (0,17) circle (1.2pt) (1,17) circle (1.2pt) (2,17) circle (1.2pt) (3,17) circle (1.2pt) (4,17) circle (1.2pt) (5,17) circle (1.2pt) (6,17) circle (1.2pt) (6.5,17) circle (1.2pt) (7.5,17) circle (1.2pt) (8.5,17) circle (1.2pt) (9.5,17) circle (1.2pt)  (10.5,17) circle (1.2pt) (11,17) circle (1.2pt);
\end{scope}
\node [above] at (6.6,17) {$(j,\overline{v})$};
\node [left] at (3.5,20) {$\sqrt{2} \ell$};
\node [left] at (2.5,19) {$\cdots$};
\node [left] at (1.5,18) {$\sqrt{2}\ell$};
\node [right] at (4.5,20) {$\sqrt{2}\ell$};
\node [right] at (5.5,19) {$\cdots$};
\node [right] at (6.5,18) {$\sqrt{2}\ell$};
\node  at (6.2,16.8) {\scriptsize $\sqrt{2}\gamma_0$};
\node [left] at (5.5,16) {$\sqrt{2}\ell$};
\node [left] at (4.5,15) {$\sqrt{2}\ell$};
\node [left] at (1.5,12) {$\cdots$};
\node [left] at (0.5,11) {$\sqrt{2}\ell$};
\node [right] at (6.7,16.7) {$\sqrt{2}\ell$};
\node [right] at (7.7,15.7) {$\sqrt{2}\ell$};
\node [right] at (8.7,14.7) {$\cdots$};
\node [right] at (9.7,13.7) {$\sqrt{2}\ell$};
\node[right] at (4,6.5) {$(i, k'_{r(i,j)})$};
\node[right] at (4,10.5) {$(i, k'_2)$};
\node[right] at (4,12.5) {$(i, k'_1)$};
\node[left] at (-0.5,16) {$A$};
\node[right] at (11,13.5) {$B$};
\node[left] at (-0.5,10) {$C$};
\node[right] at (11,12.5) {$D$};
\node [above] at(-0.5,17) {\scriptsize $E_1$};
\node [above] at(0,17)    {\scriptsize $E_2$};
\node [above] at (6,17)   {\scriptsize $E_3$};
\node [above] at (10.5,17) {\scriptsize $E_4$};
\node [above] at(11,17)    {\scriptsize $E_5$};
\end{tikzpicture}
\end{minipage}}
\caption{$|j-i|+h(i,j)= 2\rho \ell$. The black and green rectangles are ${\bf P}_{i,k}$ and ${\bf P}_{j,\overline{v}}$, respectively. The product of $Y_{i,k}$ and any one of the red paths is a dominant monomial, which can be obtained by translating paths in $\mathscr{P}_{i,k}$ to paths in $\mathscr{P}_{i,k'_t}$ with respect to ${\bf PS}(j,v)$, where $(i,k'_t)$ is one of the black bullet points.}\label{F: figure grid 2}
\end{figure}

Subcase 2.1. Similar to Subcase 1.1 and Subcase 1.2, we obtain that if $p \notin {\bf G}(i,k)$, then there exists no path $\widetilde{p}\in \mathscr{P}_{j, \overline{v}}$ such that $\mathfrak{m}(p)\mathfrak{m}(\widetilde{p}) \in S_{(i, k)(j, \overline{v})}$ is a dominant monomial; if $p \in {\bf G}(i,k)$, then there exists no path $\widetilde{p}\in \mathscr{P}_{j, \overline{v}}\setminus{\bf G}(j,\overline{v})$ such that $\mathfrak{m}(p)\mathfrak{m}(\widetilde{p}) \in S_{(i, k)(j, \overline{v})}$ is a dominant monomial.

Subcase 2.2. If $p \in {\bf G}(i,k)$ and there exists a path $\widetilde{p}\in {\bf G}(j, \overline{v})$ such that $\mathfrak{m}(p)\mathfrak{m}(\widetilde{p}) \in S_{(i, k)(j, \overline{v})}$ is a dominant monomial, then either the path $p$ is $Y_{i,k}$ or the path $p$ can be converted to $Y_{i,k}$ by raising moves of width $\ell$. Moreover, either the path $\widetilde{p}$ is one of the paths in $R_2$, where
\[
R_2=\{Y_{j-\rho \ell- \gamma_0-(t-1)\ell,\overline{v}+\rho \ell+\gamma_0+(t-1)\ell}Y_{i,k+2\rho \ell+2t \ell}^{-1}Y_{j+t \ell,\overline{v}+t \ell}: 1\leq t\leq r(i,j)\},
\]
see red paths in Figure \ref{F: figure grid 2}, or the path $\widetilde{p}$ can be converted to one of the paths in $R_2$ by lowering moves of width $\ell$. In fact, the product of $Y_{i,k}$ and any one of the paths in $R_2$ is a dominant monomial $m_t$, $1\leq t\leq r(i,j)$, which can be obtained by a translation of paths in $\mathscr{P}_{i,k}$ to paths in $\mathscr{P}_{i,k'_t}$ with respect to ${\bf PS}(j,v)$, where $k'_t$ is defined in (\ref{eq:def of rij}). That is, $m_t=Y_{a_t,\alpha_t}Y_{b_t,\beta_t}$, where $Y_{a_t,\alpha_t}Y_{b_t,\beta_t}$ is defined in (\ref{Eq:a_tb_t}). Using the same calculation method as Subcase 1.3, we obtain that the multiplicity of the dominant monomial $m_t$ is equal to $f(t) = \frac{ \rho \prod_{s=1}^{t-1}(\rho+t+s)}{t!}$. The sum of $\varepsilon$-characters of the modules $L(m_t)$, $1\leq t\leq r(i,j)$, is $\chi_{(i,k),(j,\overline{v})}$, which is defined in Definition \ref{Def: chi_(i,k,j,v)}.

\begin{figure}
\centering
\begin{minipage}[b]{0.48\textwidth}
\resizebox{0.5\width}{0.5\height}{
\begin{tikzpicture}
\begin{scope}[thick, every node/.style={sloped,allow upside down}]
\draw[fill] (5.5,40) circle (1pt) (-0.5,34) circle (1pt) (12.5,33) circle (1pt)  (-0.5,29) circle (1pt) (12.5,31) circle (1pt) (7,36.5) circle (1pt);
\draw (-0.5,34)--node {\midarrow}(1,35.5);
\draw[dotted,thick] (1,35.5)--node {\midarrow}(2.5,37);
\draw (2.5,37)--node {\midarrow}(4,38.5)--node {\midarrow}(5.5,40)--node {\midarrow}(7,38.5)--node {\midarrow}(8.5,37);
\draw[dotted,thick](8.5,37)--node {\midarrow}(10,35.5);
\draw(10,35.5)--node {\midarrow}(12.5,33)--(6.5,27) -- (-0.5,34);
\node [above] at (5.5,40) {$(i,k)$};
\draw (7,35.5)--node {\midarrow}(8.5,37);
\draw (8.5,34)--node {\midarrow}(10,32.5);
\draw (10,32.5)--node {\midarrow}(11.5,34);
\draw[green](-0.5,29)--(2,31.5);
\draw [green,thick,dotted](2,31.5) --(3.5,33);
\draw[green](3.5,33)--(7,36.5)--(9,34.5);
\draw[green,thick,dotted](9,34.5)--(10.5,33);
\draw[green](10.5,33)--(12.5,31)--(5,23.5) -- (-0.5,29);
\node [above] at (7,36.5) {\scriptsize $(j,\overline{v})$};
\draw [red] (-0.5,29)--node {\midarrow}(0.5,30)--node {\midarrow}(2,31.5);
\draw [red,thick,dotted](2,31.5) --node {\midarrow}(3.5,33);
\draw[red] (3.5,33)--node {\midarrow}(5,34.5)--node {\midarrow}(5.5,34)--node {\midarrow}(7.5,36)--node {\midarrow}(9,34.5);
\draw[red,thick,dotted](9,34.5)--node {\midarrow}(10.5,33);
\draw[red](10.5,33)--node {\midarrow}(12.5,31);
\draw[red](3.5,33)--node {\midarrow}(4,32.5)--node {\midarrow}(5.5,31)--node {\midarrow}(7,32.5)--node {\midarrow}(9,34.5);
\draw[red,thick,dotted](9,34.5)--node {\midarrow}(10.5,33);
\draw[red](10.5,33)--node {\midarrow}(12.5,31);
\draw [red] (-0.5,29)--node {\midarrow}(0.5,30)--node {\midarrow}(2,31.5)--node {\midarrow}(2.5,31)--node {\midarrow}(4,29.5)--node {\midarrow}(5.5,28)--node {\midarrow}(7,29.5)--node {\midarrow}(8.5,31)--node {\midarrow}(10.5,33)--(12.5,31);
\draw [red] (-0.5,29)--node {\midarrow}(0.5,30)--node {\midarrow}(1,29.5)--node {\midarrow}(2.5,28)--node {\midarrow}(4,26.5)--node {\midarrow}(5.5,25)--node {\midarrow}(7,26.5)--node {\midarrow}(8.5,28)--node {\midarrow}(10,29.5)--node {\midarrow}(12,31.5)--node {\midarrow}(12.5,31);
\draw[fill,black](5.5,34) circle (2pt) (5.5,31) circle (2pt)(5.5,28) circle (2pt)(5.5,25) circle (2pt);
\draw [black] (4,38.5)--node {\midarrow}(5.5,37)--node {\midarrow}(7,35.5)--node {\midarrow}(8.5,34)--node {\midarrow}(10,35.5);
\draw [black] (2.5,37)--node {\midarrow}(4,35.5)--node {\midarrow}(5,34.5);
\draw [black] (5.5,34)--node {\midarrow}(7,32.5)--node {\midarrow}(8.5,31)--node {\midarrow}(10,29.5);
\draw [black] (1,35.5)--node {\midarrow}(2.5,34)--node {\midarrow}(3.5,33);
\draw [black] (5.5,31)--node {\midarrow}(7,29.5)--node {\midarrow}(8.5,28);
\draw [black] (1,29.5)--node {\midarrow}(2.5,31)--node {\midarrow}(4,32.5)--node {\midarrow}(5.5,34);
\draw [black] (2.5,28)--node {\midarrow}(4,29.5)--node {\midarrow}(5.5,31);
\draw [black] (4,26.5)--node {\midarrow}(5.5,28)--node {\midarrow}(7,26.5);
\draw [black] (2.5,34)--node {\midarrow}(4,35.5)--node {\midarrow}(5.5,37)--node {\midarrow}(7,38.5);
\draw [black](6.5,27)--(7,26.5);
\draw [blue,thick] (-0.5,36.5)--(12.5,36.5);
\draw[fill,blue]  (-0.5,36.5) circle (1.2pt) (0.5,36.5) circle (1.2pt) (2,36.5) circle (1.2pt)  (3.5,36.5) circle (1.2pt) (5,36.5) circle (1.2pt) (6.5,36.5) circle (1.2pt)  (7,36.5) circle (1.2pt) (7.5,36.5) circle (1.2pt)  (9,36.5) circle (1.2pt) (10.5,36.5) circle (1.2pt)  (12,36.5) circle (1.2pt) (12.5,36.5) circle (1.2pt);
\node[left] at (-0.5,34) {$A$};
\node[right] at (12.5,33) {$B$};
\node[left] at (-0.5,29) {$C$};
\node[right] at (12.5,31) {$D$};
\node[left] at (4.75,39.25) {$\sqrt{2}\ell$};
\node[left] at (3.25,37.75) {$\sqrt{2}\ell$};
\node[left] at (1.75,36.25) {$\cdots$};
\node[right] at (6.25,39.25) {$\sqrt{2}\ell$};
\node[right] at (7.75,37.75) {$\sqrt{2}\ell$};
\node[right] at (9.25,36.25) {$\cdots$};
\node at (7.25,36.2) {\tiny $\sqrt{2}\gamma'$};
\node at (6.5,36.2) {\tiny $\sqrt{2}\gamma_0$};
\node[left] at (5,34.5) {$(a_1,\alpha_1)$};
\node[right] at (7.5,36) {$(b_1,\beta_1)$};
\node[left] at (3.5,33) {$(a_2,\alpha_2)$};
\node[right] at (9,34.5) {$(b_2,\beta_2)$};
\node[left] at (2.5,32) {$\cdots$};
\node[right] at (10,33.5) {$\cdots$};
\node[left] at (0.5,30) {$(a_{r(i,j)},\alpha_{r(i,j)})$};
\node[right] at (12,31.5) {$(b_{r(i,j)},\beta_{r(i,j)})$};
\node[right] at (5.5,34) {$(i,k'_1)$};
\node[right] at (5.5,31) {$(i,k'_2)$};
\node[right] at (5.5,25) {$(i,k'_{r(i,j)})$};
\node at (6,23) {$(a)$};
\node[above] at (-0.5,36.5) {\scriptsize $E_1$};
\node[above] at(0.5,36.5)  {\scriptsize $E_2$};
\node[above] at(6.5,36.5)  {\tiny $E_3$};
\node[above] at(7.5,36.5)  {\tiny $E_4$};
\node[above] at(12,36.5)  {\scriptsize $E_5$};
\node[above] at(12.5,36.5) {\scriptsize $E_6$};
\draw[fill] (5,34.5) circle (1.5pt)  (7.5,36) circle (1.5pt) (0.5,30) circle (1.5pt) (12,31.5) circle (1.5pt)  (3.5,33) circle (1.5pt)  (9,34.5) circle (1.5pt) (2,31.5) circle (1.5pt)  (10.5,33) circle (1.5pt);
\end{scope}
\end{tikzpicture} }
\end{minipage}
\begin{minipage}[b]{0.48\textwidth}
\resizebox{0.5\width}{0.5\height}{
\begin{tikzpicture}
\begin{scope}[thick, every node/.style={sloped,allow upside down}]
\draw[fill] (5.5,22) circle (1pt) (-0.5,16) circle (1pt) (12.5,15) circle (1pt)  (-0.5,11) circle (1pt) (12.5,13) circle (1pt) (7,18.5) circle (1pt);
\draw (-0.5,16)--(1,17.5);
\draw[dotted,thick] (1,17.5)--(2.5,19);
\draw (2.5,19)--(5.5,22)--(8,19.5);
\draw[dotted,thick](8,19.5)--(9.5,18);
\draw(9.5,18)--(12.5,15)--(6.5,9) -- (-0.5,16);
\node[above] at (5.5,22) {$(i,k)$};
\draw[green](-0.5,11.05)--(2,13.55);
\draw[green,thick,dotted](2,13.55) --(3.5,15.05);
\draw[green](3.5,15.05)--(7,18.55)--(8.5,17.05);
\draw[green,thick,dotted](8.5,17.05)--(10,15.55);
\draw[green](10,15.55)--(12.5,13.05)--(5,5.55) -- (-0.5,11.05);
\node[above] at (7,18.5) {$(j,\overline{v})$};
\draw[red](-0.5,11)--node {\midarrow}(2,13.5);
\draw[red,thick,dotted](2,13.5) --node {\midarrow}(3.5,15);
\draw[red](3.5,15)--node {\midarrow}(5,16.5)--node {\midarrow}(6.5,15)--node {\midarrow}(8.5,17);
\draw[red,thick,dotted](8.5,17)--node {\midarrow}(10,15.5);
\draw[red](10,15.5)--node {\midarrow}(11.5,14)--node {\midarrow}(12.5,13);
\draw[red](-0.5,16)--node {\midarrow}(1,17.5);
\draw[red,dotted,thick] (1,17.5)--node {\midarrow}(2.5,19);
\draw[red](2.5,19)--node {\midarrow}(4,20.5)--node {\midarrow}(5,19.5)--node {\midarrow}(6.5,21)--node {\midarrow}(8,19.5);
\draw[red,dotted,thick](8,19.5)--node {\midarrow}(9.5,18);
\draw[red](9.5,18)--node {\midarrow}(12.5,15);
\draw[cyan] (-0.55,16)--(1.05,17.5);
\draw[cyan,dotted,thick] (1.05,17.5)--(2.55,19);
\draw [cyan](2.55,19.05)--node {\midarrow}(3.55,18.05)--node {\midarrow}(5.05,19.55)--node {\midarrow}(6.55,21.05)--node {\midarrow}(8.05,19.55);
\draw[cyan,dotted,thick](8.05,19.55)--node {\midarrow}(9.55,18.05);
\draw[cyan](9.55,18.05)--node {\midarrow}(12.55,15.05);
\draw[cyan](-0.55,11.05)--node {\midarrow}(2.05,13.55);
\draw [cyan,thick,dotted](2.05,13.55) --node {\midarrow}(3.55,15.05);
\draw[cyan](3.55,15.05)--node {\midarrow}(5.05,13.55)--node {\midarrow}(6.55,12)--node {\midarrow}(8,13.5)--node {\midarrow}(10.05,15.5)--node {\midarrow}(12.55,13);
\draw [violet] (-0.4,11)--node {\midarrow}(2.1,13.5)--node {\midarrow}(3.6,12)--node {\midarrow}(5.1,10.5)--node {\midarrow}(6.6,9)--node {\midarrow}(8.1,10.5)--node {\midarrow}(9.6,12)--node {\midarrow}(11.6,14)--node {\midarrow}(12.6,13);
\draw (2,16.5)--node {\midarrow}(3.5,15);
\draw (8,16.5)--node {\midarrow}(9.5,15)--node {\midarrow}(11,16.5);
\draw [violet](-0.4,16)--node {\midarrow}(1.1,17.5)--node {\midarrow}(2.1,16.5)--node {\midarrow}(3.6,18)--node {\midarrow}(5.1,19.5)--node {\midarrow}(6.6,21)--node {\midarrow}(8.1,19.5);
\draw[violet,dotted,thick](8.1,19.5)--node {\midarrow}(9.6,18);
\draw[violet](9.6,18)--node {\midarrow}(11.1,16.5)--node {\midarrow}(12.6,15);
\draw [black](5,19.5)--node {\midarrow}(6.5,18)--node {\midarrow}(8,16.5);
\draw [black](3.5,18)--node {\midarrow}(5,16.5);
\draw [black](3.5,12)--node {\midarrow}(5,13.5)--node {\midarrow}(6.5,15);
\draw [black](5,10.5)--node {\midarrow}(6.5,12);
\draw [black] (6.5,18)--node {\midarrow}(8,19.5);
\draw [black] (6.5,15)--node {\midarrow}(8,13.5)--node {\midarrow}(9.5,12);
\draw [black] (6.5,12)--node {\midarrow}(8,10.5);
\draw [black] (8.5,17)--node {\midarrow}(9.5,18);
\draw [blue,thick] (-0.5,18.5)--(12.5,18.5);
\draw[fill,blue] (-0.5,18.5) circle (1.2pt) (2,18.5) circle (1.2pt) (3.5,18.5) circle (1.2pt) (5,18.5) circle (1.2pt) (6.5,18.5) circle (1.2pt) (7,18.5) circle (1.2pt)  (8.5,18.5) circle (1.2pt) (10,18.5) circle (1.2pt) (11.5,18.5) circle (1.2pt) (12.5,18.5) circle (1.2pt);
\end{scope}
\node[left] at (-0.5,16) {$A$};
\node[right] at (12.5,15) {$B$};
\node[left] at (-0.5,11) {$C$};
\node[right] at (12.5,13) {$D$};
\node [left] at (4.9,21.5) {$\sqrt{2}\ell$};
\node [left] at (3.4,20) {$\sqrt{2}\ell$};
\node [left] at (3,14.5) {$\cdots$};
\node  at (5.9,21.6) {$\sqrt{2}\gamma$};
\node [left] at (8.2,20.25) {$\sqrt{2}\ell$};
\node [right] at (9.5,16) {$\cdots$};
\node  at (6.6,18.2) {\scriptsize$\sqrt{2}\gamma_0$};
\node  at (5.25,16.25) {$\sqrt{2}\gamma'$};
\node  at (5.25,13.25) {$\sqrt{2}\gamma'$};
\node  at (5.25,10.3) {$\sqrt{2}\gamma'$};
\node  at (6,15.5) {$\sqrt{2}\gamma$};
\node at (6,5) {$(b)$};
\draw [fill,black](5.5,16) circle (2pt) (5.55,13) circle (2pt)(5.55,10) circle (2pt)(5.5,7) circle (2pt);
\node [right] at (5.5,16) {$(i,k'_1)$};
\node [right] at (5.5,13) {$(i,k'_2)$};
\node [right] at (5.5,7) {$(i,k'_{r(i,j)})$};
\draw [fill](5,16.5) circle (1.5pt) (3.5,15) circle (1.5pt) (2,13.5) circle (1.5pt) (0.5,12) circle (1.5pt);
\node [left] at (5,16.5) {$(a_1, \alpha_1)$};
\node [left] at (3.5,15){$(a_2,\alpha_2)$};
\node [left] at (0.5,12) {$(a_{r(i,j)}, \alpha_{r(i,j)})$};
\draw [fill] (7.5,18) circle (1.5pt) (9,16.5) circle (1.5pt) (10.5,15) circle (1.5pt) (12,13.5) circle (1.5pt);
\node [right] at (7.5,18)  {$(b_1,\beta_1)$};
\node [right] at (9,16.5)  {$(b_2,\beta_2)$};
\node [right] at (12,13.5) {$(b_{r(i,j)},\beta_{r(i,j)})$};
\node [above] at (-0.5,18.5)  {\scriptsize $E_1$};
\node [above] at (2,18.5)   {\scriptsize $E_2$};
\node [above] at (6.5,18.5)  {\tiny $E_3$};
\node [above] at (11.5,18.5)  {\scriptsize $E_4$};
\node [above] at (12.5,18.5) {\scriptsize $E_5$};
\end{tikzpicture}}
\end{minipage}
\caption{$2\rho \ell<|j-i|+h(i,j)< 2(\rho+1) \ell$, $\rho=1$. The black and green rectangles are ${\bf P}_{i,k}$ and ${\bf P}_{j,\overline{v}}$, respectively. (a) The product of the path $Y_{i,k}$ and any one of red paths in $\mathscr{P}_{j, \overline{v}}$ is a dominant monomial, which can be obtained by translating paths in $\mathscr{P}_{i,k}$ to paths in $\mathscr{P}_{i,k'_t}$ with respect to ${\bf PS}(j,v)$, where $(i,k'_t)$ is one of the black bullet points. (b) The product of the same color paths is a dominant monomial.} \label{F: figure grid 3}
\end{figure}
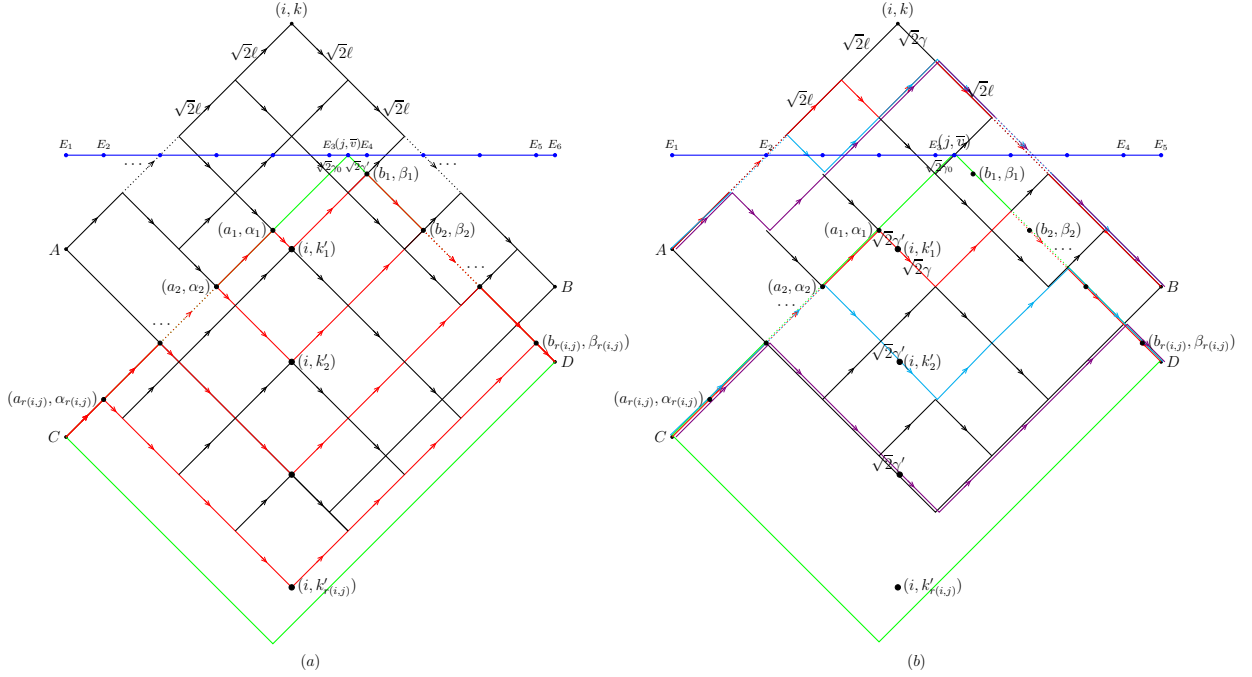

\textbf{Case 3}. $2\rho \ell <|j-i|+h(i,j)<2(\rho+1) \ell$, $\rho\in \ZZ_{>0}$.

Subcase 3.1. Assume that $i=a \ell+a_0$ for some $a, a_0\in \mathbb{Z}_{\geq 0}$, $0\leq a_0< \ell$, and $n+1-i=a' \ell+a'_0$ for some $a', a'_0\in \mathbb{Z}_{\geq 0}$, $0\leq a'_0< \ell$. Let $y_0=i+k-a_0$ and
\begin{align*}
{\bf G}(i,k)=&\{((0,i+k),(a_0,y_0))\}\uplus \mathfrak{G}_{a_0,y_0,a,a'}(i,k)\\
&\uplus \{((n+1-a'_0, n+1-i+k-a'_0),(n+1,n+1-i+k))\} \subseteq \mathscr{P}_{i,k}.
\end{align*}
For example, in Figure \ref{F: figure grid 3} $(a)$, ${\bf G}(i,k)$ is the set of all black or red paths from $A$ to $B$ and the sizes of these square boxes are $\sqrt{2}\ell\times \sqrt{2}\ell$.
By the definition of $r(i,j)$ in (\ref{eq:def of rij}), if $r(i,j)=0$, then there is no translation of paths. Furthermore, if $r(i, j) > 0$, there exist $b_0, b'_0\in \mathbb{Z}_{\geq 0}$ such that $j-\gamma_0=b \ell+b_0$, where $b=\rho+r(i,j)-1$, and $n+1-j-\gamma'=b' \ell+b'_0$, where $b'=r(i,j)-1$. In Figure \ref{F: figure grid 3} (a), $r(i,j)=4$, $\rho=1$. Denote $E=(j,\overline{v})$. We have $|E_1E_2|=b_0$, $|E_3E|=\gamma_0$, $|EE_4|=\gamma'$, and $|E_4E_5|=b'_0$. Let
\begin{align*}
{\bf G}(j,\overline{v})=\mathscr{G}_{b_0, j+\overline{v}-b_0, b, \gamma'}(j,\overline{v})\uplus\mathfrak{G}_{b_0+\gamma',j+\overline{v}-b_0+ \gamma', b, b'}(j,\overline{v})\uplus  \mathscr{G}^{j+\gamma', \overline{v}+\gamma', b', \gamma_0 }(j,\overline{v}) \subseteq \mathscr{P}_{j,\overline{v}}.
\end{align*}
For example, in Figure \ref{F: figure grid 3} $(a)$, ${\bf G}(j,\overline{v})$ is the set of all red or black paths from $C$ to $D$ and the sizes of these square boxes are $\sqrt{2}\ell\times \sqrt{2}\ell$.

On the one hand, similar to Subcase 1.2, we obtain that if $p \in {\bf G}(i,k)$, then there exists no path $\widetilde{p}\in \mathscr{P}_{j, \overline{v}}\setminus{\bf G}(j,\overline{v})$ such that $\mathfrak{m}(p)\mathfrak{m}(\widetilde{p}) \in S_{(i, k)(j, \overline{v})}$ is a dominant monomial.

On the other hand, if $p \in {\bf G}(i,k)$ and there exists a path $\widetilde{p}\in {\bf G}(j,\overline{v})$ such that $\mathfrak{m}(p)\mathfrak{m}(\widetilde{p})\in S_{(i, k)(j, \overline{v})}$ is a dominant monomial, then either the path $p$ is $Y_{i,k}$ or the path $p$ can be converted to $Y_{i,k}$ by raising moves of width $\ell$. Moreover, the path $\widetilde{p}$ is one of the paths in $R_3$, where
\begin{align*}
R_3=\{Y_{j-\rho \ell-\gamma_0-t \ell,\overline{v}+\rho \ell+\gamma_0+t \ell}Y_{i,k+2\rho \ell+2(t+1) \ell}^{-1}Y_{j+\gamma'+t \ell,\overline{v}+\gamma'+t \ell}: 0\leq t\leq r(i,j)-1\},
\end{align*}
see red paths in Figure \ref{F: figure grid 3} $(a)$, or the path $\widetilde{p}$ can be converted to one of the paths in $R_3$ by lowering moves of width $\ell$. In fact, the product of $Y_{i,k}$ and any one of the paths in $R_3$ is a dominant monomial $m_{t+1}$, $0\leq t\leq r(i,j)-1$, which can be obtained by a translation of paths in $\mathscr{P}_{i,k}$ to paths in $\mathscr{P}_{i,k'_{t+1}}$ with respect to ${\bf PS}(j,v)$, where $k'_{t+1}$ is defined in (\ref{eq:def of rij}). That is, $m_{t+1}=Y_{a_{t+1},\alpha_{t+1}}Y_{b_{t+1},\beta_{t+1}}$, where $Y_{a_{t+1},\alpha_{t+1}}Y_{b_{t+1},\beta_{t+1}}$ is defined in (\ref{Eq:a_tb_t}). Using the same calculation method as Subcase 1.3, we obtain that the multiplicity of the dominant monomial $m_{t+1}$ is equal to $g(t) = \frac{ (\rho+1)\prod_{s=2}^t (\rho+t+s)}{t!}$ if $1\leq t\leq r(i,j)-1$, and equal to $1$ if $t=0$. The sum of $\varepsilon$-characters of the modules $L(m_{t+1})$, $0\leq t\leq r(i,j)-1$, is the second term on the right-hand side of Equation $(\ref{D:the sum 3})$.

Subcase 3.2. Assume that $i=a \ell+a_0$ for some $a, a_0\in \mathbb{Z}_{\geq 0}$, $0\leq a_0< \ell$, and $n+1-i-\gamma=a' \ell+a'_0$ for some $a', a'_0\in \mathbb{Z}_{\geq 0}$, $0\leq a'_0< \ell$. Let
\begin{align*}
{\bf G'}(i,k)=&\mathscr{G}_{a_0, i+k-a_0, a, \gamma}(i,k)\uplus \mathfrak{G}_{a_0+\gamma, i+k-a_0+\gamma, a, a' }(i,k)\\
&\uplus \{((n+1-a'_0, n+1-i+k-a'_0),(n+1,n+1-i+k))\} \subseteq \mathscr{P}_{i,k}.
\end{align*}
For example, in Figure \ref{F: figure grid 3} $(b)$, ${\bf G'}(i,k)$ is the set of all red, black, violet, or cyan paths from $A$ to $B$.
By the definition of $r(i,j)$ in (\ref{eq:def of rij}), if $r(i,j)=0$, then there is no translation of paths. Assume that $r(i,j)>0$. If $a_{r(i,j)}-\gamma$, $b_{r(i,j)}+\gamma\in \hat{I}\cup\{n+1\}$, then there exist $b_0, b_0' \in \ZZ_{\ge 0}$ such that $j-\gamma_0=b \ell+b_0$, where $b=\rho+r(i,j)-1$, and $n+1-j=b' \ell+b'_0$, where $b'=r(i,j)$. If $a_{r(i,j)}-\gamma$ or $b_{r(i,j)}+\gamma\notin \hat{I}\cup\{n+1\}$, then there exist $b_0, b_0'\in \ZZ_{\geq 0}$ such that $j-\gamma_0=b \ell+b_0$, where $b=\rho+r(i,j)-2$, and $n+1-j=b' \ell+b'_0$, where $b'=r(i,j)-1$. In Figure \ref{F: figure grid 3} (b), $r(i,j)=4$, $\rho=1$. Denote $E=(j,\overline{v})$. We have $|E_1E_2|=b_0$, $|E_3E|=\gamma_0$, and $|E_4E_5|=b'_0$. Let
\begin{align*}
{\bf G'}(j,\overline{v})=\{((0,j+\overline{v}),(b_0,j+\overline{v}-b_0))\}\uplus\mathfrak{G}_{b_0,j+\overline{v}-b_0, b, b'}(j,\overline{v})\uplus \mathscr{G}^{j, \overline{v}, b', \gamma_0}(j,\overline{v}) \subseteq \mathscr{P}_{j,\overline{v}}.
\end{align*}
For example, in Figure \ref{F: figure grid 3} $(b)$, ${\bf G'}(j,\overline{v})$ is the set of all black, violet, cyan, or red paths from $C$ to $D$.

On the one hand, similar to Subcase 1.2, we obtain that if $p \in {\bf G'}(i,k)$, then there exists no path $\widetilde{p}\in \mathscr{P}_{j, \overline{v}}\setminus{\bf G'}(j,\overline{v})$ such that $\mathfrak{m}(p)\mathfrak{m}(\widetilde{p}) \in S_{(i, k)(j, \overline{v})}$ is a dominant monomial.

On the other hand, if $p$ is in ${\bf G'}(i,k)$ and there exists a path $\widetilde{p}\in {\bf G'}(j,\overline{v})$ such that $\mathfrak{m}(p)\mathfrak{m}(\widetilde{p}) \in S_{(i, k)(j, \overline{v})}$ is a dominant monomial, then either the path $p$ is one of the paths in $R'_3$, where
\begin{align*}
R'_3=&\{Y_{i-t \ell,k+t \ell}Y_{i+\gamma-t \ell,k+\gamma+ t\ell}^{-1}Y_{i+\gamma,k+\gamma}: 1\leq t\leq b'\},
\end{align*}
see the red, cyan, and violet paths in Figure \ref{F: figure grid 3} $(b)$, or the path $p$ can be converted to one of the paths in $R'_3$ by raising moves of width $\ell$. Moreover, either the path $\widetilde{p}$ is one of the paths in $R''_3$, where
\begin{align*}
R''_3=&\{Y_{i+\gamma-t\ell,k+\gamma+t \ell+2\rho \ell}Y_{i+\gamma,k+\gamma+2\rho \ell+2t \ell}^{-1}Y_{j+t \ell,\overline{v}+t \ell}: 1\leq t\leq b'\},
\end{align*}
see red, cyan, and violet paths in Figure \ref{F: figure grid 3} $(b)$, or the path $\widetilde{p}$ can be converted to one of the paths in $R''_3$ by lowering moves of width $\ell$. The product of one of the paths in $R'_3$ and corresponding path of $R''_3$ is a dominant monomial $m_t$, $1\leq t\leq b'$, which can be obtained by a translation of paths in $\mathscr{P}_{a_t, \alpha_t}$ to paths in $\mathscr{P}_{a_t,\overline{\alpha_t}}$ with respect to ${\bf PS}(b_t,\beta_t)$, where $a_t, \alpha_t, b_t,\beta_t$ are defined in (\ref{Eq:a_tb_t}), $\overline{\alpha_t}$ is defined in (\ref{Eq:small index}). Using the same calculation method as Subcase 1.3, we obtain that the multiplicity of the dominant monomial $m_t$ is equal to $f(t) = \frac{ \rho \prod_{s=1}^{t-1}(\rho+t+s)}{t!}$. The sum of $\varepsilon$-characters of the modules $L(m_t)$, $1\leq t\leq b'$, is the first term on the right-hand side of Equation $(\ref{D:the sum 3})$.

Similar to Subcase 1.1, if the path $p$ is not in ${\bf G}(i,k) \cup {\bf G'}(i,k)$, then there exists no path $\widetilde{p}$ in $\mathscr{P}_{j, \overline{v}}$ such that $\mathfrak{m}(p)\mathfrak{m}(\widetilde{p})$ in $S_{(i, k)(j, \overline{v})}$ is a dominant monomial.

In summary, all dominant monomials $m$ in $S_{(i, k)(j, \overline{v})}$ except the highest $l$-weight monomial can be obtained by path translations or raising and lowering moves of width $\ell$ from dominant monomials obtained by translations of paths. The sum of $\varepsilon$-characters of modules $L(m)$ is $\chi_{(i,k),(j,\overline{v})}$ which is defined in Definition \ref{Def: chi_(i,k,j,v)}.
\end{proof}

\begin{remark}\label{R:multiplicity}
According to the calculation of multiplicity of a dominant monomial $m\neq Y_{i,k}Y_{j, \overline{v}}$ in $S_{(i, k)(j, \overline{v})}$, we obtain that if the multiplicity of $m$ is $1$, then $m$ can be obtained by translations of paths. If the multiplicity of $m$ is greater than $1$, then a copy of $m$ can be obtained by translations of paths. This copy of $m$ is equal to $\mathfrak{m}(p_1)\mathfrak{m}(p_2)$ for some paths $p_1, p_2$. The other copies of $m$ can be obtained by raising and lowering moves from $\mathfrak{m}(p_1)\mathfrak{m}(p_2)$.
\end{remark}

\begin{lemma}\label{Le:dominant is obtained by lowering and raising moves}
Let $p_1$, $p_2$ be the paths in Remark \ref{R:multiplicity} and $m=\mathfrak{m}(p_1)\mathfrak{m}(p_2)$. If a copy of $m$ in $S_{(i, k)(j, \overline{v})}$ can be obtained by raising and lowering moves from $\mathfrak{m}(p_1)\mathfrak{m}(p_2)$, then the monomials of
$\chi_{\varepsilon}(L(m))$ for this copy of $m$ are contained in $S_{(i, k)(j, \overline{v})}$.
\end{lemma}

\begin{proof}
Suppose that a copy of $m$ in $S_{(i, k)(j, \overline{v})}$ can be obtained by raising and lowering moves from $\mathfrak{m}(p_1)\mathfrak{m}(p_2)$.
There are two ways to obtain $\mathfrak{m}(p_1)\mathfrak{m}(p_2)$. One way to obtain $\mathfrak{m}(p_1)\mathfrak{m}(p_2)$ is by a translation of paths in $\mathscr{P}_{i,k}$ to paths in $\mathscr{P}_{i,k'_t}$ with respect to ${\bf PS}(j,v)$ for some $1\leq t \leq r(i,j)$, where $k'_t$, $r(i,j)$ are defined in (\ref{eq:def of rij}). The other way is by a translation of paths in $\mathscr{P}_{a_t,\alpha_t}$ to paths in $\mathscr{P}_{a_t,\overline{\alpha_t}}$ with respect to ${\bf PS}(b_t,\beta_t)$ for some $1\leq t \leq b'$, where $a_t$, $\alpha_t$, $b_t$, $\beta_t$, $b'$ are defined in (\ref{Eq:a_tb_t}), $\overline{\alpha_t}$ is defined in Equation (\ref{Eq:small index}). We will prove the result for the first case. The proof of the result for the second case is similar.

Without loss of generality, we may assume that $i\leq j$. Since this copy of $m$ can be obtained by translating paths in $\mathscr{P}_{i,k}$ to paths in $\mathscr{P}_{i,k'_t}$ with respect to ${\bf PS}(j,v)$, where $k'_t=k+2c_t \ell$ for some $c_t\in \ZZ_{>0}$, $1\leq t\leq r(i,j)$, we can apply Lemma \ref{Le: II translation of degree 2}. This gives us $m=\mathfrak{m}(p_1)\mathfrak{m}(p_2)=Y_{i-r,k'_t-r}Y_{j+r,\overline{v}+r}$, where $\mathfrak{m}(p_1)=Y_{i,k}$, $\mathfrak{m}(p_2)=Y_{i-r,k'_t-r}Y_{i,k'_t}^{-1}Y_{j+r,\overline{v}+r}$, $r=\frac{k'_t-\overline{v}-h_0(i,j)}2+1$. In addition, the lowest $l$-weight monomial of the module $L(m)$ is given by
\[
\mathfrak{m}(\widetilde{p}_1)\mathfrak{m}(\widetilde{p}_2)=Y_{n-j-r+1,\overline{v}+n+r+1}^{-1}Y_{n-i+r+1,n+k'_t-r+1}^{-1},
\]
where $\widetilde{p}_1\in \mathscr{P}_{i,k}$, $\mathfrak{m}(\widetilde{p}_1)=Y_{n-j-r+1,\overline{v}+n+r+1-2c_t \ell}^{-1}Y_{n+1-j,n+\overline{v}+1-2c_t \ell}Y_{n-i+r+1,n+k'_t-r+1-2c_t \ell}^{-1}$ and $\widetilde{p}_2\in \mathscr{P}_{j,\overline{v}}$, $\mathfrak{m}(\widetilde{p}_2)=Y_{n+1-j,n+1+\overline{v}}^{-1}$.

Since the copy of $m$ in $S_{(i, k)(j, \overline{v})}$ can be obtained by raising and lowering moves from $\mathfrak{m}(p_1)\mathfrak{m}(p_2)$, we have that $m=\mathfrak{m}(p'_1)\mathfrak{m}(p'_2)$ for some $p_1'$, $p_2'$, where $p'_1$ is obtained from $p_1$ by some lowering moves, $p'_2$ is obtained from $p_2$ by some raising moves. Let $\widetilde{p}'_1$ and $\widetilde{p}'_2$ be the paths obtained from $\widetilde{p}_1$ and $\widetilde{p}_2$ by corresponding lowering and raising moves, respectively. Then $\mathfrak{m}(\widetilde{p}'_1)\mathfrak{m}(\widetilde{p}'_2)=\mathfrak{m}(\widetilde{p}_1)\mathfrak{m}(\widetilde{p}_2)$ is the lowest $l$-weight monomial of $L(m)$ for this copy of $m$. Since the vertical distance of points $(i,k)$ and $(i,k'_t)$ is equal to the vertical distance of points $(n+1-j,n+1+\overline{v}-2c_t \ell)$ and $(n+1-j,n+1+\overline{v})$, we have that if the paths $p'_1$ and $p'_2$ are non-overlapping, then the paths $\widetilde{p}'_1$ and $\widetilde{p}'_2$ are non-overlapping. 
Therefore, the highest $l$-weight and the lowest $l$-weight monomials of $\chi_{\varepsilon}(L(m))$ for this copy of $m$ are contained in $S_{(i,k)(j,v)}$. Similarly, we can obtain that the other monomials of $\chi_{\varepsilon}(L(m))$ for this copy of $m$ are contained in $S_{(i,k)(j,v)}$. Hence, the monomials of $\chi_{\varepsilon}(L(m))$ for this copy of $m$ are contained in $S_{(i, k)(j, \overline{v})}$.
\end{proof}

\subsection{Proof of Theorem \ref{Th:path description of degree 2} (2)}
Following Lemma \ref{Le:dominant is obtained by path translations or pairs of lowering and raising moves}, we know that all dominant monomials $m\neq Y_{i,k}Y_{j,\overline{v}}$ in $S_{(i, k)(j, \overline{v})}$ can be obtained by path translations or raising and lowering moves of width $\ell$ from dominant monomials which are obtained by translations of paths. The sum of the $\varepsilon$- characters of modules $L(m)$ is exactly $\chi_{(i,k),(j,\overline{v})}$, as defined in Definition \ref{Def: chi_(i,k,j,v)}. Following Lemma \ref{Le: II translation of degree 2}, Lemma \ref{Le: two translations}, and Lemma \ref{Le:dominant is obtained by lowering and raising moves}, we obtain that the monomials of $\chi_{\varepsilon}(L(m))$ are contained in $S_{(i, k)(j, \overline{v})}$. That is, $\chi_{(i,k),(j,\overline{v})}$ is contained in $S_{(i, k)(j,\overline{v})}$. Since the module $L(Y_{i,k}Y_{j,\overline{v}})$ is irreducible, we obtain that
\[
\chi_{\varepsilon}(L(Y_{i,k}Y_{j,\overline{v}}))=S_{(i,k)(j,\overline{v})}-\chi_{(i,k),(j,\overline{v})}.
\]
Following Theorem \ref{Th:simple modules are snake modules} and Definition \ref{def:small values of indices}, we have
\begin{align*}
\chi_{\varepsilon}(L(Y_{i,k}Y_{j,v}))=\chi_{\varepsilon}(L(Y_{i,k}Y_{j,\overline{v}}))
=S_{(i,k)(j,\overline{v})}-\chi_{(i,k) (j,\overline{v})}.
\end{align*}

\subsection*{Acknowledgments}
The authors would like to thank the reviewer for the valuable comments and suggestions, which have helped improve the quality of the paper. The work was partially supported by the National Natural Science Foundation of China (Nos. 12171213, 12271224) and the Fundamental Research Funds for the Central University of China (No. lzujbky-2023-ey06). JR Li is supported by the Austrian Science Fund (FWF): P-34602, Grant DOI: 10.55776/P34602, and PAT 9039323, Grant-DOI 10.55776/PAT9039323.}


\begin{thebibliography}{999}
\bibitem{A07} Y. Abe, \textit{Tensor representations for the quantum algebras at roots of unity}, Int. Math. Res. Not. IMRN (22) (2007) Art. ID rnm093, 43.

\bibitem{BC19} M. Brito, V. Chari, \textit{Tensor products and q-characters of HL-modules and monoidal categorifications}, J. \'Ec. polytech. Math. 6 (2019) 581--619.

\bibitem{BM17} M. Brito, E. Mukhin, \textit{Representations of quantum affine algebras of type $B_{N}$}, Trans. Amer. Math. Soc. 369(4) (2017) 2775--2806.

\bibitem{BK96} J. Beck, V. Kac, \textit{Finite-dimensional representations of quantum affine algebras at roots of unity}, J. Amer. Math. Soc. 9(2) (1996) 391--423.

\bibitem{C02} V. Chari, \textit{Braid group actions and tensor products}, Int. Math. Res. Not. (7) (2002) 357--382.

\bibitem{CH10} V. Chari, D. Hernandez, \textit{Beyond Kirillov-Reshetikhin modules}, In: Quantum affine algebras, extended affine Lie algebras, and their applications, Contemp. Math. 506 (2010) 49--81.

\bibitem{CP94} V. Chari, A. Pressley, \textit{A guide to quantum groups}, Cambridge University Press, Cambridge (1994) xvi+651.

\bibitem{CP95} V. Chari, A. Pressley, \textit{Quantum affine algebras and their representations}, Representations of groups (Banff, {AB}, 1994), CMS Conf. Proc. Amer. Math. Soc., Providence, RI 16 (1995) 59--78.

\bibitem{CP97b} V. Chari, A. Pressley, \textit{Quantum affine algebras at roots of unity}, Represent. Theory 1 (1997) 280--328.

\bibitem {Dri85} V.-G. Drinfeld, \textit{Hopf algebras and the quantum Yang-Baxter equation}, Dokl. Akad. Nauk SSSR 283(5) (1985) 1060--1064.

\bibitem{Dri88} V.-G. Drinfeld, \textit{A new realization of Yangians and of quantum affine algebras}, Dokl. Akad. Nauk SSSR 296(1) (1987) 13--17; translation in Soviet Math. Dokl. 36(2) (1988), 212--216.

\bibitem{FM02} E. Frenkel, E. Mukhin, \textit{The $q$-characters at roots of unity}, Adv. Math. 171(1) (2002) 139--167.

\bibitem{FR98} E. Frenkel, N. Reshetikhin, \textit{The $q$-characters of representations of quantum affine algebras and deformations of W-algebras}, Recent developments in quantum affine algebras and related topics ( Raleigh, {NC}, 1998), Contemp. Math. Amer. Math. Soc., Providence, RI, 248 (1999) 163--205.

\bibitem{G16} A.-S. Gleitz, \textit{Representations of $U_q(L\mathfrak{sl}_2)$ at roots of unity and generalised cluster algebras}, European J. Combin. 57(2016) 94--108.

\bibitem{GDL22} J.-M. Guo, B. Duan, Y.-F. Luo, \textit{Combinatorics of the $q$-characters of Hernandez-Leclerc modules}, J. Algebra 605 (2022) 253--295.

\bibitem{H10} D. Hernandez, \textit{Simple tensor products}, Invent. Math. 181(3) (2010) 649--675.

\bibitem{HL10} D. Hernandez, B. Leclerc, \textit{Cluster algebras and quantum affine algebras}, Duke Math. J. 154(2) (2010) 265--341.

\bibitem{HL13} D. Hernandez, B. Leclerc, \textit{Monoidal categorifications of cluster algebras of type $A$ and $D$}, Symmetries, integrable systems and representations, Springer Proc. Math. Stat. 40 (2013) 175--193.

\bibitem{J24} I.-S. Jang, \textit{Path description for $q$-characters of fundamental modules in type $C$}, Comm. Algebra 52(3) (2024) 1237--1254.

\bibitem{Jim85} M. Jimbo, \textit{A $q$-difference analogue of $U(\mathfrak{g})$ and the Yang-Baxter equation}, Lett. Math. Phys. 10(1) (1985) 63--69.

\bibitem{Le11} B. Leclerc, \textit{Quantum loop algebras, quiver varieties, and cluster algebras}, Representations of algebras and related topics, EMS Ser. Congr. Rep. Eur. Math. Soc., Z\"urich (2011) 117--152.

\bibitem{L93} G. Lusztig, \textit{Introduction to quantum groups}, Progress in Mathematics 110 (1993) xii+341.

\bibitem{MY12a} E. Mukhin, C.A.S. Young, \textit{Path description of type B $q$-characters}, Adv. Math. 231(2) (2012) 1119--1150.

\bibitem{MY12b} E. Mukhin, C.A.S. Young, \textit{Extended T-systems}, Selecta Math. 18(3) (2012) 591--631.

\bibitem{T23} J. Tong, B. Duan, Y.-F. Luo, \textit{A combinatorial model for $q$-characters of fundamental modules of type $D_n$}, Comm. Algebra 53(2) (2025) 467--488.
\end{thebibliography}
\end{document}